\documentclass[12pt]{article}

\usepackage{graphicx}
\usepackage{latexsym,amssymb}
\usepackage{amsthm}
\usepackage{indentfirst}
\usepackage{amsmath}
\usepackage{color}
\usepackage{xcolor}
\usepackage{fourier}
\usepackage[colorlinks=true,backref=page]{hyperref}

\newtheorem{theoremalph}{Theorem}
\newtheorem{Coroll}[theoremalph]{Corollary}

\newtheorem{Theorem}{Theorem}[section]
\newtheorem*{Theorem A}{Theorem A}
\newtheorem*{Theorem A'}{Theorem A'}
\newtheorem*{Theorem C'}{Theorem C'}

\newtheorem*{Conj*}{Conjecture}

\newtheorem{Definition}[Theorem]{Definition}
\newtheorem{Proposition}[Theorem]{Proposition}
\newtheorem{Lemma}[Theorem]{Lemma}

\newtheorem*{Remark}{Remark}

\newtheorem{Remark-numbered}[Theorem]{Remark}
\newtheorem{Remarks-numbered}[Theorem]{Remarks}
\newtheorem{Corollary}[Theorem]{Corollary}

\newtheorem*{Claim}{Claim}
\newtheorem{Claim-numbered}{Claim}

 \def\NN{{\mathbb N}} 

 \def\RR{{\mathbb R}} 
\def\TT{{\mathbb T}}

   \def\cN{{\cal N}} \def\cT{{\cal T}}
\def\cC{{\cal C}}   \def\cO{{\cal O}} 
   \def\cP{{\cal P}} \def\cV{{\cal V}}
\def\cE{{\cal E}}    \def\cW{{\cal W}}
\def\cF{{\cal F}}

\newcommand{\id}{\operatorname{Id}}

\newcommand{\sing}{{\operatorname{Sing}}}

\newcommand{\orb}{\operatorname{Orb}}

\def\dim{\operatorname{dim}}

\def\diam{\operatorname{Diam}}

\def\orb{\operatorname{Orb}}
\def\supp{\operatorname{Supp}}

\def\Interior{\operatorname{Interior}}
\def\lip{\operatorname{Lip}}

\begin{document}

\title{{Topologically stable manifolds\\ for index-$1$ singular dominated splittings}}

\author{Sylvain Crovisier\footnote{S.Crovisier was partially supported by \emph{ISDEEC} ANR-16-CE40-0013, and by the ERC project 692925 \emph{NUHGD}.} \; and Dawei Yang\footnote{D. Yang  was partially supported by National Key R\&D Program of China (2022YFA1005801), by NSFC 12171348, NSFC 12325106, {ZXL2024386 and Jiangsu Specially Appointed Professorship.}
}}

\maketitle

\begin{abstract} 
For $C^2$ vector fields, we study regular ergodic measures whose supports admit singular dominated splittings with one of the bundles having dimension $1$. For such a measure $\mu$, we prove that if any periodic orbit 
within the support of $\mu$ (when it exists) has at least one negative Lyapunov exponent, and if the dynamics on the support of $\mu$ is not topologically equivalent to an irrational flow on a $2$-torus, then $\mu$-almost every point $x$ admits a $2$-dimensional topologically stable manifold $V^s(x)$: we mean that $V^s(x)$ is an embedded disc such that the orbit any point within it converges to the orbit of $x$ up to a time-reparametrization. Note that we do not assume any hyperbolicity for $\mu$.

We also establish an analogous conclusion
for compact invariant sets $\Lambda$ with a singular dominated splitting,
assuming some mild contraction property (any regular ergodic measure properly supported in $\Lambda$
must have at least one negative Lyapunov exponent).
This result will be used in our future work on the Palis density conjecture for three-dimensional vector fields.
\end{abstract}

\section{Introduction}

Let us assume that $X$ is a vector field over a 
compact Riemannian manifold without boundary $M$ and $(\varphi_t)_{t\in \RR}$ is the flow generated by $X$. 

Given a point $x\in M$, one defines its topologically stable set ${\frak W}^s(x)$: a point $y$ belongs to ${\frak W}^s(x)$ if and only if there is an increasing homeomorphism $\theta:\RR^+\to\RR^+$ which satisfies $d(\varphi_{\theta(t)}(y),\varphi_t(x))\to 0$ as $t\to+\infty$.
 Due to the shear of the flow,
 ${\frak W}^s(x)$ is in general not equal to the stronger stable set $W^s(x)$, defined by
 $$W^s(x):=\big\{y\in M:~\lim_{t\to\infty}d(\varphi_{t}(y),\varphi_t(x))=0\big\}.$$
 
In this paper, we will establish the existence of such topologically stable manifolds for typical points of general regular ergodic measures when the vector field is $C^2$ and satisfies some weak generic property. We will also require the measure to admit some dominated splitting, but we do not require the measure to be hyperbolic (as in Pesin theory~\cite{Pes}).

One way to study flows is to intersect the orbits along a transverse section as considered first by Poincar\'e for the local dynamics near a periodic orbit. For studying the global dynamics, difficulties occur in the presence of \emph{singularities},
that are the points $\sigma\in M$ such that $X(\sigma)=0$.
We denote by ${\rm Sing}(X)$ the set of singularities of $X$. 
Singularities only give simple dynamics: they are in fact the fixed points of the flow.
The other points are called \emph{regular} points.
Difficulties appear when regular orbits accumulates on singularities,
as it is the case for the famous Lorenz attractor \cite{Lo}.

The \emph{tangent flow} of $X$,  i.e. the derivative with respect to the space variable $D\varphi_t: TM\to TM$,
can give some information on the stable sets: for instance any point in a uniformly hyperbolic invariant compact set with no singularity admits a stable manifold, see~\cite{hirsch-pugh-shub}.
Liao in~\cite{Lia63} generalized the above idea of Poincar\'e on the regular part of the system and introduced the notion of \emph{linear Poincar\'e flow} $(\psi_t)_{t\in \RR}$. 
At any regular point $x$, one defines the normal space $\cN_x=\{v\in T_x M:~\left<v,X(x)\right>=0\}$. This gives the normal bundle $\cN_{M\setminus{\rm Sing}(X)}$ above the regular set of $X$, which is a priori non-compact.
The flow $\psi_t:~\cN_{M\setminus{\rm Sing}(X)}\to \cN_{M\setminus{\rm Sing}(X)}$ is defined for any vector $v\in\cN_x$ with $x\in M\setminus{\rm Sing}(X)$ by projecting orthogonally $D\varphi_t(v)$ on $\cN_{\varphi_t(x)}$:
$$\psi_t(v)=D\varphi_t(v)-\frac{\left<X(\varphi_t(x)),D\varphi_t(v)\right>}{\left|X(\varphi_t(x))\right|^2}X(\varphi_t(x)).$$
Hyperbolic structures on the tangent dynamics, like dominated splittings, have been extensively studied for diffeomorphisms beyond uniform hyperbolicity. They give very weak forms of hyperbolicity, but still robust dynamics. For vector fields, the existence of a dominated splitting for the tangent flow may be a very strong assumption: it sometimes imply uniform contraction or expansion as in \cite{BGY}.
It is natural to also consider dominated splittings on the normal bundle. As summarized by Bonatti and da Luz \cite{BdL}: ``for flows, hyperbolic structures live on the normal bundle for the linear Poincar\'e flow, but not on the tangent bundle''.

Let us consider some compact invariant set $\Lambda\subset M$.

A $\psi_t$-invariant splitting $\cN_{\Lambda\setminus \sing(X)}=\cN^{cs}\oplus \cN^{cu}$ is said to be \emph{dominated} if there is some time $\tau>0$ such that for any $x\in\Lambda\setminus \sing(X)$ and any unit vectors $v^{cs}\in\cN^{cs}$ and $v^{cu}\in\cN^{cu}$, one has
$$\|\psi_{\tau}.v^{cs}\|\le \tfrac 1 2 \|\psi_{\tau}.v^{cu}\|.$$
In contrast to the cases of diffeomorphisms and of the tangent flow $(D\varphi_t)$, dominated splittings of the linear Poincar\'e flow are in general not robust due to the existence of singularities. For this reason we require the dominated splitting to satisfy a compatibility condition with the singularities.
For instance \cite[Definition 1.1]{CYZ2} has introduced the notion of \emph{singular domination}, but similar notions appeared before in different works.
In the present paper we consider a definition slightly different from there and focus on the case where
$\dim(\cN^{cs})=1$.

\begin{Definition}
A compact invariant set $\Lambda$ is said to admit an \emph{index-1 singular dominated splitting}
if there is a dominated splitting $\cN_{\Lambda\setminus{\rm Sing}(X)}=\cN^{cs}\oplus\cN^{cu}$ with $\dim\cN^{cs}=1$ such that at any singularity $\sigma\in\Lambda$ (if there exists some):
\begin{itemize}
\item[--] $DX(\sigma)$ is invertible,
\item[--] there is a dominated splitting $T_\sigma M=E^{ss}\oplus E^{cu}$ where $E^{ss}$ is one-dimensional and uniformly contracted by the tangent flow,
\item[--] the associated strong stable manifold $W^{ss}(\sigma)$ satisfies $W^{ss}(\sigma)\cap \Lambda=\{\sigma\}$.
\end{itemize}
\end{Definition}

We now introduce a weak notion of stable manifold at a regular point:

\begin{Definition}
For any regular point $x$, we say that an embedded $C^1$-disc $V^s(x)$ containing $x$ is \emph{topologically stable} if
for any $y\in V^s(x)$:
\begin{itemize} 
\item[--] $X(y)\in T_yV^s(x)$,
\item[--] there exists an orientation-preserving homeomorphism $\theta\colon \RR\to\RR$ such that
$$d(\varphi_{\theta(t)}(y),\varphi_t(x))\underset{t\to 0}\longrightarrow 0.$$
\end{itemize}
\end{Definition}

Let us consider an ergodic measure $\mu$ for $(\varphi_t)_{t\in \RR}$ which is \emph{regular},
i.e. which satisfies $\mu({\rm Sing}(X))=0$. Note that one may have ${\rm supp}(\mu)\cap {\rm Sing}(X)\neq\emptyset$.
Pesin theory ensures that if one of the Lyapunov exponents of the measure is negative, then $\mu$-almost every point admits a stable manifold. In the next theorem we do not assume information on negative Lyapunov exponents.

\begin{theoremalph}\label{Thm:measure-stable}
Let $M$ be a compact Riemannian manifold without boundary, $X$ be a $C^2$ vector field over $M$ and $\mu$ be a regular 
ergodic measure of the flow generated by $X$. Let us assume that on the support of $\mu$:
\begin{enumerate}
\item there exists an index-$1$ singular dominated splitting $\cN_{{\rm supp}(\mu)\setminus{\rm Sing}(X)}=\cN^{cs}\oplus\cN^{cu}$,
\item\label{i2} any periodic orbit (if it exists) has at least one negative Lyapunov exponent,
\item\label{i3} the dynamics is not topologically equivalent to an irrational flow on a $2$-torus.
\end{enumerate}
Then $\mu$-almost every point $x$ admits a $2$-dimensional topologically stable disc $V^s(x)$
which is tangent to $\cN^{cs}(x)\oplus X(x)$.
\end{theoremalph}
A more precise version stated in Section~\ref{s.main} provides some uniformity on the discs $V^s(x)$.
Note also that for generic vector fields, if $\mu$ is not supported on a source, then it always satisfy
conditions~\eqref{i2} and~\eqref{i3}.

The statement is non-trivial even when $\supp(\mu)$ contains no singularity.
When it contains a singularity, the induced dynamics cannot be topologically equivalent to an irrational flow on
a $2$-torus and the following corollary holds.

\begin{Coroll}\label{Cor:generic-measure-stable}
Let $M$ be a compact Riemannian manifold without boundary, $X$ be a $C^2$ vector field over $M$ and $\mu$ be a regular ergodic measure of the flow generated by $X$. Let us assume that on the support of $\mu$:
\begin{itemize}
\item[--] there exists an index-$1$ singular dominated splitting $\cN_{{\rm supp}(\mu)\setminus{\rm Sing}(X)}=\cN^{cs}\oplus\cN^{cu}$,
\item[--] any periodic orbit (if it exists) has at least one negative Lyapunov exponent.
\end{itemize}
If the support of $\mu$ contains a singularity, then
$\mu$-almost every point $x$ admits a $2$-dimensional topologically stable disc $V^s(x)$
which is tangent to $\cN^{cs}(x)\oplus X(x)$.

\end{Coroll}

The following analogous result holds for compact invariant sets when any proper compact invariant subset satisfies the contraction property. 

\begin{theoremalph}\label{Thm:minially-non-contracted-stable}
Let $M$ be a Riemannian compact manifold without boundary, $X$ be a $C^2$ vector field over $M$ and $\Lambda$ be a compact set invariant by the flow generated by $X$ which is not reduced to a periodic orbit.
Let us assume that :
\begin{itemize}
\item[--] there exists an index-$1$ singular dominated splitting $\cN_{\Lambda\setminus{\rm Sing}(X)}=\cN^{cs}\oplus\cN^{cu}$,
\item[--] any regular ergodic measure $\mu$ (if there exists) whose support is a proper subset of $\Lambda$ has at least one negative Lyapunov exponent,
\item[--] the dynamics is not topologically equivalent to an irrational flow on a $2$-torus.
\end{itemize}
Then any point $x\in \Lambda\setminus \sing(X)$ admits a $2$-dimensional topologically stable disc $V^s(x)$
which is tangent to $\cN^{cs}(x)\oplus X(x)$.

\end{theoremalph}

Theorem~\ref{Thm:minially-non-contracted-stable} and its more precise version stated in Section~\ref{s.main}
will be used in a future work to prove a conjecture of Palis for three-dimensional vector fields~\cite{CY}.
This was one of the main motivations for this paper.

Our results and their proofs continue and generalize previous works:
for one-dimensional maps by Ma\~n\'e \cite{Man85}, surface diffeomorphisms by Pujals and Sambarino \cite{PS1},
higher-dimen\-sio\-nal diffeomorphisms with a contractive codimension-one dominated splitting \cite{PS2,PS3,CP},
and non-singular three-dimensional vector fields by Arroyo and Rodriguez-Hertz \cite{ARH}.
Note that Theorems~\ref{Thm:measure-stable} and~\ref{Thm:minially-non-contracted-stable} have counterparts for diffeomorphisms in any dimension; in this case a stronger version by Pujals, Sambarino and the first author has been announced in~\cite[Theorem 4.4]{Crovisier-ICM}.

As mentioned above, in our proofs we will have to deal with the singularities.
Therefore we will have to rescale the flow (as in~\cite{Lia89,GY}), to consider blow-ups, and to use an abstract model for the non-linear Poincar\'e flow introduced in our earlier work~\cite{CY2}. Note that our way to construct normally expanded irrational torus (in Section~\ref{ss.topological}) is different from previous approaches: we use more local arguments, whereas Pujals and Sambarino \cite{PS1} and Arroyo and Rodriguez-Hertz \cite{ARH} used more semi-global arguments.

\section{Fibered dynamics with a dominated splitting}\label{s.fibered}

{
We will work with the nonlinear Poincar\'e flow and its compactification which has been built in our previous paper \cite{CY2}.
It can be seen as a local fibered flow with an identification structure as we recall here.
In this section, we will also discuss the dynamics of local fibered flow preserving a dominated splitting.}

\subsection{Local fibered flow and identifications}
The definitions come from \cite[Definition 1.1, Definition 1.2]{CY2}.

\subsubsection{Definition of local fibered flows}

\begin{Definition}[Local fibered flow]\label{d.local-flow}
Let $(\varphi_t)_{t\in \RR}$ be a continuous flow over a compact metric space $K$, and let $\cN\to K$ be a continuous Riemannian vector bundle.
A \emph{local $C^k$-fibered flow} $P$ on $\cN$ is a continuous family of $C^k$-diffeomorphisms $P_{t}\colon \cN_x\to \cN_{\varphi_t(x)}$, for $(x,t)\in K\times \RR$,
preserving the $0$-section of $\cN$ with the following property.

There is $\beta_0>0$ such that for each $x\in K$, $t_1,t_2\in \RR$, and $u\in \cN_x$ satisfying
$$\|P_{s.t_1}(u)\|\leq \beta_0 \text{ and } \|P_{s.t_2}(P_{t_1}(u))\|\leq \beta_0 \text{ for each }s\in[0,1],$$
then we have
$$P_{t_1+t_2}(u)=P_{t_2}\circ P_{t_1}(u).$$
\end{Definition}

{ We consider local fibered flows as in Definition~\ref{d.local-flow}. The following notations will be used.}

\smallskip

\noindent
{\bf Notations.} -- One sometimes denotes a point $u\in \cN_x$  as $u_x$  to emphasize the base point $x$.
\hspace{-1cm}
\smallskip

\noindent
-- The length of a $C^1$ curve $\gamma\subset \cN_x$ (with respect to the metric of $\cN_x$)
is denoted by $|\gamma|$.
\smallskip

\noindent
-- A ball centered at $u$ and with radius $r$ inside a fiber $\cN_x$ is denoted by $B(u,r)$.
\smallskip

\noindent
-- For $x\in K$, $t\in \RR$ and $u\in \cN_x$, one denotes
by $DP_t(u)$ the derivative of $P_t$ at $u$ along $\cN_x$.
In particular $(DP_t(0_x))_{t\in \RR, x\in K}$ defines a linear flow over the $0$-section of $\cN$.

\subsubsection{Definition of identifications}

\begin{Definition}\label{Def:identification}
Let us consider a \emph{local $C^k$-fibered flow} $P$ on $\cN\to K$
over a  continuous flow $(\varphi_t)_{t\in \RR}$ over a compact metric space $K$.
A \emph{$C^k$-identification} $\pi$ on an open set $U\subset K$
is a continuous family of $C^k$-diffeomorphisms $\pi_{y,x}\colon \cN_y\to \cN_x$ indexed by pairs of close points $x,y \in U$ satisfying for some
$\beta_0,r_0>0$ the following:
for any $x,y,z\in U$ and $u\in \cN_y$ with $d(z,x), d(z,y)<r_0$
and $\|u\|<\beta_0$,
$$\pi_{z,x}\circ \pi_{y,z}(u)=\pi_{y,x}(u).$$
\end{Definition}
In particular $\pi_{x,x}$ coincides with the identity on $B(0,\beta_0)$.
\medskip

\noindent
{\bf Notations.} -- We will sometimes denote $\pi_{y,x}$
by $\pi_x$. Also the projection $\pi_{y,x}(0)=\pi_x(0_y)$ of $0\in \cN_y$ on $\cN_x$ will be denoted
by $\pi_x(y)$.
\medskip

\noindent
-- We will denote by ${\rm Lip}$ be the set of orientation-preserving bi-Lipschitz homeomorphisms $\theta$ of $\RR$
(and by ${\rm Lip}_{1+\rho}$ the set of maps in ${\rm Lip}$ whose Lipschitz constant is smaller than $1+\rho$).

\begin{Definition}\label{d.compatible}
The identification $\pi$ on $U$ is \emph{compatible} with the flow $(P_t)$ if:
\begin{enumerate}

\item \emph{No small period.}
For any $\varkappa>0$, there is $r>0$ such that for any $x\in \overline{U}$
and $t\in [-2,2]$ with $d(x,\varphi_t(x))<r$, then we have $|t|< \varkappa$.

\item \emph{Local injectivity.} For any $\delta>0$, there exists $\beta>0$ such that for any $x,y\in U$:\\
if $d(x,y)<r_0$ and $\|\pi_{x}(y)\|\leq \beta$, then $d(\varphi_t(y),x)\leq \delta$ for some $t\in [-1/4,1/4]$.

\item \emph{Local invariance.} For any $x,y\in U$
and $t\in [-2,2]$ such that $y$ and $\varphi_t(y)$ are $r_0$-close to $x$,
and for any $u\in B(0,\beta_0)\subset \cN_{y}$, we have
$$\pi_{x}\circ P_t(u)=\pi_{x}(u).$$

\item \emph{Global invariance.}
For any $\delta,\rho>0$, there exists $r,\beta>0$ such that:

For any $y,y'\in K$ with $y\in U$ and $d(y,y')<r$,
for any $u\in  \cN_y$, $u'\in \cN_{y'}$ with $\pi_y(u')=u$,
and any intervals $I,I'\subset \RR$ containing $0$ and satisfying
$$\|P_s(u)\|<\beta\text{ and } \|P_{s'}(u')\|<\beta\text{ for any } s\in I \text{ and any }
s'\in I',$$
there is $\theta\in {\rm Lip}_{1+\rho}$ such that $\theta(0)=0$ and
 $d(\varphi_s(y),\varphi_{\theta(s)}(y'))<\delta$ for any $s\in I\cap \theta^{-1}(I')$.
Moreover, if some $v\in \cN_y$ satisfies [$\forall s\in I\cap \theta^{-1}(I')$,
$\|P_s(v)\|<\beta$], then:
\begin{itemize}
\item[--] $v'=\pi_{y'}(v)$ satisfies
$\|P_{\theta(s)}(v')\|<\delta$ for each $s\in I\cap \theta^{-1}(I')$;
\item[--] if $\varphi_s(y)\in U$
for some $s\in I\cap \theta^{-1}(I')$, then $\pi_{\varphi_s(y)}\circ P_{\theta(s)}(v')=P_s(v)$.
\end{itemize}
\end{enumerate}

\end{Definition}

\begin{Remarks-numbered}\label{r.identification}\rm
a) These definitions are still satisfied if one reduces
$r_0$ or $\beta_0$. Their value will be reduced in the following sections in order to
satisfy additional properties.
\medskip

\noindent
b) One can also rescale the time and keep a compatible identification: the flow $t\mapsto \varphi_{t/C}$ for $C>1$
still satisfies the definitions above, maybe after reducing the constant $r_0$.
\medskip

\noindent
c) Applying a rescaling as discussed in the item (b),
one can replace the interval $[-2,2]$ in the ``No small period'' Property and in the Local invariance
by any larger interval $[-L,L]$; one can replace the interval $[-1/4,1/4]$ in the local injectivity
by any small interval $[-\eta,\eta]$.
\medskip

\noindent
d) The ``No small period" assumption (which does not involve the projections $\pi_x$) is equivalent to the non-existence of periodic orbits of period $\leq 2$
which intersect $\overline U$.
In particular, by reducing $r_0$, one can assume the following property:
\smallskip

\emph{For any $x\in U$ and any $t\in [1,2]$, we have $d(x,\varphi_t(x))\ge r_0$.}
\medskip

\noindent
e) For $x\in U$, the Local injectivity prevents the existence of  $y\in U$ that is $r_0$-close to $x$,
is different from $\varphi_t(x)$ for any $t\in [-1/4,1/4]$, and
such that $\pi_x(y)=0_x$.
In particular:
\smallskip

\emph{If $x,\varphi_t(x)\in U$ and $t\not\in (-1/2,1/2)$ satisfy
$\pi_x(\varphi_t(x))=0_x$, then $x$ is periodic.}\hspace{-1cm}\mbox{}
\medskip

\noindent
f) The Global invariance says that when two orbits $(P_s(u))$ and
$(P_s(u'))$ of the local fibered flow are close to
the zero-section of $\cN$ and have two points which are identified by $\pi$,
then they are associated to orbits of the flow $\varphi$ that are close
(up to a reparametrization $\theta$).
In this case, any orbit of $(P_t)$ close to the zero-section
above the first $\varphi$-orbit can be projected to an orbit of $(P_t)$ above the second $\varphi$-orbit.
\medskip

\noindent
g) The Global invariance can be applied to pairs of points $y,y'$ where the condition
$d(y,y')<r$ has been replaced by a weaker one $d(y,y')<r_0$. In particular, this gives:
\medskip

\emph{For any $\delta,\rho>0$, there exist $\beta>0$ such that:
if  $y,y'\in K$, $u\in  \cN_y$, $u'\in \cN_{y'}$ and the intervals $I,I'\subset \RR$ containing $0$ satisfy
\begin{itemize}
\item[--] $d(y,y')<r_0$ and $y\in U$,
\item[--] $\pi_y(u')=u$,
$\|P_s(u)\|<\beta\text{ and } \|P_{s'}(u')\|<\beta\text{ for any } s\in I \text{ and any }
s'\in I',$
\end{itemize}
there is $\theta\in {\rm Lip}_{1+\rho}$ such that $|\theta(0)|\leq 1/4$ and $d(\varphi_s(y),\varphi_{\theta(s)}(y'))<\delta$ for any $s\in I\cap \theta^{-1}(I')$.}
\end{Remarks-numbered}

\subsubsection{Properties of identifications}
The next property (\cite[Proposition 7.3]{CY2}) states that one cannot find two reparametrizations
of a same orbit, that shadow each other, coincide for some parameter and differ by at least $2$ for another parameter.

\begin{Proposition}[No shear inside orbits]\label{p.no-shear}
Let us assume that there exists a $C^k$ identification on an open set $U\subset K$
which is compatible with the local fibered flow $(P_t)_{t\in\RR}$
and that the set $\Delta:=\{x| \forall |t|\leq \tfrac 1 2,\; \varphi_t(x)\notin U\}$
is disjoint from $\overline U$. If $r_0>0$ is small, then
for any $x\in U$, any increasing homeomorphism $\theta$ of $\RR$, any closed interval $I$ containing $0$ satisfying $[d(\varphi_t(x),\varphi_{\theta(t)}(x))\leq r_0,\;\forall t\in I]$ and $\varphi_{\theta(0)}(x)\in U$,
the following holds:
\begin{itemize}
\item[--] If $\theta(0)> 1/2$, then $\theta(t)> t+2$, $\forall t\in I$ when $\varphi_t(x), \varphi_{\theta(t)}(x)\in U$;
\item[--] If $\theta(0)\in [-2,2]$, then $\theta(t) \in [t-1/2,t+1/2]$, $\forall t\in I$ when $\varphi_t(x),\varphi_{\theta(t)}(x)\in U$;
\item[--] If $\theta(0)< -1/2$, then $\theta(t)<t-2$, $\forall t\in I$ when $\varphi_t(x),\varphi_{\theta(t)}(x)\in U$.
\end{itemize}
\end{Proposition}

The following closing lemma (\cite[Proposition 7.5]{CY2}) is an example of properties given by identifications.
\begin{Lemma}[Closing lemma] \label{l.closing0}
Let us assume that there exists a $C^k$ identification on an open set $U\subset K$
which is compatible with the local fibered flow $(P_t)_{t\in\RR}$
and let $\beta_0,r_0>0$ be small enough.
Let us consider:
\begin{itemize}
\item[--] $x\in U$ and $T\geq 4$ with $\varphi_T(x) \in U\cap B(x,r_0)$,
\item[--] a fixed point $p\in \cN_x$ for $\overline P_T:=\pi_x\circ P_T$
such that $\|P_t(p)\|<\beta_0$ for each $t\in [0,T]$,
\item[--]  a sequence $(y_k)_{k\in\NN}$ in a compact set of $U\cap B(x,r_0/2)$
such that $\pi_x(y_k)$ converges to $p$.
\end{itemize}
Then, taking a subsequence, $(y_k)_{k\in\NN}$
converges to a periodic point $y\in K$
such that $\pi_x(y)=p$.

Moreover, if $T'$ denotes the period of $y$, then we have
$$DP_{T'}(0_y)= D\pi_{x}(0_y)^{-1}\circ D\overline P_T(p) \circ D\pi_x(0_y).$$
\end{Lemma}

For the next statement (\cite[Corollary 7.6]{CY2}), we consider an open set $V$ such that $K\subset U\cup V$.
We reduce $r_0>0$ so that $d(K\setminus U, K\setminus V) > r_0$.
\begin{Corollary}\label{c.closing0}
Let us assume that $\beta_0,r_0>0$ are small enough.
If $x\in K\setminus V$ has an iterate $y=\varphi_T(x)$ in $B(x,r_0)$ with $T\geq 4$ and if there exists
a subset $B\subset \cN_x$ containing $0_x$ such that:
\begin{itemize}
\item[--] $P_t(B)\subset B(0_{\varphi_t(x)},\beta_0)$ for any $0<t<T$,
\item[--] $\overline P_T:=\pi_x\circ P_T$ sends $B$ into itself,
\item[--] the sequence $\overline P_T^k(0_x)$ converges to a fixed point $p\in B$ of $\overline P$,
\end{itemize}
then the positive orbit of $x$ by $\varphi$ converges to a periodic orbit of the flow $(\varphi_t)_{t\in\RR}$.
\end{Corollary}

\subsection{Dominated splitting of a local fibered flow}

\subsubsection{Definition of a dominated splitting}
\begin{Definition}\label{d.dominated}
The local flow $(P_t)$ admits a \emph{dominated splitting}
$\cN=\cE\oplus \cF$ if $\cE$, $\cF$ are sub-bundles of $\cN$ that are invariant
by the linear flow $(DP_t(0))$ and if there exists $\tau_0> 0$ such that
for any $x\in K$, for any unit vector $u\in \cE(x)$ and unit vector $v\in \cF(x)$ and for any $t\geq \tau_0$ we have:
$$\|DP_t(0_x).u\|\leq \frac 1 2 \|DP_t(0_x).v\|.$$
Moreover we say that $\cE$ is \emph{$2$-dominated} if there exists $\tau_0>0$ such that for any $x\in K$,
any unit vector $u\in \cE(x)$ and any unit vector $v\in \cF(x)$, and for any $t\geq \tau_0$ we have:
$$\max(\|DP_t(0_x).u\|,\|DP_t(0_x).u\|^2)\leq \frac 1 2 \|DP_t(0_x).v\|.$$
\end{Definition}

When there exists a dominated splitting $\cN=\cE\oplus \cF$ and $V\subset K$ is an open subset,
one can prove that $\cE$ is uniformly contracted by considering the induced dynamics on $K\setminus V$
and checking that the following property is satisfied.
\begin{Definition}
The bundle $\cE$ is \emph{uniformly contracted on the open set $V$} if
there exists $t_0>0$ such that for any $x\in K$ satisfying $\varphi_t(x)\in V$
for any $0\leq t\leq t_0$ we have:
$$\|DP_{t_0}{|\cE(x)}\|\leq \frac 1 2.$$
We say that $\cE$ is \emph{uniformly contracted} if it is uniformly contracted on $K$.
\end{Definition}

\subsubsection{Normally expanded irrational tori}\label{ss:tori}

We give a setting of
a dominated splitting $\cE\oplus \cF$ such that $\cE$ is not uniformly contracted.
\begin{Definition}\label{d.torus}
A \emph{normally expanded irrational torus} is
an invariant compact subset $\cT\subset K$ such that
\begin{itemize}
\item[--] the dynamics of $\varphi|_{\cT}$
is topologically equivalent to an irrational flow on $\TT^2$,
\item[--] there exists a dominated splitting $\cN|_{\cT}=\cE\oplus \cF$ and $\cE$ has one-dimensional fibers,
\item[--] for some $x\in U\cap \cT$ and $r>0$, $\pi_x(\{z\in K, d(x,z)<r\})$ is a $C^1$-curve
tangent to $\cE(x)$.
\end{itemize}
\end{Definition}

The name is justified as follows.
\begin{Lemma}\label{l.torus}
For any normally expanded irrational torus $\cT$,
the Lyapunov exponent along $\cE$ of the (unique) invariant measure of $\varphi$ on $\cT$
is equal to zero; in particular $\cF$ is uniformly expanded (i.e uniformly
contracted by backward iterations).
\end{Lemma}
\begin{Remark-numbered}\label{r.torus}
With the technics of Section~\ref{s.topological-hyperbolicity}, one
can also prove that the $\alpha$-limit set of any point $z$ in a neighborhood
coincides with $\cT$.
\end{Remark-numbered}
\begin{proof}
Let us choose a global transversal $\Sigma\simeq \TT^1$ containing $x$
for the restriction of $\varphi$ to $\cT$. The dynamics is conjugated to a suspension
of an irrational rotation of $\Sigma$.
We consider the sequence $(t_k)$ of positive returns of the orbit of $x$ inside a neighborhood
of $x$ in $\Sigma$. Note that $|t_{k+1}-t_k|$ is uniformly bounded.
For every $y\in \Sigma$ close to $x$
there exists a sequence $(t'_k)$ such that $|t_{k+1}-t_k|$ and $|t'_{k+1}-t'_k|$ are close and
$\varphi_{t'_k}(y)$ is close to $\varphi_{t_k}(x)$ and belongs to $\Sigma$.
In particular by choosing $y$ close enough to $x$,
there exist $\varepsilon_1,\varepsilon_2>0$ arbitrarily small such that
$$\forall k\geq 0, \; \varepsilon_1\leq d(\varphi_{t'_k}(y), \varphi_{t_k}(x))\leq \varepsilon_2.$$
Let $I\subset \Sigma$ be the interval bounded by $x,y$.
By Definition~\ref{d.torus} and the Local Injectivity, the transversal $\Sigma\cap B(x,r)$ is mapped by $\pi_x$ homeomorphically inside an interval of the $C^1$-curve
$\gamma=\pi_x(\{z\in K, d(x,z)<r\})$ and by the Global invariance
$\pi_x\circ P_{t_k}$ sends $\pi_x(y)$ to
$\pi_x(\varphi_{t'_k}(y))$ and similarly sends $\pi_x(I)\subset \gamma$ inside $\gamma$.
Moreover, there exist $\varepsilon'_1,\varepsilon'_2>0$ (that can be taken arbitrarily small if $y$ is close enough to $x$)
such that for each $k$,
$$\varepsilon'_1\leq |\pi_x\circ P_{t_k}\circ\pi_x(I)|\leq \varepsilon'_2.$$

Since $t_{k+1}-t_k$ is bounded, it implies that the Lyapunov exponent along $\cE$ vanishes.
\end{proof}

\subsubsection{Contraction on periodic orbits and criterion for $2$-domination}
When there exists a dominated splitting $\cE\oplus \cF$ where $\cE$ is one-dimensional,
the uniform contraction of the bundle $\cE$ above each periodic orbit of $K$,
implies that it is $2$-dominated.
\begin{Proposition}\label{p.2domination}
Let us assume that
\begin{itemize}
\item[--] there exists a dominated splitting $\cN=\cE\oplus \cF$ and the fibers of $\cE$ are one-dimensional,
\item[--] $\cE$ is uniformly contracted on an open set $V$ containing $K\setminus U$.
\end{itemize}
Then either the bundle $\cE$ is $2$-dominated, or there exists a periodic orbit
$\cO$ in $K$ whose Lyapunov exponents are all positive.
\end{Proposition}
\begin{proof}
If there is no $2$-domination, there exists a sequence $(x_n)$ in $K$,
such that
$$\|DP_n{|_{\cE(x_n)}}\|^2\geq \frac 1 2 \inf_{v\in \cF(x_n), \|v\|=1}\|DP_n(v)\|.$$
One can extract a $\varphi$-invariant measure from the sequence
$$\mu_n:=\frac 1 n \int_{t=0}^{n}\delta_{\varphi_{t}(x_n)}\; dt$$
and the maximal Lyapunov exponent $\lambda^\cE$ along $\cE$, the minimal Lyapunov exponent $\lambda^\cF$
along $\cF$ satisfy
$2\lambda^\cE\geq \lambda^\cF$.
In particular, $\lambda^\cF> \lambda^\cE\geq \lambda^\cF-\lambda^\cE>0$.
Since $\cE$ is one-dimensional,
one deduces that $\varphi$ admits an ergodic measure $\mu$ whose Lyapunov exponents are both positive.
Since $\cE$ is uniformly contracted on $V$, the support of $\mu$ has to intersect $U$.
For $\mu$-almost every point $x\in K$, there exists a neighborhood
$V_x$ of $0$ in $\cN_x$ such that $\|(DP_{-t})|_{V_x}\|$ decreases exponentially
as $t\to +\infty$. In particular, one can take $x\in U$ recurrent
and find a large time $T>0$ such that
$\widetilde P_{-T}:=\pi_x\circ P_{-T}$ sends
$V_x$ into itself as a contraction.
{By} Corollary~\ref{c.closing0}, 
there is a periodic point $y$ in $K$ with some period $T'>0$
and a fixed point $p\in V_x$ for $\widetilde P_{-T}$
such that the tangent map $DP_{-T'}(0_y)$ is conjugate to the tangent map
$D\widetilde P_{-T}(p)$ (by Lemma~\ref{l.closing0}).
Hence the Lyapunov exponents of $y$ are all positive.
\end{proof}

\subsection{Plaque families}\label{s.plaque}

We now introduce center-stable plaques  $\cW^{cs}(x)$ that are candidates to be the local stable manifolds
of the dynamics tangent to $\cE$.
A symmetric discussion gives center-unstable plaques $\cW^{cu}$ tangent to $\cF$.

\subsubsection{Standing assumptions}\label{ss.assumptions}
In this Section~\ref{s.plaque}, we consider:
\begin{itemize}
\item[--] a bundle $\cN$ with $d$-dimensional fibers, a local fibered flow $(\cN,P)$ over a topological flow $(\varphi,K)$ and an identification $\pi$
on an open set $U$, compatible with $(P_t)$,
\item[--] a dominated splitting $\cN=\cE\oplus \cF$ of the bundle $\cN$,
\item[--] an open set $V$ containing $K\setminus U$.
\end{itemize}
Reducing $r_0$ we assume that the distance between $K\setminus V$ and $K\setminus U$ is much larger than $r_0$.

We also fix an integer $\tau_0\geq 1$ satisfying Definition~\ref{d.dominated} of the domination.
We choose $\lambda>1$ such that $\lambda^{4\tau_0}<2$.
 In particular
for any $x\in K$, for any unit vector $u\in \cE(x)$ and any unit vector $v\in \cF(x)$, we have:
\begin{equation}\label{e.domination}
\forall t\geq\tau_0,~~~\|DP_t(0).u\|\leq \lambda^{-2t} \|DP_t(0).v\|.
\end{equation}

In each space $\cN_x=\cE(x)\oplus \cF(x)$, $x\in K$, we introduce the constant cone
$$\cC^\cE(x):=\{u=u^\cE+u^\cF\in \cN_x=\cE(x)\oplus \cF(x), \|u^\cE\|> \|u^\cF\|\} \cup\{0_x\},$$
and $\cC^\cF(x)$ in a symmetric way. They vary continuously with $x$.
Moreover the dominated splitting implies that for any $t\geq \tau_0$ the cone fields are contracted:
$$DP_t(0_x).\overline{\cC^\cF(x)}\subset \cC^\cF(\varphi_t(x))
\text{ and } DP_{-t}(0_x).\overline{\cC^\cE(x)}\subset \cC^\cE(\varphi_{-t}(x)).$$
We say that a vector is tangent to the cone $\cC^\cE(x)$, when the vector belongs to $\cC^\cE$.

The cone can be extended trivially at any $u\in\cN_x$: the affine structure of $\cN_x$ allows to identify the tangent spaces
of $\cN_x$ at $x$ and $u$ and define $\cC^{\cE}(u)=\cC^{\cE}(x)$, $\cC^{\cF}(u)=\cC^{\cF}(x).$

Given $\eta>0$, let us introduce the set of points whose forward (respectively backward) orbits remain $\eta$-close to the $0$-section:
$$K_\eta^+=\{u\in\cN,~\|P_t(u)\|\le\eta,~\forall t\ge0\},$$
$$K_\eta^-=\{u\in\cN,~\|P_{-t}(u)\|\le\eta,~\forall t\ge0\}.$$
The invariance of the cone fields and the existence of the bundles extend over these sets.
The proof of the following lemma is standard, hence omitted.
\begin{Lemma}\label{Lem:extend-bundle}
There is $\eta_0>0$ with the following properties.
\begin{itemize}
\item The cone fields are invariant in the following sense: for any $t\ge\tau_0$ and $u\in \cN$,
\begin{itemize}
\item $DP_t(u).\overline{\cC^\cF(u)}\subset \cC^\cF(P_t(u)), \forall t\ge\tau_0$, provided $\|P_{[0,t]}(u)\|\le\eta_0$;
\item $DP_{-t}(u).\overline{\cC^\cE(u)}\subset \cC^\cE(P_{-t}(u)), \forall t\ge\tau_0$, when $\|P_{[-t,0]}(u)\|\le\eta_0.$ 
\end{itemize}
\item The bundle $\cE$ can be extended continuously and invariantly on $K_{\eta_0}^+$ and the bundle $\cF$ can be extended continuously and invariantly on $K_{\eta_0}^-$. The extensions are unique and are still denoted by $\cE$ and $\cF$. The invariance means that
$$DP_t(u)(\cE(u))=\cE(P_t(u)),~~~\forall u\in K_{\eta_0}^+,~\forall t\ge 0,$$
$$DP_{-t}(u)(\cF(u))=\cF(P_{-t}(u)),~~~\forall u\in K_{\eta_0}^-,~\forall t\ge 0.$$
\end{itemize}
\end{Lemma}

\smallskip

\subsubsection{Plaque family for fibered flows}
The center-stable plaques are defined on the positively invariant set $K_\eta^+$ and the center-unstable plaques are defined on the negatively invariant set $K_\eta^-$. Recall $\eta_0$ in Lemma~\ref{Lem:extend-bundle}.
\begin{Definition}\label{d.plaque}
Let us fix $\eta\in(0,\eta_0)$.
A \emph{$C^k$-plaque family tangent to $\cE$ over a subset $Y\subset K_\eta^+$} 
is a family of $C^k$-diffeomorphisms onto their image $\{\Theta_u\colon \cE(u)\to \cN_x,\; x\in K,\; u\in Y\cap \cN_x\}$, such that $\Theta_u(u)=u$, the image of $D\Theta_u(u)$ coincides with $\cE(u)$ and such that $\Theta_u$ depends continuously on $u\in Y$ for the $C^k$-topology.

\emph{ For $\alpha>0$, we denote by $\cW^{cs}_{\alpha}(u)$ the ball centered at $u$ and of radius $\alpha$
inside $\Theta_u(\cE(u))$ with respect to the induced metric on $\cN_x$ and we denote by $\cW^{cs}(u)=\cW^{cs}_1(u)$.
Sometimes the plaque family will be denoted by $\cW^{cs}=\{\cW^{cs}(u), u\in K_\eta^+\}$. For $x\in K$, we denote $\cW^{cs}_{\alpha_\cE}(x)=\cW^{cs}_{\alpha_\cE}(0_x)$.}

The plaque family is \emph{locally (forward) invariant} by the time-one map of the flow $(P_t)$ if there exists $\alpha_\cE>0$
such that for  any $u\in K_\eta^+$ we have
$$P_1(\cW^{cs}_{\alpha_\cE}(u))\subset \cW^{cs}(P_1(u)).$$

\medskip

One has similar definitions for $C^k$-plaque family tangent to $\cF$ associated to $K_\eta^-$.

\end{Definition}

Hirsch-Pugh-Shub's plaque family theorem~\cite{hirsch-pugh-shub} generalizes to local fibered flows.
\begin{Theorem}\label{t.plaque}
For any local fibered flow $(\cN,P)$ admitting a dominated splitting $\cN=\cE\oplus \cF$
there exists $\eta\in(0,\eta_0)$ and a $C^1$-plaque family tangent to $\cE$ over $K_\eta^+$ which is locally invariant by $P_1$.
If the flow is $C^2$ and if $\cE$ is $2$-dominated, then the plaque family can be chosen $C^2$.

\smallskip

Similar statements hold for $\cF$ and $K_\eta^-$.
\end{Theorem}

\subsubsection{Plaque family for sequences of diffeomorphisms}
The proof of Theorems~\ref{t.plaque} is very similar to~\cite[Theorem 5.5]{hirsch-pugh-shub}.
It is a consequence of a more general result that we state now.
We denote by $d=d^\cE+d^\cF$ the dimensions of the fibers of $\cN,\cE,\cF$ and endow $\RR^d$ with the standard euclidean metric. For $\chi>0$ let us define
the horizontal cone
$$\cC_\chi=\bigg\{(x,y)\in\RR^{d^\cE}\times \RR^{d^\cF}, \chi\|x\|> \|y\|\bigg\}\cup\bigg\{0\bigg\}.$$
\begin{Definition}\label{d.sequence}
A \emph{sequence of $C^k$-diffeomorphisms of $\RR^d$ bounded by constants $\beta,C>0$}
is a sequence $\underbar F$ of diffeomorphisms $f_n\colon U_n\to V_n$, $n\in \NN$, where
$U_n,V_n\subset \RR^d$ contain $B(0,\beta)$, such that $f_n(0)=0$ and
such that the $C^k$-norms of $f_n, f_n^{-1}$ are bounded by $C$.

We denote by $\sigma(\underbar F)$ the shifted sequence $(f_{n+1})_{n\geq 0}$ associated to $\underbar F=(f_n)_{n\geq 0}$.
\smallskip

The sequence has a \emph{dominated splitting} if there exists $\tau_0\in \NN$
such that for any $n\in \NN$ and for any 
{$z\in B(0,\beta)\cap f_n^{-1}(B(0,\beta))\cap \dots \cap (f_{n+\tau_0-1}\circ\cdots\circ f_n)^{-1}(B(0,\beta))$,}
the cone $\cC_1$ is mapped by $D(f_{n+\tau_0-1}\circ \dots \circ f_n)^{-1}(z)$
inside the smaller cone $\cC_{1/2}$.
\smallskip

The center stable direction of the dominated splitting is \emph{$2$-dominated}
if there exists $\tau_0\in \NN$ such that for any 
{$z\in B(0,\beta)\cap f_n^{-1}(B(0,\beta))\cap \dots \cap (f_{n+\tau_0-1}\circ\cdots\circ f_n)^{-1}(B(0,\beta))$}
and for any unit vectors $u,v$ satisfying
$D(f_{n+\tau_0-1}\circ \dots \circ f_n)(z).u\in\cC_1$ and $v\in \RR^d\setminus \cC_1$,
then
$$ \|D(f_{n+\tau_0-1}\circ \dots \circ f_n)(z).u\|^2\leq \frac 1 2 \|D(f_{n+\tau_0-1}\circ \dots \circ f_n)(z).v\|. $$
\end{Definition}

\begin{Theorem}\label{t.generalizedplaque}
For any $C,\beta,\tau_0$, there exists $\alpha\in (0,\beta)$, and for any sequence of $C^1$-diffeomor\-phisms
$\underbar F=(f_n)$ of $\RR^d$ bounded by $\beta,C$ with a dominated splitting, associated to the constant $\tau_0$, there
exists a $C^1$-map $\psi=\psi(\underbar F)\colon \RR^{d^\cE}\to \RR^{d^\cF}$ such that:
\begin{itemize}
\item[--] For any $z,z'$ in the graph $\{(x,\psi(\underbar F)(x)), x\in \RR^{d^\cE}\}$, the difference $z'-z$ is contained in $\cC_{\frac 1 2}$.
\item[--] (Local invariance.) $f_0\left( \{(x,\psi(\underbar F)(x)), x\in \RR^{d^\cE}, |x|<\alpha\}\right) \subset \{(x,\psi(\sigma(\underbar F))(x)), x\in \RR^{d^\cE}\}$.
\item[--] The function $\psi$ depends continuously on $\underbar F$ for the $C^1$-topology:
for any $R,\varepsilon>0$, there exists $N\geq 1$ and $\delta>0$
such that if the two sequences $\underbar F$ and $\underbar F'$ satisfy
$$\|(f_n-f_n')|_{B(0,\beta)}\|_{C^1}\leq \delta \text{ ~~~~~~~for } 0\leq n \leq N,$$
then $\|(\psi(\underbar F)-\psi(\underbar F'))|_{B(0,R)}\|_{C^1}$ is smaller than $\varepsilon$.
\item[--] For sequences of $C^2$-diffeomorphisms $\underbar F$ such that the center stable direction of the dominated splitting is $2$-dominated (still for the constant $\tau_0$), the function $\psi(\underbar F)$ is
$C^2$ and depends continuously on $\underbar F$ for the $C^2$-topology.
\end{itemize}
\end{Theorem}
\noindent
The proof of this theorem is standard. It is obtained by
\begin{itemize}
\item[--] introducing a sequence of diffeomorphisms
$(\widehat f_n)$ defined on the whole plane $\RR^d$ which coincide with the diffeomorphisms
$f_n$ on a uniform neighborhood of $0$ and with the linear diffeomorphism
$Df_n(0)$ outside a uniform neighborhood of $0$,
\item[--] applying a graph transform argument.
\end{itemize}

Theorem~\ref{t.plaque} is a direct consequence of Theorem~\ref{t.generalizedplaque}:
for each $u\in K_\eta^+\cap\cN_x$, we consider the sequence of local diffeomorphisms
$P_1\colon \cN_{\varphi_n(x)}\to \cN_{\varphi_{n+1}(x)}$.
There exist uniformly bounded linear isomorphisms which identify $\cN_{\varphi_n(x)}$
with $\RR^d$ and send the spaces $\cE(u)=\cE(\varphi_n(x))$ and $\cF(u)=\cF(\varphi_n(x))$
to $\RR^{d^\cE}\times \{0\}$ and $\{0\}\times \RR^{d^\cF}$. Since the isomorphisms are bounded
we get a sequence of diffeomorphisms as in Definition~\ref{d.sequence}.
Theorem~\ref{t.generalizedplaque} provides a plaque in $\cN_x$ and which depends continuously on $x$.

\subsubsection{Uniqueness}
There is no uniqueness in Theorem~\ref{t.generalizedplaque},
but once we have fixed the way of choosing $\widehat f_n$, the invariant graph becomes unique
(and this is used to prove the continuity in Theorem~\ref{t.generalizedplaque}).
Also the following classical lemma holds.

\begin{Proposition}\label{p.uniqueness}
In the setting of Theorem~\ref{t.generalizedplaque}, up to reducing $\alpha$, the following property holds.
If there exists $z'$ in the graph of $\psi(\underbar F)$ and $z\in \RR^d$ such that  for any $n\geq 0$
\begin{itemize}
\item[--] the iterates of $z$ and $z'$ by $f_n\circ\dots\circ f_0$ are defined and belong to
$B(0,\alpha)$,
\item[--] $(f_n\circ\dots\circ f_0(z))-(f_n\circ\dots\circ f_0(z'))\in \cC_1$,
\end{itemize}
then $z$ is also contained in the graph of $\psi(\underbar F)$.
\end{Proposition}
\begin{proof}
Let us assume by contradiction that $z=(x,y)$ is not contained in the graph of $\psi$ and
let us denote $\widehat z=(x,\psi(x))$ with $\psi(x)\neq y$.
The line containing the iterates of $z$ and $\widehat z$
by the sequence $f_{n-1}\circ\dots\circ f_0$, $n\geq 1$,
remains tangent to the cone $\RR^d\setminus \cC_1$ (by the dominated splitting).
The line containing the iterates of $z$ and $z'$
and the line containing the iterates of $\widehat z$ and $z'$ are tangent to $\cC_1$
by our assumption, and by the two items of Theorem~\ref{t.generalizedplaque}.

The domination implies that the distances between the $n^\text{th}$ iterates of $z,\widehat z$
gets exponentially larger than their distance to the $n^\text{th}$ iterate of $z'$.
This contradicts the triangular inequality.
\end{proof}

\begin{Remark-numbered}\label{r.plaque-invariance}
The plaque family $\cW^{cs}$ given by Theorem~\ref{t.plaque} is a priori only invariant
by the time-$1$ map $P_1$ of the flow but the previous proposition shows that
if $\eta, \alpha_\cE>0$ are small enough, then for any $u\in K_\eta^+$ and $z\in \cW^{cs}(u)$ such that
$P_{n}(z)\in \cW^{cs}_{\alpha_\cE}(P_n(u))$ for each $n\in\NN$,
we have $P_{t}(z)\in \cW^{cs}(P_t(u))$ for any $t>0$.
Indeed, by invariance of the cone $\cC^\cF$, the point $P_{t}(z)$ belongs to $\cC^\cE(P_t(u))$ for any $t\geq 0$.
\end{Remark-numbered}

\subsubsection{Coherence}
The uniqueness allows us to deduce that when plaques intersect then they have to \emph{match}, i.e.
to be contained in a larger sub-manifold.

\begin{Proposition}\label{p.coherence}
Fix a plaque family $\cW^{cs}$ as given by Theorem~\ref{t.plaque}.
Up to reducing the constants $\eta,\alpha_\cE>0$, the following property holds.
For any $y\in K$, $u\in K^+_\eta\cap \cN_y$, $u'\in K_\eta^+$,
and for any sets $X$ and $X'$ satisfying $u\in X\subset \cW^{cs}_{\alpha_\cE}(u)$, $u'\in X'\subset \cW^{cs}_{\alpha_\cE}(u')$, if
\begin{enumerate}
\item\label{i.infinite-return} $(\varphi_t(y))_{t\in[0,\infty)}$ has arbitrarily large iterates in the
$r_0$-neighborhood of $K\setminus V$;
\item\label{i.intersect-control} $y$ belongs to the $r_0$-neighborhood of $K\setminus V$, $\pi_{y}(X')\cap X\neq \emptyset$ and
$\diam (P_t(X)), \diam (P_t(X'))$ are smaller than $\alpha_\cE$ for all $t\geq 0$,
\end{enumerate}
then $\pi_{y}(X')$ is contained in $\cW^{cs}(u)$.
\end{Proposition}
\begin{proof}
Proposition~\ref{p.uniqueness} gives $\alpha>0$.
One chooses $\delta>0$ small enough such that for any $x$ in the $r_0$-neighborhood of $K\setminus V$, for any $x'$ is $\delta$-close to $x$, one has that $\pi_x(\cC_{1/2}^\cE(x'))\subset \cC_1^{\cE}(\pi_x(x'))=\cC_1^{\cE}(x)$.

Take $v\in X$ and $v'\in X'$ such that $\pi_y(v')=v$. Provided $\eta,\alpha_\cE$ are small enough, one has that $\|P_t(v)\|$ and $\|P_t(v')\|$ are small for all $t\ge 0$. Then by the Global invariance (the version of Remark~\ref{r.identification}.(g)), there is a bi-Lipschitz parametrization $\theta$ such that $|\theta(0)|\le 1/4$ and $d(\varphi_t(y),\varphi_{\theta(t)}(y'))<\delta$ for all $t\ge 0$, where $y'$ is the base point of $u'$ and $v'$, i.e. $u'\in \cN(y')$.

By Item~\ref{i.infinite-return}, one can choose a sequence of times $T_n\to+\infty$ such that $\varphi_{T_n}(y)$ is in the $r_0$-neighborhood of $K\setminus V$. Thus, we also have a sequence $(T_n'=\theta(T_n))_{n\in\NN}$ tending to $+\infty$.

It suffices to prove that for any $v''\in X'$, one has that $z=\pi_y(v'')\in \cW^{cs}(u)$. By Theorem~\ref{t.generalizedplaque}, one has that $P_{T_n'}(v')-P_{T_n'}(v'')\in \cC_{1/2}^\cE(\varphi_{T_n'}(y'))$. By the choice of $\delta$, one has that  $\pi_{\varphi_{T_n}(y)}(P_{T_n'}(v''))-\pi_{\varphi_{T_n}(y)}(P_{T_n'}(v'))\in \cC_{1}^\cE(\varphi_{T_n}(y))$. By the Global invariance again, one has that $P_{T_n}(z)-P_{T_n}(v)\in \cC^\cE(\varphi_{T_n}(y))$. Since the complement of the cone field $\cC^\cE$ is invariant by forward iterations,
and $P_{t}(z)-P_{t}(v)\in \cC^\cE(\varphi_{t}(y))$ for a sequence of arbitrarily large times $t$, then $P_{t}(z)- P_{t}(v)\in \cC^\cE(\varphi_{t}(y))$ for all $t>0$. One has the following facts:
\begin{itemize}
\item[--] By our assumptions, and requiring $\alpha_\cE<\alpha$ we have
$\|P_t(v)-P_t(u)\|<\alpha$ for any $t\geq 0$.
\item[--] Since $P_t(z)\in P_t(\pi_{y}(X'))$,
we also have 
$\|P_t(z)-P_t(u)\|\leq \alpha$ for any $t\geq 0$ by the Global invariance.
\end{itemize}
Thus, the assumptions of Proposition~\ref{p.uniqueness} are satisfied for $P_1$ and one can conclude $z\in \cW^{cs}(u)$. This complete the proof by using Proposition~\ref{p.uniqueness}.
\end{proof}

\subsubsection{Limit dynamics in periodic fibers}

We state a consequence of the existence of plaque families. It will be used for the center-unstable plaques $\cW^{cu}$.
Note that any $u\in \cN$ such that $\|P_{-t}(u)\|$ is small for any $t\geq 0$  has a plaque $\cW^{cu}(u)$ by Theorem~\ref{t.plaque}.

\begin{Proposition}\label{p.fixed-point}
For any local fibered flow  $(P_t)$ on a bundle $\cN$ admitting a dominated splitting $\cN=\cE\oplus \cF$
where $\cE$ is one-dimensional, if $\eta>0$ is small enough the following property holds.

For any periodic point $z\in K$ with period $T$ and any $u\in K^-_\eta\cap \cN_z$ satisfying $0_z\not\in \cW^{cu}(u)$,
there exists $p\in \cN_z$ such that $P_{2T}(p)=p$ and $P_{-t}(u)$ converges to the orbit of $p$
when $t$ goes to $+\infty$.
\end{Proposition}
\begin{proof}

Let $\alpha_\cE,\alpha_\cF$ be the constants associated to $\cE,\cF$ as in Theorem~\ref{t.plaque}.
The plaques $\cW^{cu}(u)$ and $\cW^{cs}(0_z)$
intersect at a (unique) point $y$.

Since $\|u\|$ is small, by the local invariance of the plaque families, $y$ is also
the intersection between $P_{1}(\cW^{cs}_{\alpha_\cE}(\varphi_{-1}(z)))$ and
$\cW^{cu}_{\alpha_\cF}(u)$.
One deduces that $P_{-1}(y)$
is the (unique) intersection point between
the plaques $\cW^{cu}(P_{-1}(u))$ and $\cW^{cs}(\varphi_{-1}(z))$.
Repeating this argument inductively,
one deduces that the backward orbit of $y$ by $P_{-1}$
remain in the plaques $\cW^{cu}(P_{-k}(u))$ and $\cW^{cs}(\varphi_{-k}(z))$.
From Remark~\ref{r.plaque-invariance}, any backward iterate $P_{-t}(y)$
belongs to $\cW^{cu}(P_{-t}(u))$.

One has that $y\neq 0_z$ by the fact that $0_z\neq\cW^{cu}(u)$. The domination implies that
 the distance $d(P_{-k}(u),P_{-k}(y))$ is exponentially smaller than
$d(P_{-k}(y), 0_{\varphi_{-k}(z)})$ as $k\to +\infty$.
So one has that $d(P_{-k}(u),P_{-k}(y))$ goes to $0$.
The same argument applied to $P_{-s}(u)$,
$s\in [0,1]$ shows that the distance of $P_{-t}(u)$
to $\cW^{cs}(\varphi_{-t}(z))$ converges to $0$ as $t\to +\infty$.
In particular, the limit set of the orbit of $u$ under $P_{-2T}$
is a closed subset $L$ of $\cW^{cs}(z)$.
In the case $L$ is a single point $p$, the conclusion of the proposition follows.

We assume now by contradiction  that $L$ is not a single point.
There exists $q\neq 0_z$ invariant by $P_{2T}$
in $L$ such that $L$ intersects the
open arc $\gamma$ in $\cW^{cs}(z)$ bounded by $0_z$ and $q$.
Note that the forward iterates $P_k(\gamma)$ by $P_1$ remain small,
hence in $\cW^{cs}(\varphi_k(z))$ by the local invariance of $\cW^{cs}$.
From Remark~\ref{r.plaque-invariance}, any iterate
$P_t(\gamma)$ is contained in $\cW^{cs}(\varphi_t(z))$, $t\in \RR$.

Up to replacing $u$ by a backward iterate, one can assume that $y$ belongs to $\gamma$.
This shows that $P_{-t}(y)$ is the intersection between
$\cW^{cu}(P_{-t}(u))$ and $\cW^{cs}(P_{-t}(z))$ for any $t\geq 0$.
The set $L$ is thus the limit set of the orbit of $y$ under $P_{-2T}$.
This reduces to a one-dimensional dynamics
for an orientation preserving diffeomorphism,
and $L$ has to be a single point, a contradiction.
\end{proof}

\subsubsection{Distortion control}

The following lemma restates the classical Denjoy-Schwartz argument in our setting.
\begin{Lemma}\label{Lem:schwartz}
Let us assume that the local fibered flow $(P_t)$ is $C^2$, that $\cE$ is one-dimensional and that
$\cW^{cs}$ is a $C^2$ locally invariant plaque family.
Then, there is $\beta_S>0$ and for any $C_{Sum}>0$, there are $C_S,\eta_S>0$ with the following property.

For any $x\in K$, for any
interval $I\subset \cW^{cs}(x)$ and any $n\in \NN$ satisfying
$$
\forall m\in \{0,\dots,n\},\; P_m(I)\subset B(0,\beta_S) \text{ and }
\sum_{m=0}^{n} |P_{m}(I)|\leq C_{Sum},$$
then (1) for any $u,v\in I$ we have
\begin{equation}\label{e.distortion}
C_S^{-1}\leq \frac {\|DP_n(u)|_{I}\|}{\|DP_n(v)|_{I}\|}\leq C_S;
\end{equation}
in particular $\|DP_n(u)|_{I}\|\leq C_S \frac{|P_n(I)|}{|I|}$;
\smallskip

\noindent
(2) any interval $\widehat I\subset \cW^{cs}(x)$ containing $I$
with $|\widehat I |\leq (1+\eta_S)|I|$ satisfies
$|P_{n}(\widehat I)|\leq 2|P_{n}(I)|$;
\smallskip

\noindent
(3) any interval $\widehat I\subset \cW^{cs}(x)$ containing $I$
with $|P_{n}(\widehat I)|\leq (1+\eta_S)|P_{n}(I)|$ satisfies
$|\widehat I |\leq 2|I|$.
\end{Lemma}

The proof is similar to~\cite[Chapter I.2]{dMvS}.

\subsection{Hyperbolic iterates}\label{ss:hyperbolicreturns}
We continue with the setting of Section~\ref{s.plaque} and
we fix two locally invariant plaques families $\cW^{cs}$ and $\cW^{cu}$ tangent to $\cE$ and $\cF$ respectively,
and two constants $\alpha_\cE,\alpha_\cF$ controlling  the local invariance and the coherence inside these plaques
as in the previous sections.
The plaques are defined at points of $K$ but also at points $u$ in $K^+_\eta$ or $K^-_\eta$ whose positive orbit or negative orbit by $P_t$ are contained in $\eta$-neighborhood of $K$.
The quantities $\eta,\alpha_\cE,\alpha_\cF>0$ may be reduced in order to satisfy further properties below.

In case $\cE$ is $2$-dominated, $\cW^{cs}$ will be a $C^2$-plaque family. Remember that $\tau_0,\lambda$ are the constants
associated to the domination, as introduced in Section~\ref{ss.assumptions}.

\subsubsection{Hyperbolic points}
We introduce a first notion of hyperbolicity.
\begin{Definition}
Let us fix $C_\cE,\lambda_\cE>1$.
A piece of orbit $(x,\varphi_{t}(x))$ in $K$ is \emph{$(C_\cE,\lambda_\cE)$-hyperbolic for $\cE$} if
for any $s\in (0,t)$, we have
$$\|DP_{s}{|\cE(x)}\| \leq C_\cE\lambda_\cE^{-s}.$$
A point $x$ is \emph{$(C_\cE,\lambda_\cE)$-hyperbolic for $\cE$} if
$(x,\varphi_{t}(x))$ is $(C_\cE,\lambda_\cE)$-hyperbolic for $\cE$ for any $t>0$.
We have similar definitions for the bundle $\cF$ (considering the flow $t\mapsto P_{-t}$).
\end{Definition}

By the continuity of the fibered flow for the $C^1$-topology, the hyperbolicity extends to nearby orbits
(the proof is easy and omitted).
\begin{Lemma}\label{l.shadowingandhyperbolicity}
Let us assume that $\cE$ is one-dimensional.
For any $\lambda'>1$, there exist $C',\delta,\rho>0$
such that for any $x,y\in K$, $t>0$ and $\theta\in \lip_{1+\rho}$ satisfying
 $\theta(0)=0$ and $$d(\varphi_s(x),\varphi_{\theta(s)}(y))<\delta \text{ for each } s\in [0,t],$$
then
$$\|DP_{\theta(t)}|\cE(y)\|\leq C'{\lambda'}^{t}\|DP_t|\cE(x)\|.$$
\end{Lemma}

Hyperbolicity implies summability for the iterations inside one-dimensional plaques.

\begin{Lemma}[Summability]\label{l.summability-hyperbolicity}
Let us assume that $\cE$ is one-dimensional and consider $\lambda_\cE,C_\cE>1$.
Then, there exist $C'_\cE>1$ and $\delta_\cE>0$ with the following property.

For any piece of orbit $(x,\varphi_t(x))$ which is $(C_\cE,\lambda_\cE)$-hyperbolic for $\cE$,
for any interval $I\subset \cW^{cs}(x)$ containing $0$
whose length $|I|$ is smaller than $\delta_\cE$, and for any interval $J\subset I$ one has
$$|P_t(J)|\leq C_\cE\lambda_\cE^{-t/2}\;|J| \text{ and }
\sum_{m=0}^{[t]} |P_m(J)|\leq C'_\cE\; |J|.$$
\end{Lemma}
\begin{proof}
Let $\eta>0$ be small such that $1+\eta<\lambda_\cE^{1/2}$.
Then, there exists $\delta_0$ such that for any $y\in K$ and any interval $I_0\subset \cW^{cs}(y)$
containing $0$ whose length is smaller than $\delta_0$, one has 
$$\forall s\in[0,1],~~~|P_{s}(I_0)|\leq (1+\eta)\|DP_{s}{|\cE(y)}\|\; |I_0| .$$
Let us choose $\delta_\cE$ satisfying $\delta_\cE C_\cE\lambda_\cE^{1/2}<\delta_0$.
One checks inductively that the length of $P_k(I)$ is smaller than $\delta_0$ for each $k\in [0,t]$.
The conclusion of the lemma follows.
\end{proof}

\subsubsection{Pliss points}
We introduce a more combinatorial notion of hyperbolicity (only used for $\cF$). Recall that the integer $\tau_0$
and the number $\lambda>1$ are fixed as in Subsection~\ref{ss.assumptions}.

\begin{Definition}\label{d.pliss}
For $T\geq 0$ and $\gamma>1$, we say that a piece of orbit $(\varphi_{-t}(x),x)$ is a \emph{$(T,\gamma)$-Pliss string}
(for the bundle $\cF$)
if there exists an integer $s\in[0,T]$ such that
$$\text{for any integer } m\in \bigg[0,\frac {t-s}{\tau_0}\bigg],\quad  \prod_{n=0}^{m-1}\|DP_{-\tau_0}|{\cF(\varphi_{-(n\tau_0+s)}(x))}\|\leq \gamma^{-m\tau_0}.$$
A point $x$ is $(T,\gamma)$-Pliss (for $\cF$) if  $(\varphi_{-t}(x),x)$ is a $(T,\gamma)$-Pliss string for any $t>0$.

For simplicity, a piece of orbit $(\varphi_{-t}(x),x)$ is a \emph{$T$-Pliss string} if it is a $(T,\lambda)$-Pliss string and $x$ is \emph{$T$-Pliss} if it is $(T,\lambda)$-Pliss, where $\lambda$ is the constant for the domination.
\end{Definition}

The next proposition will allow us to find iterates that are Pliss points and belong to $U$.

\begin{Proposition}\label{l.summability}
Let us assume that $\cE$ is one-dimensional and
let $W$ be a set such that $\cE$ is uniformly contracted on $\bigcup_{s\in [0,1]}\varphi_s(W)$.
Then, there exist $C_\cE,\lambda_\cE>1$ such that any $T_\cF\geq 0$ large enough has the following property.

If there are $x\in K$ and integers $0\leq k\leq \ell$ such that
\begin{itemize}
\item[--] $(\varphi_{-\ell}(x),\varphi_{-k}(x))$ is a $T_\cF$-Pliss string,
\item[--] for any $j\in \{1,\dots,k-2\}$
either $\varphi_{-j}(x)\in W$ or the piece of orbit $(\varphi_{-\ell}(x),\varphi_{-j}(x))$ is not a $T_\cF$-Pliss string,
\end{itemize}
then $(\varphi_{-k}(x),x)$ is $(C_\cE,\lambda_\cE)$-hyperbolic for $\cE$.

\medskip

Similarly if there are $x\in K$ and $k\geq 0$ such that
\begin{itemize}
\item[--] $\varphi_{-k}(x)$ is $T_\cF$-Pliss,
\item[--] for any $j\in \{1,\dots,k-2\}$
either the point $\varphi_{-j}(x)$ is in $W$ or $\varphi_{-j}(x)$ is not $T_\cF$-Pliss,
\end{itemize}
then $(\varphi_{-k}(x),x)$ is $(C_\cE,\lambda_\cE)$-hyperbolic for $\cE$.
\end{Proposition}

\begin{proof}
The proof is essentially contained in \cite[Lemma 9.20]{CP}.
Recall that $\lambda>1$ is the constant for the domination.
There exist $C_0,\lambda_0>1$ such that for any piece of orbit
$(y, \varphi_t(y))$ in $\bigcup_{s\in [0,1]}\varphi_s(W)$, one has
$\|DP_t{|\cE(y)}\|\leq C_0\lambda_0^{-t}$.
Let $C_1>1$ such that
\begin{equation}\label{C1}
\forall y\in K,~\forall s\in [-\tau_0,\tau_0],~~~\|DP_{s}{|\cE(y)}\|\leq C_1\lambda^{-s}.
\end{equation}
One takes $C_\cE=C_0^2C_1^3$ and $\lambda_\cE>1$ smaller than $\min(\lambda,\lambda_0)$.
One then chooses $T_\cF\geq 0$ large such that $C_0C_1\lambda^{-T_\cF}<\lambda_\cE^{-T_\cF}$.
\smallskip

Let $x\in K$, $0\leq k\leq \ell$ be as in the statement of the lemma. We introduce the set
$$\cP=\bigg\{j\in\{1,\dots,k-2\},\; (\varphi_{-\ell}(x),\varphi_{-j}(x)) \text{ is a } T_\cF\text{-Pliss string}\bigg\}.$$

The set $\cP$ decomposes into intervals $\{a_i,1+a_i,\dots,b_i\}\subset \{1,\dots,k-2\}$, with $i=1,\dots,i_0$, such that
$b_i+1<a_{i+1}$. By convention we set $b_0=0$.

\begin{Claim-numbered}
$b_i-a_i\geq T_\cF$ unless $\{a_i,\dots,b_i\}$ contains $1$ or $k-2$.
\end{Claim-numbered}
\begin{proof}
We only consider the case that $\{a_i,\dots,b_i\}$ does not contain $1$ or $k-2$.

 By maximality of each interval, $(\varphi_{-\ell}(x),\varphi_{-b_i}(x))$ has to be a $0$-Pliss string. This forces that $b_i-a_i\geq T_\cF$.
\end{proof}

\begin{Claim-numbered}\label{c.pliss}
Consider $n_1,n_2\in \{0,\dots,k\}$ with $n_1<n_2$ and such that
$(\varphi_{-\ell}(x),\varphi_{-n_2}(x))$ is a $0$-Pliss string and $(\varphi_{-\ell}(x),\varphi_{-j}(x))$ is not a $0$-Pliss string
for $n_1<j<n_2$. Then for any $0\leq m < (n_2-n_1)/\tau_0$,
$$\|DP_{m\tau_0}{|\cE(\varphi_{-n_2}(x))}\|\leq \lambda^{-m\tau_0}.$$
\end{Claim-numbered}
\begin{proof}
One checks inductively that
\begin{equation}\label{e.pliss}
\prod_{n=0}^{m-1}\|DP_{\tau_0}{|\cF(\varphi_{-n_2+n\tau_0}(x))}\|\leq \lambda^{m\tau_0}.
\end{equation}
Indeed if this inequality holds up to an integer $m-1$ and fails for $m$,
the piece of orbit $(\varphi_{-n_2}(x),\varphi_{-n_2+m\tau_0}(x))$ is a $0$-Pliss string.
It may be concatenate with $(\varphi_{-\ell}(x),\varphi_{-n_2}(x))$, implying that $(\varphi_{-\ell}(x),\varphi_{-n_2+m\tau_0}(x))$
is a $0$-Pliss string. This is a contradiction since $n_2-m\tau_0> n_1$.

The estimate of the claim follows from~\eqref{e.pliss} by domination.
\end{proof}

The proposition will be a consequence of the following properties.
\begin{Claim-numbered}\label{c.return0}
\begin{enumerate}

\item 
The piece of orbit $(\varphi_{-b_i}(x), \varphi_{-b_{i-1}}(x))$ is $(C_0C_1,\lambda_\cE)$-hyperbolic for $\cE$,
for any $i\in \{1,\dots,i_0\}$, unless $i=i_0$ and $b_{i_0}=k-2$.

Moreover if $i\neq 1$, then  $\|DP_{b_i-b_{i-1}}{|\cE(\varphi_{-b_i}(x))}\|\leq \lambda_\cE^{-(b_{i}-b_{i-1})}$.

\item If $b_{i_0}=k-2$, then $(\varphi_{-k}(x), \varphi_{-b_{i_{0}-1}}(x))$ is $(C_0C_1,\lambda_\cE)$-hyperbolic for $\cE$.

\item If $b_{i_0}<k-2$, then $(\varphi_{-k}(x), \varphi_{-b_{i_0}}(x))$ is $(C_0C_1,\lambda_\cE)$-hyperbolic for $\cE$.
\end{enumerate}
\end{Claim-numbered}
\begin{proof}
In order to check the first item, one introduces the smallest $j\in \{a_i,\dots,b_i\}$ such that
$(\varphi_{-\ell}(x),\varphi_{-j}(x))$ is a $0$-Pliss string: it exists unless $i=i_0$ and $b_{i_0}=k-2$.
By our assumptions, the piece of orbit $(\varphi_{-b_i}(x), \varphi_{-j}(x))$ is contained in $\bigcup_{s\in [0,1]}\varphi_s(W)$, hence is $(C_0,\lambda_\cE)$-hyperbolic for $\cE$.
Then Claim~\ref{c.pliss} gives
$\|DP_{m\tau_0}{|\cE(\varphi_{-j}(x))}\|\leq \lambda^{-m\tau_0}$ for any $m\in \{0,\dots, (j-b_{i-1})/\tau_0\}$.
One concludes the first part of item 1 by combining these estimates with~\eqref{C1}.
\smallskip

Note that one also gets the estimate:
$$\|DP_{b_i-b_{i-1}}|{\cE}(\varphi_{-b_i}(x))\|\leq C_0C_1\lambda_\cE^{-(b_{i}-b_{i-1})}\bigg(\frac \lambda {\lambda_\cE}\bigg)^{-(j-b_{i-1})}.$$
If $i\geq 2$, one gets $j-b_{i-1}\geq j-a_i= T_\cF$, hence by our choice of $T_\cF$:
$$\|DP_{b_i-b_{i-1}}|{\cE}(\varphi_{-b_i}(x))\|\leq C_0C_1\lambda_\cE^{-(b_{i}-b_{i-1})}\bigg(\frac \lambda {\lambda_\cE}\bigg)^{-T_\cF}
\leq \lambda_\cE^{-(b_{i}-b_{i-1})}.$$
This gives the second part of item 1.
\smallskip

The proofs of items 2 and 3 are similar to the proof of item 1:
for item 2, one introduces the smallest $j\geq a_{i_0}$ such that $(\varphi_{-\ell}(x),\varphi_{-j}(x))$ is a $0$-Pliss string;
for item 3, one introduces the smallest $j\geq k$ such that $(\varphi_{-\ell}(x),\varphi_{-j}(x))$ is a $0$-Pliss string and use the fact that
$(\varphi_{-\ell}(x),\varphi_{-k}(x))$ is a $T_\cF$-Pliss string.
\end{proof}

From Claim~\ref{c.return0}, one first checks that for each $i\in \{1,\dots,i_0\}$
$$\|DP_{k-b_i}{|\cE}(\varphi_{-k}(x))\|\leq C_0C_1\lambda_\cE^{k-b_i}.$$
For each $m\in \{0,\dots,k\}$, either there exists $i\in \{1,\dots,i_0\}$ such that $b_{i-1}\leq m\leq b_i$
or $b_{i_0}\leq m \leq k$. Using items 1 or 3, one concludes
$$\|DP_{k-m}{|\cE}(\varphi_{-k}(x))\|\leq C_0^2C_1^2\lambda_\cE^{k-m}.$$
Combining with~\eqref{C1}, one gets the required bound on $\|DP_{k-t}{|\cE}(\varphi_{-k}(x)\|$ for any $t\in [0,k]$.

This proves the lemma for pieces of orbits $(\varphi_{-\ell}(x),x)$ such that $(\varphi_{-\ell}(x),\varphi_{-k}(x))$ is
a $T_\cF$-Pliss string.
The proof for half orbits $\{\varphi_{-t}(x),t>0\}$ such that $\varphi_{-k}(x)$ is $T_\cF$-Pliss is similar.
The proposition is proved.
\end{proof}

\subsubsection{Pliss points in $\cN$}
The notion of Pliss points extends to the bundle $\cN$.

\begin{Definition}
Let us fix $C_\cF,\lambda_\cF>1$, $T_\cF\geq 0$
and consider $u\in K^-_\eta$.
A point $u$ is  \emph{$(T_\cF,\lambda_\cF)$-Pliss (for $\cF$)} if
there exists an integer $s\in [0,T_\cF]$ such that for any $m\in \NN$,
$$ \prod_{n=0}^{m-1}\|D P_{-\tau_0}{|\cF( P_{-(n\tau_0+s)}( u))}\| \leq \lambda_\cF^{-m\tau_0}.$$

\end{Definition}

By continuity and invariance of the spaces $\cF(u)$, one gets

\begin{Lemma}\label{l.cont3}
For any $\lambda_\cF\in (1,\lambda)$ and $T_\cF\geq 0$, there exists $\eta>0$ such that
for any $y\in K$ which is $T_\cF$-Pliss and for 
any point $u\in\cN_y$, if $u\in K_\eta^-$,
then $u$ is $(T_\cF,\lambda_\cF)$-Pliss.
\end{Lemma}

{The proof of Lemma~\ref{l.cont3} is standard by continuity, hence are omitted.}

\subsubsection{Unstable manifolds}\label{ss.unstable}
Pliss points have uniform unstable manifolds 
in the plaques
(see e.g.~\cite[Section 8.2]{abc-measure}).

\begin{Proposition}\label{p.unstable}
Consider a center-unstable plaque family $\cW^{cu}$ as given by Theorem~\ref{t.plaque}.
For any $\lambda_\cF>1$, $\beta_\cF>0$ and $T_\cF\geq 0$, there exists $\alpha>0$ and $\eta>0$ such that
for any $u\in K_\eta^-$,
if $u$ is  $(T_\cF,\lambda_\cF)$-Pliss,
then:
$$\forall t\geq 0,\quad \diam(P_{-t}(\cW^{cu}_{\alpha}(u)))\leq \beta_\cF\lambda_\cF^{-t/2}.$$
In particular, from Remark~\ref{r.plaque-invariance}, the image $P_{-t}(\cW^{cu}_{\alpha}(u))$ is contained in $\cW^{cu}(P_{-t}(u))$ for every $t\ge0$.
\end{Proposition}

\subsubsection{Lipschitz holonomy and rectangle distortion}\label{ss.rectangle}
Let us assume that the bundle $\cE$ is one-dimensional.
It is well-known that for a one-codimensional invariant foliation
whose leaves are uniformly contracted, the holonomies between transversals are Lipschitz.
In order to state a similar property in our setting we define the notion of rectangle.

\begin{Definition}\label{d.rectangle}
A \emph{rectangle} $R\subset \cN_x$ is a subset which is homeomorphic to $[0,1]\times B_{d-1}(0,1)$
by a homeomorphism $\psi$,
where $B_{d-1}(0,1)$ is the $(\dim\cN_x-1)$-dimensional unit ball such that:
\begin{itemize}
\item[--] the set $\psi(\{0,1\}\times B_{d-1}(0,1))$ is a union of two $C^1$-discs tangent to $\cC^\cF$,
and is called the  \emph{$\cF$-boundary $\partial^{\cF}R$},
\item[--] the curve $\psi([0,1]\times \{0\})$ is $C^1$ and tangent to $\cC^\cE$.
\end{itemize}
A rectangle $R$ has \emph{distortion bounded by $\Delta>1$}
if for any two $C^1$-curves $\gamma,\gamma'\subset R$ tangent to $\cC^\cE$
with endpoints in the two connected components of $\partial^{\cF}R$, then
$$\Delta^{-1} |\gamma|\leq |\gamma'| \leq \Delta |\gamma|.$$
\end{Definition}

\begin{Proposition}\label{p.distortion}
Assume that the local fibered flow is $C^2$.
For any $T_\cF,\lambda_\cF>1$,
there exist $\Delta>0$ and $\beta>0$ with the following property.
For any $y$, and $t>0$, and $R\subset \cN_y$ such that:
\begin{itemize}
\item[--] $(\varphi_{-t}(y),y)$ is $(T_\cF,\lambda_\cF)$-Pliss for $\cF$,
\item[--] $P_{-s}(R)$ is a rectangle and has diameter smaller than $\beta$ for each $s\in [0,t]$,
\item[--] if $D_1,D_2$ are the two components of $\partial^\cF(P_{-t}(R))$, then
$$d(D_1,D_2)>10.\max(\diam(D_1),\diam(D_2)),$$
\end{itemize}
then the rectangle $R$ has distortion bounded by $\Delta$.
\end{Proposition}
\begin{proof}
It is enough to prove the version of this result stated for a sequence of $C^2$-diffeomorphisms
with a dominated splitting. Then the argument is the same as~\cite[Lemma 3.4.1]{PS1}.
\end{proof}

\section{Topological hyperbolicity}\label{s.topological-hyperbolicity}

\noindent
{\bf Standing assumptions.}
In the whole section, $(\cN,P)$ is a $C^2$ local fibered flow over a topological flow $(K,\varphi)$
and $\pi$ is an identification compatible with $(P_t)$ on an open set $U$
such that:
\begin{enumerate}
\item[(A1)] there exists a dominated splitting $\cN=\cE\oplus \cF$ and the fibers of $\cE$ are one-dimensional,
\item[(A2)] $\cE$ is uniformly contracted on an open set $V$ containing $K\setminus U$,
\item[(A3)] $\cE$ is uniformly contracted over any periodic orbit $\cO\subset K$.
\end{enumerate}
From the last item and Proposition~\ref{p.2domination}, the bundle $\cE$ is $2$-dominated.
By Theorem~\ref{t.plaque}, one can fix a $C^2$-plaque family $\cW^{cs}$ tangent to $\cE$.

\smallskip

The goal of this section is to prove the following theorems (see Subsection~\ref{ss.conclusion-topological}):

\begin{Theorem}\label{Thm:recurrent-contraction}
Under the assumptions (A1), (A2), (A3), one of the following cases occurs:
\begin{enumerate}
\item[--] $K$ contains a normally expanded irrational torus.
\item[--] $\cE$ is \emph{topologically contracted} for recurrent points: for any $\rho>0$, there is $\varepsilon_0>0$ such that for any point $x\in K$ satisfying $\omega(x)=K$ (if it exists), the image $P_t({\cW}^{cs}_{\varepsilon_0}(x))$
is well-defined for any $t\ge 0$, has diameter smaller than $\rho$ and satisfies
$$\lim_{t\to+\infty} |P_t({\cW}^{cs}_{\varepsilon_0}(x))|=0.$$
\end{enumerate}

\end{Theorem}

\begin{Corollary}\label{Cor:ergodic-measure}
Under the assumptions (A1), (A2), (A3), if there is an ergodic measure $\mu$ such that $\supp(\mu)$ does not contain a normally expanded irrational torus, then $\mu$ is topologically contracted along $\cE$: for any $\rho>0$ there is $\varepsilon_0>0$ such that for $\mu$-almost every $x\in K$, the image $P_t({\cW}^{cs}_{\varepsilon_0}(x))$
is well-defined for any $t\ge 0$, has diameter smaller than $\rho$ and satisfies
$$\lim_{t\to+\infty} |P_t({\cW}^{cs}_{\varepsilon_0}(x))|=0.$$
\end{Corollary}
Corollary~\ref{Cor:ergodic-measure} is a direct consequence of Theorem~\ref{Thm:recurrent-contraction}
in restriction to $\supp(\mu)$.
\begin{Remark}
From Remark~\ref{r.torus}, it is enough to assume that $\supp(\mu)$ is not a normally expanded irrational torus.
\end{Remark}
\medskip

The topological contraction holds on a larger set under some minimality condition.

\begin{Theorem}\label{Thm:topologicalcontracting}
Under the assumptions (A1), (A2), (A3), one of the following properties occurs:
\begin{enumerate}
\item[--] There exists a non empty proper invariant compact subset $K'\subset K$ such that $\cE|_{K'}$ is not uniformly
contracted.
\item[--] $K$ is a normally expanded irrational torus.
\item[--] $\cE$ is \emph{topologically contracted}: for any $\rho>0$ there is $\varepsilon_0>0$ such that the image $P_t({\cW}^{cs}_{\varepsilon_0}(x))$
is well-defined for any $t\ge 0$, $x\in K$, has diameter smaller than $\rho$, and satisfies
$$\lim_{t\to+\infty}\sup_{x\in K} |P_t({\cW}^{cs}_{\varepsilon_0}(x))|=0.$$
\end{enumerate}
\end{Theorem}
\medskip

\noindent
{\bf Choice of constants.}
Let us remark that by assumption (A2), the bundle $\cE$ is also uniformly contracted on a neighborhood of ${\rm Closure}(V)$;
hence, on the set $\bigcup_{s\in [0,\varepsilon]}\varphi_s(V)$, for some $\varepsilon>0$ small.
By Remark~\ref{r.identification}.(b), one can rescale the time so that $\varepsilon=1$ and assume:
\begin{enumerate}
\item[(A2')] $\cE$ is uniformly contracted on $\bigcup_{s\in [0,1]}\varphi_s(V)$, where $V$ is an open set containing $K\setminus U$.
\end{enumerate}

Recall $\beta_0>0$ in the definition of identifications (Definition~\ref{Def:identification}).

As introduced in Subsection~\ref{ss.assumptions} we denote by $\tau_0\in \NN$ and $\lambda>1$ the constants associated to the $2$-domination $\cE\oplus \cF$.
Proposition~\ref{l.summability} associates to the set $W:=V$
the constants $T_\cF\geq 0$ defining Pliss points for $\cF$ and $C_{\cE},\lambda_{\cE}$
defining the hyperbolicity for $\cE$. We also choose arbitrarily $\lambda_{\cF}\in (1,\lambda)$.

Sections~\ref{s.plaque}
gives some constants $\alpha_\cE$ and $\alpha_\cF$ controlling the size of center-stable and center-unstable plaques. Proposition~\ref{p.unstable} gives $\alpha>0$ and $\eta>0$ such that for any point $u\in K^-_\eta$
which is $(T_\cF,\lambda_{\cF})$-Pliss for $\cF$, the backward iterates $P_{-t}(\cW_\alpha^{cu}(u))$ have diameter smaller than
$\alpha_\cF \lambda_{\cF}^{-t/2}$. 

We also consider small constants $\beta_\cF,r,\delta_0,\alpha'>0$ which will be chosen in this order during this section:
they control distances inside the spaces $\cN_x$, $K$, or $\cW^{cs}_x$.

\subsection{Topological stability and $\delta$-intervals}

\subsubsection{Dynamics of $\delta$-intervals}
We introduce a crucial notion for this section.

\begin{Definition}\label{Def:deltainterval} Consider $\delta\in (0,\delta_0]$ and $x\in K$.
A curve $I\subset {\cW}^{cs}_x$ (not reduced to a single point),
is called a \emph{$\delta$-interval}
if $0_x\in I$ and for any $t\ge 0$, one has
$$|P_{-t}(I)|\le\delta.$$
\end{Definition}

One example of $\delta$-interval is given by a periodic point $z\in K$
together with a non-trivial interval in $\cW^{cs}(z)$ that is periodic for $(P_t)$
and contains $0_z$.

\begin{Definition}
A $\delta$-interval $I$ is \emph{periodic}
if it coincides with $P_{-T}(I)$ for some $T>0$.

\noindent We say that a $\delta$-interval $I$ at $x$
is \emph{contained in the unstable set of some
periodic $\delta$-interval} if:
\begin{enumerate}
\item the $\alpha$-limit set $\alpha(x)\subset K$ of $x$ is the orbit of a periodic point $y$,
\item $y$ admits a periodic $\delta$-interval $\widehat I_y$,
\item the limit set of $P_{-t}(I)$ is the orbit of some (maybe trivial) interval $I_y\subset \widehat I_y$.
\end{enumerate}
\end{Definition}

The next property will be proved in Section~\ref{ss.Lyapunov}.
\begin{Lemma}\label{l.periodic}
{There is $\delta_0>0$ such that for any $\delta\in(0,\delta_0]$,}
for any periodic $\delta$-interval $I\subset \cN_q$, there exists $\chi>0$ with the following property.

Let $z$ be close to $q$, let $L\subset \cN_z$ be an arc which is close to $I$ in the Hausdorff topology
and contains $0_z$ and let $T>0$ such that $|P_{-t}(L)|\leq \delta$ for any $t\in [0,T]$.
Then $|P_{-T}(L)|>\chi$.
\end{Lemma}
Proposition~\ref{Prop:dynamicsofinterval} below describes dynamics of $\delta$-intervals. It is an analogue to ~\cite[Theorem 3.2]{PS2}.

\begin{Proposition}\label{Prop:dynamicsofinterval}
{There is $\delta_0>0$ such that if there is a $\delta$-interval $I\subset \cW^{cs}_x$ for $\delta\in(0,\delta_0]$}
then
\begin{itemize}
\item[--] either $K$ contains a normally expanded irrational torus,
\item[--] or $I$ is contained in the unstable set of some periodic $\delta$-interval.
\end{itemize}
\end{Proposition}

\begin{Remark}
In the first case one can even show that $\alpha(x)$ is a normally expanded irrational torus.
We will not use it.
\end{Remark}
\smallskip

\noindent
{\it Strategy of the proof of Proposition~\ref{Prop:dynamicsofinterval}.}
The next five subsections are devoted to the proof:
\begin{itemize}
\item[--] One introduces a \emph{limit} $\delta$-interval $I_\infty$
from the backward orbit of $I$ (Section~\ref{ss.limit}).
\item[--] $I_\infty$ has returns close to itself
(Section~\ref{ss.return-I-infty}). Under some ``non-shifting" condition, one gets a periodic $\delta$-interval
(Section~\ref{ss.criterion-periodic}) and the last case of the proposition holds.
\item[--] If the ``non-shifting" condition does not hold, there exists a normally expanded irrational torus
which attracts $x$, $I$ and $I_\infty$ by backward iterations
(Sections~\ref{ss.aperiodic} and~\ref{ss.topological}).
\end{itemize}
The conclusion of the proof is given in Section~\ref{ss.Lyapunov}.

\subsubsection{Topological stability}\label{sss.lyapunov-stable}
Before proving Lemma~\ref{l.periodic} and Proposition~\ref{Prop:dynamicsofinterval},
we derive a consequence.

\begin{Proposition}\label{Pro:lyapunovstablity}
If there is no normally expanded irrational torus,
then $\cE$ is \emph{topologically stable}:
there is $\varepsilon_0>0$ and for any $\varepsilon_1\in(0,\varepsilon_0)$, there is $\varepsilon_2>0$ such that 
$$\forall x\in K~\textrm{and}~\forall t>0,~~~P_t({\cW}^{cs}_{\varepsilon_2}(x))\subset {\cW}^{cs}_{\varepsilon_1}(\varphi_t(x)).$$
\end{Proposition}
\begin{proof}
By Remark~\ref{r.plaque-invariance} it is enough to check that 
$|P_t({\cW}^{cs}_{\varepsilon_2}(x))|$ is bounded by $\varepsilon_1$.
{One can choose $\delta_0$ small so that Proposition~\ref{Prop:dynamicsofinterval} holds.} We argue by contradiction. If the topological stability does not hold, then there exist {$\delta\in(0,\delta_0]$}
and a sequence $(x_n)$ in $K$
and, for each $n$, an interval $I_n\subset {\cW}^{cs}(x_n)$ containing $0$
and a time $T_n>0$ such that:
\begin{itemize}
\item[--] $|I_n|\to 0$ as $n\to +\infty$.
\item[--] $|P_{T_n}(I_n)|=\delta$
and $|P_{t}(I_n)|<\delta$ for all $0<t<T_n$.
\end{itemize}
Taking a subsequence, one can assume that $(\varphi_{T_n}(x_n))$
converges to a point $x\in K$ and $(P_{T_n}(I_n))$
to an interval $I$. We have $|I|=\delta$ and
$|P_{-t}(I)|\le\delta$ for all $t>0$, so that
$I$ is an $\delta$-interval.

Since there is no normally expanded irrational torus, the second case of Proposition~\ref{Prop:dynamicsofinterval} is satisfied: $I$ is contained in the unstable set of periodic $\delta$-interval $I_0$. Set $L_n=P_{T_n}(I_n)$, $n$ large. There is $S>0$ such that for $n$ large enough, one has that $P_{-S}(L_n)$ is close to $I_0$.
Lemma~\ref{l.periodic} implies that $P_{-T_n+S}(P_{-S}(L_n))$
has length uniformly bounded away from $0$.
This contradicts the fact that the length of $I_n=P_{-T_n}(L_n)$
goes to $0$ when $n\to \infty$. 
\end{proof}
\subsection{Limit $\delta$-interval $I_\infty$}\label{ss.limit}
One can obtain infinitely many $\delta$-intervals
with length uniformly bounded away from zero at points of the backward orbit
of a $\delta$-interval.
The goal of this section is to prove this property.

\begin{Proposition}\label{p.limit}
If $\delta>0$ is small enough,
for any $x\in K$ and any $\delta$-interval $I$ at $x$,
there exists an increasing sequence $(n_k)$ in $\NN$
and $\delta$-intervals $\widehat I_k$ at $\varphi_{-n_k}(x)$ such that:
\begin{itemize}
\item[--] $P_{-n_k}(I)$ and $P_{n_\ell-n_k}(\widehat I_{\ell})$
are contained in $\widehat I_{k}$ for any $\ell\leq k$,
\item[--] $\varphi_{-n_k}(x)$ is $T_\cF$-Pliss
and belongs to $K\setminus V$,  for any $k\geq 0$,
\item[--] $(\widehat I_k)$ converges to some $\delta$-interval $I_\infty$ at some point $x_\infty\in K\setminus V$.
\end{itemize}
\end{Proposition}

\subsubsection{Existence of hyperbolic returns}
\begin{Lemma}\label{Lem:hyperbolicreturns}
If $\delta>0$ is small enough,
any point $x\in K$ which admits a $\delta$-interval
has infinitely many backward iterates $\varphi_{-n}(x)$, $n\in \NN$, in $K\setminus V$ that are $T_\cF$-Pliss.
\end{Lemma}

\begin{proof}
Since the backward iterates $P_{-n}(I)$, $n\in \NN$, of a $\delta$-interval $I$
are still $\delta$-intervals, it is enough to show that
any point $x$ has at least one backward iterate by $\varphi_{1}$ in $K\setminus V$ that is $T_\cF$-Pliss. The proof is done by contradiction.

{Let $\chi>0$ such that
$1+\chi<\min(\lambda,\lambda_{\cE}).$}
If $\delta_1>0$ is small enough, then for any $\delta$-interval $I$ with $\delta<\delta_1$,
at any point $x$, we have
$$\|DP_{-t}{|\cE(x)}\|\leq (1+\chi)^t\;\frac{|P_{-t}(I)|}{|I|}\leq\frac{(1+\chi)^t\;\delta}{|I|},~~~\forall t\ge0.$$
With the domination estimate~\eqref{e.domination} of Section~\ref{ss.assumptions}, one gets for any $k\geq 0$:
$$\prod_{j=0}^{k-1}\|DP_{-\tau_0}{|\cF(\varphi_{-j\tau_0}(x))}\| \leq \frac{(1+\chi)^{k\tau_0}\;\lambda^{-2k\tau_0}\;\delta}{|I|}.$$
Using $(1+\chi)<\lambda$ and
Pliss lemma (see~\cite[Lemma 11.8]{Man87}), there exists an arbitrarily large integer $i$ such that for any $k\geq 0$ one has:
$$\prod_{j=0}^{k-1}\|DP_{-\tau_0}{|\cF({\varphi_{-(j+i)\tau_0}(x)})}\| \leq \lambda^{-k\tau_0}.$$
This proves that $x$ has arbitrarily large $0$-Pliss
backward iterates by $\varphi_{\tau_0}$.

Let us fix any of these $0$-Pliss backward iterates $\varphi_{-k\tau_0}(x)$.
By contradiction, we assume that there is no iterate $\varphi_{-n}(x)$ in $K\setminus V$ that is
$T_\cF$-Pliss
with $1\leq n\leq k\tau_0$. We can thus apply Proposition~\ref{l.summability} with $W=V$; for any such $n$, one gets
$$\|DP_{n}{|\cE(\varphi_{-k\tau_0}(x))}\|\leq C_{\cE}\lambda_{\cE}^{-n}.$$
As before, this implies that
$$|I|\leq |P_{-k\tau_0}(I)|C_{\cE}\lambda_{\cE}^{-k\tau_0}(1+\chi)^{k\tau_0}\leq \delta\; C_{\cE}\; (1+\chi)^{k\tau_0}\; \lambda_{\cE}^{-k\tau_0}.$$
Since $(1+\chi)/\lambda_{\cE}<1$ and $k$ is arbitrarily large, one gets $|I|=0$ which is a contradiction.
\end{proof}

\subsubsection{Rectangles associated to $\delta$-intervals of Pliss points}

When $\delta>0$ is smaller than $\eta$, any point $u$ in a $\delta$-interval $I$
belongs to $K^-_\eta$.
Consequently, it has a space $\cF(u)$ and a center-unstable plaque $\cW^{cu}(u)$.
Moreover by Lemma~\ref{l.cont3}, if its base point $x\in K$ is $T_\cF$-Pliss,
then $u$ is $(T_\cF,\lambda_{\cF})$-Pliss.

We will need to build a rectangle $R(I)$ foliated by unstable plaques for each $\delta$-interval $I$
above a Pliss point $x\in K$. As before $d$ denotes the dimension of the fibers $\cN_x$.

\begin{Proposition}\label{p.rectangle}
Fix $\beta_\cF\in (0,\beta_0/4)$.
There exist $\alpha_{min},\delta_0,C_{R}>0$  such that
for any $\delta\in (0,\delta_0]$, any $T_\cF$-Pliss point $x\in K\setminus V$ and any $\delta$-interval $I$ at $x$,
we associate a rectangle $R(I)\subset \cN_x$ which is the image of $[0,1]\times B_{d-1}(0,1)$
by a homeomorphism $\psi$ such that:
\begin{enumerate}
\item $\psi\colon [0,1]\times \{0\}\to \cN_x$ is a $C^1$-parametrization of $I$,
\item each $u\in R(I)$ belongs to a (unique) leaf
$\psi(\{z\}\times B_{d-1}(0,1))$ denoted by $W^u_{R(I)}(u)$;
it is contained in $\cW^{cu}(u)$
(and $\cW^{cu}(\psi(z,0))$), and it contains a disc with radius $\alpha_{min}$,

\item ${\rm Volume}(R(I))\geq C_{R}.|I|$ inside the fiber $\cN_x$,
\item for any $t>0$, ${\rm Diam}(P_{-t}(W^u_{R(I)}(u)))<\beta_\cF\lambda_{\cF}^{-t/2}$; hence $P_{-t}(R(I))\subset B(0_{\varphi_{-t}(x)},2\beta_\cF)\subset \cN_{\varphi_{-t}(x)}$.
\end{enumerate}
Moreover, if $x,x'$ are $T_\cF$-Pliss points with $\delta$-intervals $I,I'$ and if $t,t'>0$
satisfy:
\begin{itemize}
\item[--] $\varphi_{-t}(x),\varphi_{-t'}(x')$ belong to $K\setminus V$ and are $r_0$-close,
\item[--] $P_{-t}(R(I))$ and the projection of $P_{-t'}(R(I'))$ by $\pi_{\varphi_{-t}(x)}$ intersect,
\end{itemize}
then, the foliations of $P_{-t}(R(I))$ and $\pi_{\varphi_{-t}(x)}\circ P_{-t'}(R(I'))$
coincide:
if $\pi_{\varphi_{-t}(x)}P_{-t'}(W^u_{R(I')}(u'))$ and $P_{-t}(W^u_{R(I)}(u))$ intersect, they are contained
in a same $C^1$-disc tangent to $\cC^\cF$.
\end{Proposition}
\begin{Remark-numbered}\label{r.intersection}
Assume {that} $\beta_\cF$ and $r>0$ {are} small enough.
By the Global invariance,
under the assumptions of the last part of Proposition~\ref{p.rectangle},
and assuming $d(\varphi_{-t}(x),\varphi_{-t'}(x'))<r$,
there exists $\theta\in \lip$ such that $\theta(0)=0$ and 
\begin{itemize}
\item[--] the distance $d(\varphi_{-\theta(s)-t'}(x'),\varphi_{-s-t}(x))$ for $s>0$
remains bounded (and arbitrarily small if $\beta_\cF>0$ and
$r$ have been chosen small enough),
\item[--]  $P_{-s}\circ \pi_{\varphi_{-t}(x)}=\pi_{\varphi_{-s-t}(x)}\circ P_{-\theta(s)}$ on $P_{-t'}(R(I'))$
when $\varphi_{-s-t}(x)$ and $\varphi_{-\theta(s)-t'}(x')$ belong to $U$.
\end{itemize}
\end{Remark-numbered}

\begin{proof}[Proof of Proposition~\ref{p.rectangle}]
By Proposition~\ref{p.distortion}, one associates to $T_{\cF},\lambda_{\cF}$
some constants $\Delta,\beta$. One can reduce $\beta_\cF$ to ensure $2\beta_\cF<\beta$.

By Proposition~\ref{p.unstable} one associates to $T_\cF$ and $\beta_\cF$,
some quantities $\alpha',\eta'$. We will assume that $\delta$ is smaller than $\eta'$
so that for any $u$ in the $\delta$-interval $I$ and any $t\geq 0$ we have
a backward exponential contraction of the unstable plaque of size $\alpha'$:
$$\forall t\geq 0,\quad \diam(P_{-t}(\cW^{cu}_{\alpha'}(u)))\leq \beta_\cF\lambda_{\cF}^{-t/2}.$$
Thus, for $t$ large enough, the third item of Proposition~\ref{p.distortion} is satisfied by domination and backward contraction.

One parametrizes the curve $I$ by $[0,1]$.
From the plaque-family Theorem~\ref{t.plaque}, there exists a continuous
map $\psi\colon [0,1]\times \RR^{d-1}\to \cN_x$
such that $\psi([0,1]\times \{0\})=I$.
Up to rescaling $\RR^{d-1}$, the image $\psi(\{s\}\times B_{d-1}(0,1))$ is small,
hence is contained in $\cW^{cu}_{\alpha'}(\psi(s,0))$ for each $s\in [0,1]$.

Note that two plaques $\psi(\{s\}\times B_{d-1}(0,1)),\psi(\{s'\}\times B_{d-1}(0,1))$, for $s\neq s'$,
can not be contained in a same $C^1$-disc tangent to $\cC^\cF$ since they intersect
the transverse curve $I$ at two different points.
By coherence, they are disjoint: indeed since the backward orbit of
$x$ has arbitrarily large backward iterates in $K\setminus V$
(see Lemma~\ref{Lem:hyperbolicreturns}), Proposition~\ref{p.coherence} applies.

Hence $\psi$
is injective on $[0,1]\times B_{d-1}(0,1)$ and is a homeomorphism on its image $R(I)$ by the invariance of domain theorem.
In particular $R(I)$ satisfies Definition~\ref{d.rectangle} and is a rectangle.
The coherence again implies that for any $u\in \psi(\{s\}\times B_{d-1}(0,1))$,
the plaque $\cW^{cu}(u)$ contains $\psi(\{s\}\times B_{d-1}(0,1))$.
By compactness, there exists $\alpha_{min}>0$ (which does not depend on $x$ and $I$) such that
$\psi(\{s\}\times B_{d-1}(0,1))$ contains $\cW^{cu}_{\alpha_{min}}(\psi(s,0))$ for any $s\in [0,1]$.

The rectangle $R(I)$ has distortion bounded by $\Delta$ (from Proposition~\ref{p.distortion}),
hence one bounds ${\rm Volume}(R(I))$ from below by using Fubini's theorem and integrating along curves
tangent to $\cC^\cE$. 
This gives the four items of the lemma.

The last part is a direct consequence of the coherence (Proposition~\ref{p.coherence}).
\end{proof}

\subsubsection{Maximal $\delta$-intervals}
In the setting of Proposition~\ref{p.limit},
let us consider all the integers $n_0=0< n_1<n_2<\cdots<n_k<\cdots$ such that
the backward iterate $x_k:=\varphi_{-n_k}(x)$ belongs to $K\setminus V$
and is $T_\cF$-Pliss.
We introduce inductively some maximal $\delta$-intervals $\widehat{I}_{k}$ at these iterates
such that $I\subset \widehat I_0$ and $P_{n_k-n_{k+1}}(\widehat I_k)\subset \widehat I_{k+1}$.
We denote $R_k=R(\widehat{I}_{k})$ as in Proposition~\ref{p.rectangle}.

Lemma~\ref{l.summability-hyperbolicity} associates to $C_{\cE},\lambda_{\cE}$ some constants
$C'_{\cE},\delta_\cE$. Let us assume $\delta<\delta_\cE$.
From the definition of the sequence $(n_k)$, Proposition~\ref{l.summability}
implies that $(\varphi_{-n_{k+1}}(x),\varphi_{-n_{k}}(x))$ is $(C_{\cE},\lambda_{\cE})$-hyperbolic for $\cE$ and for each $k$.
Lemma~\ref{l.summability-hyperbolicity} then gives
\begin{equation}\label{e.bounded}
\sum_{m=n_k}^{n_{k+1}}|P_{n_k-m}(\widehat I_{k})|
\leq C'_{\cE}\; |\widehat I_{k+1}|.
\end{equation}

\subsubsection{Non-disjointness}

\begin{Lemma}[Existence of intersections]\label{Sub:disjointcase}
If $\delta_0>0$ is small enough, {then} for any $\delta\in (0,\delta_0]$, for any $r>0$
there exist $k<\ell$ arbitrarily large such that
$d(x_k,x_\ell)<r$ and the interior of $\pi_{x_k}({R}_\ell)$ and the interior of $R_k$ intersect.
\end{Lemma}
\begin{proof}
Since $(\pi_{x,y})$ is a continuous family of diffeomorphisms on the open set
$U$ containing $K\setminus V$, the following property holds provided $r$ is small enough:

For any $x,y\in K\setminus V$ with $d(x,y)<r$, the projection $\pi_{y}\colon \cN_x\to\cN_y$ satisfies:
$$\pi_{y}(B(0,\beta_0/2))\subset B(0,\beta_0),$$
$${\rm det}(D\pi_{x,y})(u)\leq 2 \text{ for any } u\in B(0,\beta_0/2).$$
By compactness, one can find a finite set $Z\subset U$
such that any $x\in K\setminus V$ satisfies $d(x,z)<r/2$ for some $z\in Z$.
For each point $x_k$ we associate some $z_k\in Z$ such that
$d(x_k,z_k)<r/2$. Since $R_k\subset B(0,2\beta_\cF)\subset \cN_{x_k}$ with $\beta_\cF<\beta_0/4$ and from
Proposition~\ref{p.rectangle}, we have
$$\pi_{z_{k}}(R_k)\subset B(0,\beta_0)\subset \cN_{z_k},$$
$${\rm Volume}(\Interior(\pi_{z_{k}}(R_k)))\geq \frac 1 2 {\rm Volume}(\Interior(R_k)) \geq \frac {C_R} {2} |\widehat I_k|.$$

Let us assume by contradiction that the statement of the lemma does not hold.
One deduces that there exists $s\geq 0$ such that for any $z\in Z$ and any $k,\ell\geq s$
such that $z_k=z_\ell=z$ we have
$$\pi_{z}(\Interior(R_k))\cap \pi_{z}(\Interior( R_\ell))=\emptyset.$$
In particular if $C_{Vol}$ denotes the supremum of $\text{Volume}(B(0_x,\beta_0))$ over $x\in K$,
$$\sum_{k=1}^\infty|\widehat{I}_k|\le 2C_R^{-1}\;C_{Vol}\text{Card}(Z).$$

With~\eqref{e.bounded} we get for any $k$:
\begin{equation}\label{e.sum}
\sum_{m=0}^{+\infty}|P_{-m}(\widehat I_k)|\leq
C'_{\cE}\sum_{\ell=k}^\infty|\widehat{I}_\ell|\leq C_{Sum}:=2C_R^{-1}\;C_{Vol}\text{Card}(Z)\; C'_{\cE}.
\end{equation}
By Denjoy-Schwartz Lemma~\ref{Lem:schwartz} one gets $\eta_S>0$ and for $k$ large one can
introduce an interval $J\subset \cW^{cs}(x_k)$ containing $\widehat I_k$ and of length
equal to $(1+\eta_S)|\widehat I_k|$. One gets
$$|P_{-m}(J)|\leq 2|P_{-m}(\widehat I_k)|,~~~\forall m\ge 0.$$
From~\eqref{e.sum}, $\sup_{m\geq 0} |P_{-m}(\widehat I_k)|$ is arbitrarily small for $k$ large,
hence $|P_{-t}(J)|$ is smaller than $\delta$ for any $t>0$. This proves
that $J$ is a $\delta$-interval, contradicting the maximality of $\widehat I_k$.
\end{proof}

\subsubsection{Non-shrinking property}
\begin{Lemma}
If $\delta_0$ is small enough, 
the length $|\widehat I_k|$ does not go to zero as $k\to \infty$.
\end{Lemma}
\begin{proof}
We first introduce an integer $N\geq 1$ such that $\beta_\cF\lambda_{\cF}^{-N/2}$ is much smaller
than $\alpha_{min}$.

We argue by contradiction. Assume that $|\widehat I_k|$ is arbitrarily small as $k$ is large.
From~\eqref{e.bounded} we deduce that for any
$\delta'\in (0,\delta)$, if $k$ is large enough, then $\widehat I_k$
is a $\delta'$-interval.

By Lemma~\ref{Sub:disjointcase},
there exist $k\neq \ell$ large such that $x_k,x_\ell$ are arbitrarily close
and we have that $\pi_{x_k}(\Interior({R}_\ell))\cap \operatorname{Interior}({R}_k)\neq\emptyset$.
By Remark~\ref{r.intersection}, the unstable foliations of
$\pi_{x_k}({R}_\ell)$ and ${R}_k$ coincide on the intersection,
hence one of the following cases occurs (see Figure~\ref{f.shrinking}).
\begin{enumerate}

\item There exists an endpoint $u$ of $\widehat I_k$
such that $W^{u}_{R_k}(u)$ intersects
$\pi_{x_k}(\widehat I_\ell)$ at a point which is not endpoint.

\item The endpoints of $\widehat I_k$ and $\pi_{x_k}(\widehat I_\ell)$
have the same unstable manifolds.

\end{enumerate}

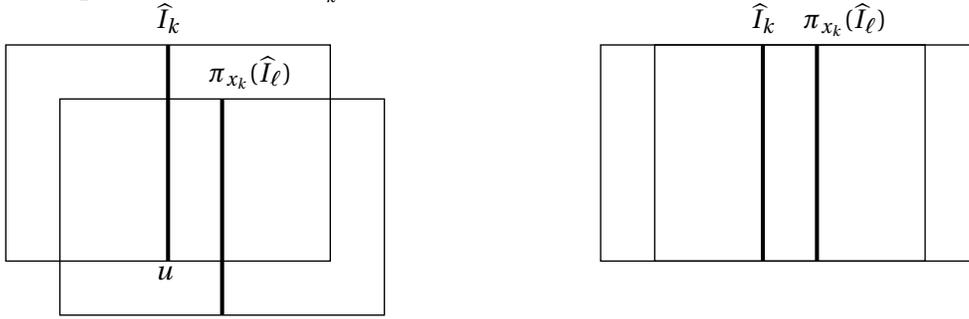
\begin{figure}[ht]
\begin{center}
\includegraphics[width=13cm]{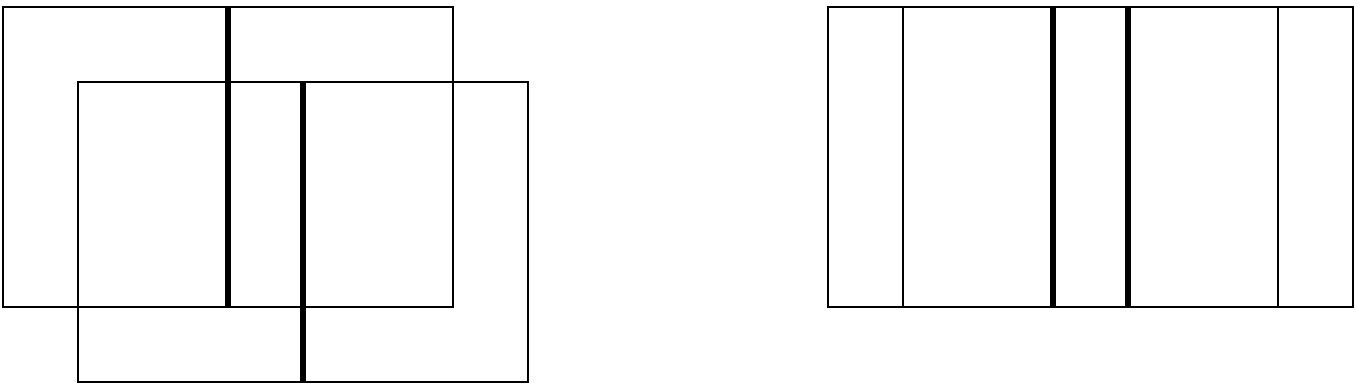}
\begin{picture}(0,0)
\put(-315,110){\small$\widehat I_k$}
\put(-315,15){\small$u$}
\put(-90,110){\small$\widehat I_k$}
\put(-295,90){\small$\pi_{x_k}(\widehat I_\ell)$}
\put(-70,110){\small$\pi_{x_k}(\widehat I_\ell)$}
\end{picture}
\end{center}
\caption{Non-shrinking\label{f.shrinking}}
\end{figure}

In the first case, by Remark~\ref{r.intersection}, there exists
a homeomorphism $\theta$ of $[0,+\infty)$ such that
$\varphi_{-t}(x_k)$ and $\varphi_{-\theta(t)}(x_\ell)$ are close for any $t>0$,
and $P_{-t}(\pi_{x_k}(R( I_\ell)))$ remains in a neighborhood
of $\varphi_{-t}(x_k)$ which is arbitrarily small
if $\delta'$ and $d(x_k,x_\ell)$ are small enough.
The rectangle $\pi_{x_k}(R(\widehat I_\ell))$
intersects $\cW^{cs}(x_k)$ along an interval $J$,
which meets $\widehat I_k$.
This proves that the union of
$J$ with $\widehat I_k$ is a $\delta$-interval
and contradicts the maximality since $J$ is not contained
in $\widehat I_k$ in this first case.

In the second case, without loss of generality, we assume $n_\ell>n_k$ and set
$T:=n_\ell-n_k$.
We introduce the map $\widetilde P_{-T}:=\pi_{x_k}\circ P_{n_k-n_\ell}$.
Since $r$ has been chosen small enough and since the endpoints of $\widehat I_k$ and $\pi_{x_k}(\widehat I_\ell)$
have the same unstable manifolds, the iterates $\widetilde P_{-T}^i(\widehat I_k)$, $i\geq 0$,
are all contained in $R_k$. 
Hence by the Global invariance, there exists
a sequence of times $0<t_1<t_2<\dots$ going to $+\infty$
such that
\begin{itemize}
\item[--] $\varphi_{-t_1}(x_k)=x_\ell$,
\item[--] $(\varphi_{-t}(x_k))_{t_i\leq t\leq t_{i+1}}$ shadows $(\varphi_{-t}(x_k))_{0\leq t\leq t_{1}}$,
\item[--] $\varphi_{-t_i}(x_k)$ is close to $x_k$
and projects by $\pi_{x_k}$ in $R_k$.
\end{itemize}
Note that the differences $t_{i+1}-t_i$ are uniformly bounded in $i$ by some constant
$T_0$. Since $\delta'$ can be chosen arbitrarily small (provided $k$ is large),
for $t=t_i$ arbitrarily large,
the interval $J=\pi_{\varphi_{-t_i}(x_k)}(R_k)\cap \cW^{cs}(\varphi_{-t_i}(x_k))$
contains $0$ and is a $\delta/2$-interval.
One can choose $t_i$ and a backward iterate $x_j$
such that $t_i\leq n_j\leq t_i+T_0$.
Since $P_{t_i-n_j}(J)$ is a $\delta/2$-interval and
$\widehat I_j$ is a $\delta'$-interval,
$\widehat I_j\cup P_{t_i-n_j}(J)$ is a $\delta$-interval.
As $j$ can be chosen arbitrarily large, $|\widehat I_j|$
is arbitrarily small, whereas $|P_{t_i-n_j}(J)|$ is uniformly bounded away from zero
(since $n_j-t_i$ is bounded). Consequently $\widehat I_j\cup P_{t_i-n_j}(J)$
is strictly larger than $\widehat I_j$, contradicting the maximality.
\end{proof}

\subsubsection{Existence of limit intervals}
The Proposition~\ref{p.limit} now follows easily from the previous lemmas,
up to extracting a subsequence from the sequence of Pliss times $(n_k)$.
\qed

\subsection{Returns of $\delta$-intervals}\label{ss.return-delta}

\subsubsection{Definition of returns and of shifting returns}

We now introduce the times which will allow to induce the dynamics near a $\delta$-interval.
\begin{Definition}\label{d.return}
Let $x\in K\setminus V$ be a $T_\cF$-Pliss point
and $I$ be a $\delta$-interval at $x$.
A time $t>0$ is a \emph{return} of $I$ if
\begin{itemize}
\item[--] $\varphi_{-t}(x)$ and $x$ are $r_0$-close, and $\Interior(R(I))\cap \Interior(\pi_x\circ P_{-t}(R(I)))\neq \emptyset$,
\item[--] for any $z,z'\in I$ such that
$\pi_x\circ P_{-t}(W^{u}_{R(I)}(z))\cap W^{u}_{R(I)}(z')\neq \emptyset$, we have
$$\pi_x\circ P_{-t}(W^{u}_{R(I)}(z))\subset W^{u}_{R(I)}(z').$$
\end{itemize}
We then denote by $\widetilde P_{-t}$ the map $\pi_x\circ P_{-t}\colon R(I)\to \cN_x$.
\smallskip

\noindent
A sequence of returns $(t_n)$ is \emph{deep} if
$t_n\to +\infty$ and if one can find one sequence $(x_n)$ in $K$ with
$\pi_x(x_n)\in R(I)$ such that
{ $\widetilde P_{-t_n}\circ \pi_x(x_n)\to 0_x$ as $n\to +\infty$}.

\end{Definition}

\begin{Remark-numbered}\label{r.deep}
When $(t_n)$ is a sequence of deep returns,
$\widetilde P_{-t_n}\circ \pi_x(R(I))$ gets arbitrarily close to $\cW^{cs}(x)$ as $n\to +\infty$.
\emph{ Indeed, since $t_n$ is large, $\widetilde P_{-t_n}\circ \pi_x(R(I))$ is thin and contained in a small neighborhood of
$\pi_x(\cW^{cs}(\varphi_{-t'_n}(x_n)))$, where $\widetilde P_{-t_n}\circ \pi_x(x_n)=\pi_x(\varphi_{-t'_n}(x_n))$.
Moreover as $\widetilde P_{-t_n}\circ \pi_x(x_n)\to 0_x$, the plaque $\pi_x(\cW^{cs}(\varphi_{-t'_n}(x_n)))$
gets close to $\cW^{cs}(x)$.}
\end{Remark-numbered}

\begin{Lemma}\label{l.return-existence}
If $\delta_0$ is small enough, 
a $T_\cF$-Pliss point $x\in K\setminus V$ with a $\delta$-interval $I$ {for $\delta\in (0,\delta_0]$},
some $t>2\log(\beta_\cF/\delta_0)/\log(\lambda_{\cF})$ satisfying $d(x,\varphi_{-t}(x))<r_0$ and  $z\in I$ satisfying $\pi_{x}\circ P_{-t}(z)\in \Interior(R(I))$
and $d(\pi_{x}\circ P_{-t}(z),I)<\delta_0$, {then  the time $t$ is a return.} 
\end{Lemma}
\begin{proof}
The leaves $W^u_{R(I)}(z)$ are tangent to $\cC^\cF$ and have uniform size $\alpha_{min}$.
The interval $I$ is tangent to $\cC^\cE$ and has size less than $\delta_0$, chosen much smaller than $\alpha_{min}$.
Assuming that $t>0$ is large enough, the images $P_{-t}(W^u_{R(I)}(z))$ have diameter smaller than $\beta_\cF\lambda_{\cF}^{t/2}<\delta_0$,
so $P_{-t}(R(I))$ has diameter smaller than $2\delta_0$. If $d(\pi_{x}\circ P_{-t}(z),I)<\delta_0$,
the images $P_{-t}(W^u_{R(I)}(z))$ can not intersect the boundary of the leaves $W^u_{R(I)}(z')$.
From the last part of Proposition~\ref{p.rectangle},
we get that $P_{-t}(W^u_{R(I)}(z))$ is disjoint or contained in $W^u_{R(I)}(z')$ for each $z,z'\in I$.
\end{proof}
\medskip

To each return, one associates a one-dimensional map $S_{-t}:I\to \cW^{cs}_x$ as follows.
\begin{Proposition}\label{p.one-dimensional}
Assume that $\beta_\cF,r,\delta_0$ are small enough and that $\delta$ is much smaller than $\delta_0$.
Consider a return $t$ of a $\delta$-interval $I\subset \cN_x$ and $d(\varphi_{-t}(x),x)<r$.

Then there exists a $\delta_0$-interval $J\subset \cW^{cs}(x)$ containing $I$
and a continuous injective map $S_{-t}\colon I\to J$ such that
for each $u\in I$, the point $\widetilde P_{-t}(u)$ belongs to $W^{u}_{R(J)}(S_{-t}(u))$.
\end{Proposition}
Notice that from Definition~\ref{d.return}, $S_{-t}(I)$ intersects $I$ along a non-trivial interval.
\begin{proof}
Assuming $\beta_\cF,r$ small enough,
By Remark~\ref{r.intersection},
the backward orbits
of $\varphi_{-t}(x)$ and $x$ stay at an arbitrarily small distance.
The $\delta$-interval $P_{-t}(I)$ projects by $\pi_x$
on a set $X$ whose backward iterates by $P_{-s}$ are contained in
$B(0_{\varphi_{-s}(x)},\delta_0/2)$ (using the Global invariance).
Since $x$ is a $T_\cF$-Pliss point, any $u\in X$ is $(T_\cF,\lambda_{\cF})$-Pliss (see Lemma~\ref{l.cont3})
and has an unstable plaque $\cW^{cu}_{\alpha}(u)$
whose backward iterates under $P_{-t}$ have diameter smaller than $\alpha_\cF\lambda_{\cF}^{-t/2}$.
One can thus project {the arc} $X$ along the plaques
$\cW^{cu}_\alpha(u)$ of points $u\in X$ to a connected set $I'\subset \cW^{cs}(x)$ which intersects $I$.
If $\delta_0$ is small enough,
and $P_{-s}(I\cup I')$ has diameter smaller than $2\diam(P_{-s}(R(I)\cup \widetilde P_{-t}(I)))<\delta_0$.
Thus $J=I\cup I'$ is a $\delta_0$-interval.

By the coherence (Proposition~\ref{p.rectangle}, item 2),
the plaques $\cW^{cu}_\alpha(u)$ of 
$u\in X$ intersect $R(J)$ along a leaf $W^u_{R(J)}(u)$.
Note that each plaque intersect $I'\subset J$ by construction.
Moreover since $d(x, \varphi_{-t}(x))$ is small, 
$X$ does not intersect the boundary of the
leaves $W^u_{R(J)}(u)$. Thus $X\subset R(J)$.

Any point in $X=\widetilde P_{-t}(I)$ can thus be projected to $J$ along the leaves of $R(J)$.
The map $S_{-t}$ is the composition of this projection with $\widetilde P_{-t}$.
\end{proof}

Deepness has been introduced for the following statement.
\begin{Lemma}\label{l.return}
Let $(t_n)$ be a deep sequence of returns of a $\delta$-interval $I$
and let $J$ be a $\delta_0$-interval containing $I$ such that
$S_{-t_n}(I)\subset J$ for each $n$.
Then, there exists $n_0\geq 1$ with the following property.
If $n(1),\dots,n(\ell)$ is a sequence of integers with $n(i)\geq n_0$
and if there exists $u$ in the interior of $J$ satisfying for each $1\leq i\leq \ell$
$$S_{-t_{n(i)}}\circ\dots\circ S_{-t_{n(1)}}(u)\in \operatorname{Interior}(J),$$
then there exists a return $t>0$ such that
$S_{-t}=S_{-t_{n(\ell)}}\circ\dots\circ S_{-t_{n(1)}}$.
\end{Lemma}
The return $t$ will be called \emph{composition}
of the returns $t_{n(0)},\dots,t_{n(\ell)}$.
\begin{proof}
Note first that the image $\widetilde P_{-t}(R(I))$ associated to a large return $t$
has diameter smaller than $2\delta$. So if this set contains a point $\widetilde P_{-t}(u)$
$\delta'$-close to $I$, then $\widetilde P_{-t}(R(I))$ is contained in the $2\delta+\delta'$-neighborhood of $I$.

The lemma is proved by induction.
The composition $S_{-t_{n(\ell-1)}}\circ\dots\circ S_{-t_{n(1)}}$ is associated
to a return $t'>0$. The point $\widetilde P_{-t'}(u)$
belongs to the image $\widetilde P_{-t_{n(\ell-1)}}(R(I))$,
hence (since $t_{n(\ell-1)}$ is large by deepness and Remark~\ref{r.deep}),
is $2\delta$-close to $I$.
By the Local injectivity and Remark~\ref{r.intersection}, there exists an increasing homeomorphism
$\theta$ such that $|\theta(0)|\leq 1/4$ and
$\varphi_{\theta(-s)-t'}(x)$ shadows $\varphi_{-s}(x)$.
Hence for $t=t'+\theta(t_{n(\ell)})$ the point $\varphi_{-t}(x)$ is close to $x$
and $\widetilde P_{-t}=\widetilde P_{-t_{n(\ell)}}\circ \widetilde P_{-t'}$
by the Global invariance. The first item of Definition~\ref{d.return}
is satisfied.

The second one is implied by Proposition~\ref{p.rectangle}:
if $\widetilde P_{-t}(W^u_{R(I)}(z))$ intersects $W^u_{R(I)}(z')$,
then these two discs match. The set $\widetilde P_{-t}(R(I))$ has diameter smaller than $2\delta$;
it contains the point $\widetilde P_{-t_{n(\ell)}}\circ \widetilde P_{-t'}(u)$;
this last point also belongs to $\widetilde P_{-t_{n(\ell)}}(R(I))$ which is included in
the $2\delta$-neighborhood of $I$. Hence $\widetilde P_{-t}(R(I))$ is contained in the $4\delta$-neighborhood of $I$
and can not intersect the boundary of the disc $W^u_{R(I)}(z')$.
This gives $\widetilde P_{-t}(W^u_{R(I)}(z))\subset W^u_{R(I)}(z')$.

This proves that $t$ is a return
such that $\widetilde P_{-t}=\widetilde P_{-t_{n(\ell)}} \circ \widetilde P_{-t'}$:
the one-dimensional
map associated to $t$ coincides with the composition of the one-dimensional maps
of the returns $t_{n(\ell)}$ and $t'$ as required.
\end{proof}

\begin{Definition}
A return $t$ is \emph{shifting} if the one-dimensional map $S_{-t}$ has no fixed point.

\emph{Let us fix an orientation on $\cW^{cs}(x)$.
It is preserved by $S_{-t}$ when $t$ is shifting.}
\smallskip

\noindent
A return \emph{shifts to the right} (resp.
\emph{to the left}) if it is a shifting return and if there exists
$u\in I$ that can be joined to $S_{-t}(u)$ by a positive arc
(resp. negative arc)  of $\cW^{cs}(x)$.
\end{Definition}

\subsubsection{Criterion for the existence of periodic $\delta$-intervals}\label{ss.criterion-periodic}

The following proposition shows that under the setting of Proposition~\ref{p.limit},
if the interval $I$ has a large non-shifting return, then 
$I$ is contained in the unstable set of some periodic $\delta$-interval.

\begin{Proposition}\label{p.periodic-return}
Let $I$ be a $\delta$-interval at a point $x\in K$ and
let $J$ be a $3\delta$-interval at a $T_\cF$-Pliss point $y\in K\setminus V$
having large non-shifting returns.
If $\pi_y\circ P_{-s}(I)$ intersects $R(J)$ for some $s\geq 0$,
then $I$ is contained in the unstable set of some periodic $\delta$-interval.
\end{Proposition}
\begin{proof}
By assumption there exist $t>0$ large and $u\in J$
such that $W^{u}_{R(J)}(u)$ is mapped into itself by
$\widetilde P_{-t}$. This implies that $\widetilde P_{-t}$ has a fixed point $v$ in $R(J)$.

Since $t$ is a large return, assuming the $\beta_\cF$ is small enough, there exists (by the local injectivity)
$t_1\in [t-1/4,t+1/4]$ such that $d(\varphi_{-t_1}(y),y)<r/2$. This allows (up to modifying $t$)
to assume that $d(\varphi_{-t}(y),y)<r/2$.

\begin{Claim}
There is a periodic point 
$q\in K$ that is 
$r_0$ close to $y$ such that $\pi_y(q)$ is a fixed point of $\widetilde P_{-t}$ in $R(J)$.
\end{Claim}

\begin{proof}[Proof of the claim.]
We build inductively 	a homeomorphism $\theta$ of $[0,+\infty)$
such that:
\begin{itemize}
\item[--] for each $k\geq 0$ and each $s\in [0,t]$, we have
$d(\varphi_{-\theta(s)}(y),\varphi_{-\theta(kt+s)}(y))<r_0/2$,
\item[--] $d(\varphi_{-t_k}(y),y)<r$ where $t_k:=\theta(kt)$ and $r$ is the constant in Remark~\ref{r.intersection},
\item[--] we have $\pi_{y}\circ P_{-t_k}=(\widetilde P_{-t})^k$.
\end{itemize}
Since $d(\varphi_{-t}(y),y)<r/2$, one defines $t_1=t$ and $\theta(s)=s$ for $s\in [0,t]$.

Let us then assume that $\theta$ has been built on $[0,kt]$.
Since $P_{-s}(v)\in P_{-s}(R(J))$ is $\beta_\cF$-close to $0_{\varphi_{-s}(y)}$ for any $s\geq 0$
and since $\pi_{\varphi_{-t_k}(y)}(v)=P_{-t_k}(v)$, Remark~\ref{r.intersection} applies and defines the homeomorphism $\theta$
on $[kt,(k+1)t]$ such that
$d(\varphi_{-s}(y),\varphi_{-\theta(kt+s)}(y))<r_0/4$ for $s\in [0,t]$.
By the local injectivity, one can choose $t_{k+1}$ with
$|t_{k+1}-\theta(kt)|\leq 1/4$,
$d(y,\varphi_{t_{k+1}}(y))<r$ and one can
modify $\theta$ near $(k+1)t$
so that $\theta((k+1)t)=t_{k+1}$.

Since $t_1=t$ is large,
and since
$d(\varphi_{-\theta(kt+s)}(y),\varphi_{-\theta((k+1)t+s)}(y))<r_0$,
the No shear property (Proposition~\ref{p.no-shear}) implies inductively $t_k-t_{k-1}\ge 2$ for each $k\ge 1$.
In particular $t_k\to+\infty$.

By the dominated splitting,
the limit set $\Lambda$ of the curves $\pi_{y}\circ P_{-t_k}(J)=\widetilde P_{-t}^k(J)$
is a union of (uniformly Lipschitz) curves in $\cN_{y}$, containing $v$
and tangent to $\cC^\cE(y)$.

One can apply Proposition~\ref{p.uniqueness}
to a periodic sequence of diffeomorphisms
$f_0,\dots,f_{[t]}$ where $f_i$
coincides with $P_{-1}$ on $\cN_{\varphi_{-m}(y)}$
for $0\leq i <t-1$
and $f_{[t]}$ coincides with
$\pi_{y}\circ P_{t-[t]+1}$ on $\cN_{\varphi_{-[t]+1}(y)}$.
We have $\widetilde P_{-t}=f_{[t]}\circ \dots\circ f_0$
and any curve in the limit set $\Lambda$ remain bounded and
tangent to $\cC^\cE(y)$ under iterations by
$(\widetilde P_{-t})^{-1}$. Consequently, they are all contained in a
same Lipschitz arc, invariant by $\widetilde P_{-t}$.
One deduces that
a subsequence of
$(\widetilde P_{-t})^k(0_y)$
converges in $\cN_{y}$
to a fixed point $p$ of $R(J)$.
\medskip

By Lemma~\ref{l.closing0}, there exists a periodic point $q\in K$
such that $\pi_{y}(q)=p\in R(J)$ is the fixed point of $\widetilde P_{-t}$.
\end{proof}

Now we prove that $I$ is contained in the unstable set of some periodic $\delta$-interval. Without loss of generality, we take $s=0$.
Since $\pi_y(q)$ and $\pi_y(x)$ belong to $R(J)$, by choosing $\beta_\cF$ small enough and by the local invariance,
one can assume that the distances $d(y,q)$ and $d(y,x)$ are smaller than any given constant.
In particular, the projection by $\pi_y$ of center-unstable plaques $\cW^{cu}(u)$ of point $u\in \cN_q$
is tangent to the cone $\cC^\cF$.

\begin{Claim}
There exists a periodic point $q_0\in K$ such that $\alpha(x)$ is the orbit of $q_0$.
\end{Claim}
\begin{proof}[Proof of the Claim.]
By the assumptions, there is $u_0\in I$ such that $\pi_y(u_0)$ is contained in $R(J)$. 
Since $\beta_\cF$ and $\delta$ are small,
by the Global invariance, there is $\theta_y\in\lip$ such that $|\theta_y(0)|\leq 1/4$ and
$d(\varphi_{-s}(x),\varphi_{-\theta_y(s)}(y))<r_0/2$ for any $s>0$.

In a similar way, there exists $\theta_q\in\lip$ such that $|\theta_q(0)|\leq 1/4$ and
$d(\varphi_{-t}(y),\varphi_{-\theta_q(t)}(q))$
remains small for any $t>0$.
By defining $\theta=\theta_q\circ \theta_y$, one deduces that
$d(\varphi_{-\theta(t)}(q),\varphi_{-t}(x))$ is small too for any $t>0$.

By {using the} Global invariance twice,
$\|P_{-t}(\pi_{y}(x))\|$ remains small in the fiber $\cN_{\varphi_{-t}(y)}$
and 
$\|P_{-t}(\pi_{q}(x))\|$ remains small in the fiber $\cN_{\varphi_{-t}(q)}$.
In particular $\pi_q(x)$ belongs to $K^-_\eta$ and has
a center unstable plaque $\cW^{cu}(\pi_q(x))$.

Let us first assume that $0_q\in\cW^{cu}(\pi_q(x))$.
Since $\|P_{-t}(\pi_{q}(x))\|$ remains small when $t\to +\infty$,
the local invariance of the plaque families implies that
$0_{\varphi_{-t}(q)}\in P_{-t}(\cW^{cu}(\pi_q(x)))$ for any $t>0$.
Projecting to the fiber of $\varphi_{-\theta_q^{-1}(t)}(y)$,
and using the Global invariance, one deduces that
$P_{-t}(\pi_y(x))$ and $P_{-t}(\pi_y(q))$ 
is connected by a small arc tangent to $\cC^\cF$ for any $t>0$.
By Proposition~\ref{p.uniqueness}, this shows that $\pi_y(x)$ and $\pi_y(q)$ belong to a same leaf $W^u_{R(J)}(u)$ of $R(J)$.
Consequently, $d(P_{-t}(\pi_y(x)),P_{-t}(\pi_y(q)))\to 0$ as $t\to +\infty$.
The Global invariance shows that $d(P_{-t}(\pi_q(x)),0)\to 0$ as $t\to +\infty$.
The Local injectivity then implies that $\varphi_{-t}(x)$ converges to the orbit of $q$.

If $\cW^{cu}(\pi_q(x))$ does not contains $0_q$,
by Proposition~\ref{p.fixed-point},
there exists a point $p\in \cN_q$ which is fixed by $P_{-2T}$ (where $T$ is the period of $q$),
such that $P_{-t}(\pi_{q}(x))$ converges to the orbit of $p$ as $t\to +\infty$.
By the Global invariance,
$P_{-\ell T}(\pi_{q}(x))$ coincides with $\pi_q(\varphi_{-\theta^{-1}(\ell T)}(x))$
for $\ell\in\NN$.
Lemma~\ref{l.closing0} implies that there exists a periodic point $q_0\in K$ such that
$\varphi_{-\theta^{-1}(\ell T)}(x)$ converges to $q_0$ as $\ell\to  +\infty$.
Since $\theta$ is a bi-Lipschitz homeomorphism, one deduces that
the backward orbit of $x$ converges to the orbit of the periodic point $q_0$.
\end{proof}

Up to replacing $x$ and $I$ by large backward iterates, there is $\theta\in \lip$
s.t. $d(\varphi_{-t}(x),\varphi_{-\theta(t)}(q_0))$ is small for any $t>0$
and by the Global invariance, 
the intervals $P_{-t}\circ \pi_{q_0}(I)$
are curves tangent to the cone $\cC^\cE(\varphi_{-t}(q_0))$
which remain small for any $t>0$.
When $t=\ell T'$ where $T'$ is the period of $q_0$,
they converge to a limit set which is a union of curves
tangent to the cone $\cC^\cE(q_0)$ and contain $0_{q_0}$.
By Proposition~\ref{p.uniqueness}, this limit set is an interval
$ I_{q_0}\subset \cW^{cs}(q_0)$ that is fixed by $P_{T'}$.
By the Global invariance,
$I_{q_0}$ is the limit of $\pi_y(P_{-\theta^{-1}(\ell T')}(I))$
as $\ell\to +\infty$.
This implies that the backward orbit of $I$ converges to the orbit
of the $\delta$-interval $I_{q_0}$.
\end{proof}

\subsubsection{Returns of limit intervals}\label{ss.return-I-infty}
We now prove that the interval $I_\infty$ has always returns.
Moreover, in the case $I_{\infty}$ is not contained in a $3\delta$-interval having arbitrarily large non-shifting return
(so that Proposition~\ref{p.periodic-return} can not be applied to some $\widehat I_k$ and an interval $J$ containing $I_{\infty}$),
we also prove that there exist shifting returns for $I_\infty$, both to the right and to the left.

\begin{Proposition}\label{p.shifting-return}
Under the setting of Proposition~\ref{p.limit}, 
\begin{itemize}
\item[--] either $I_\infty$ is contained in a $3\delta$-interval having arbitrarily large non-shifting return,
\item[--] or $I_{\infty}$ has {two deep sequences of shifting returns { one shifting to the left and the other one to the right.}}
\end{itemize}
\end{Proposition}
\begin{proof} We assume that the first case does not hold and one chooses an orientation, hence an order on
$\cW^{cs}_{x_\infty}$. Denote by $a_\infty<b_\infty$ the endpoints of $I_\infty$.
There exists a $\frac 3 2\delta$-interval $L$ in $\cW^{cs}_{x_\infty}$ such that $R(L)$ contains all the
$\pi_{x_\infty}(\widehat I_k)$, $k$ large.
One can thus project $\pi_{x_\infty}(\widehat I_k)$ to $\cW^{cs}_{x_\infty}$ by the unstable holonomy
and denotes $a_k<b_k$ the endpoints of this projection. In the same way one denotes
$c_k\in [a_k,b_k]$ the projection of $\pi_{x_\infty}(x_k)$.
By Lemma~\ref{l.return-existence}, for any $k<\ell$ such that $k$ and $\ell-k$ are large,
there exists a (large) return $t>0$ such that $\widetilde P_{-t}(\pi_{x_\infty}(x_k))=\pi_{x_\infty}(x_\ell)$.
Since all large returns of $L$ are shifting, such iterates $x_k,x_\ell$ of $x$ do not project by $\pi_{x_\infty}$ to a same
unstable manifold.
Hence one can assume without loss of generality that
for each $k<\ell$ we have
$c_k>c_\ell$.

\begin{Claim}
$I_\infty$ admits a deep sequence of returns $t>0$ shifting to the left.
\end{Claim}
\begin{proof}
For each $k<\ell$, there exists
a return $t>0$ of $L$ such that
$\widetilde P_{-t}(\pi_{x_\infty}(x_k))=\pi_{x_\infty}(x_\ell)$.
This return shifts to the left.
We claim that it is a return for $I_\infty$.
Denote $a'_k=S_{-t}(a_k)$.
The return $t$ has been chosen so that
$\pi_{x_\infty}\circ P_{-n_\ell+n_k}(\widehat I_k)= \widetilde P_{-t}\circ\pi_{x_\infty}(\widehat I_k)$.
Since $\widehat I_\ell$ contains $P_{-n_\ell+n_k}(\widehat I_k)$,
this gives $a_\ell\leq a'_k$ and in particular $a_\ell<a_k$.
Repeating this argument, one gets a decreasing subsequence $(a_i)$ containing $a_k$ and $a_\ell$.
It converges to $a_\infty$ so that $a_\infty< a'_k<a_k$.
This implies that both $a_k$ and $a'_k$ belong to $I_\infty$ and that $t>0$ is also a return for $I_\infty$.
Note that when $k$ is large and $\ell$ much larger, the time $t>0$ is large and the intervals
$\pi_{x_\infty}(\widehat I_k), \pi_{x_\infty}(\widehat I_\ell)$ are close to $\cW^{cs}_{x_\infty}$.
In particular the sequence of returns $t>0$ one obtains by this construction is deep.
\end{proof}
\medskip

Let us fix a return $t$ shifting left as given by the previous claim.
We then choose $k\geq 1$ large and $\ell$ much larger and build a return which shifts to the right.
As explained above we have $a_\infty<a_\ell<a_k$. Let us denote by
$\bar a_\infty<\bar a_\ell<\bar a_k$ their images by $S_{-t}$. Since $S_{-t}$ shifts left, for $k$ large
$\bar a_k$ which is close to $\bar a_\infty$ satisfies $\bar a_k<a_\infty$.
See Figure~\ref{f.return}.
Let us denote $t'>0$ a time such that the pieces of orbits $\varphi_{[-t',0]}(x_k)$
and $\varphi_{[-t,0]}(x_\infty)$ remain close (up to reparametrization), so that
$\pi_{x_\infty}\circ \varphi_{-t'}(x_k)=\widetilde P_{-t}\circ \pi_{x_\infty}(x_k)$.

\begin{figure}[ht]
\begin{center}
\includegraphics[width=10cm]{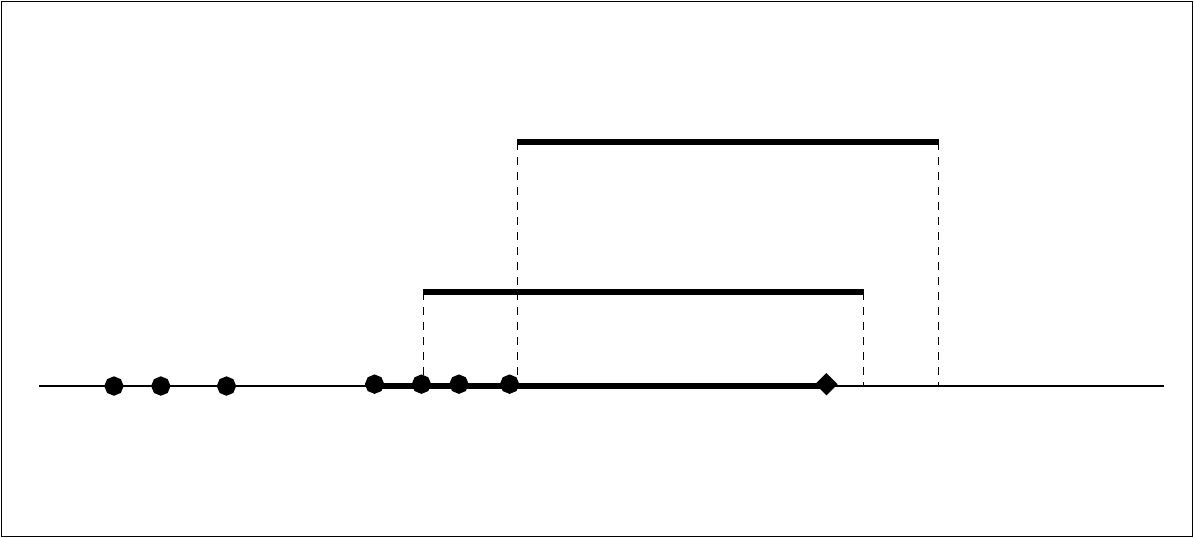}
\begin{picture}(0,0)
\put(-207,25){\scriptsize$a_\infty$}
\put(-192,25){\scriptsize$a_\ell$}
\put(-182,25){\scriptsize$a'_k$}
\put(-170,25){\scriptsize$a_k$}
\put(-95,25){\scriptsize$b_\infty$}
\put(-265,25){\scriptsize$\bar a_\infty$}
\put(-251,25){\scriptsize$\bar a_\ell$}
\put(-238,25){\scriptsize$\bar a_k$}
\put(-150,40){$I_\infty$}
\put(-130,65){$\pi_{x_\infty}(\widehat I_\ell)$}
\put(-110,100){$\pi_{x_\infty}(\widehat I_k)$}
\put(-280,110){$\cN_{x_\infty}$}
\end{picture}
\end{center}
\caption{Shifting to the two sides\label{f.return}}
\end{figure}
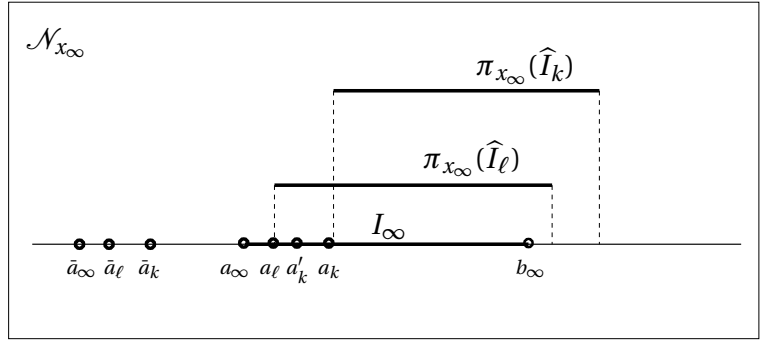

The rectangle associated to the $3\delta$-interval $L\cup S_{-t}(L)$
contains $\pi_{x_\infty}\circ \varphi_{-t'}(x_k)$ and $\pi_{x_\infty}(x_\ell)$.
In particular there exists a return $s>0$ such that
$\widetilde P_{-s}(\pi_{x_\infty}\circ \varphi_{-t'}(x_k))=\pi_{x_\infty}(x_\ell)$.
Since $\ell-k$ is large, $s$ is large and is a shifting return by our assumptions.
We denote $a_\infty'<a_k'$ the images of $\bar a_\infty<\bar a_k$ by $S_{-s}$.
Note that there exists a time $s'>0$ such that
$\pi_{x_\infty}\circ \varphi_{-s'}(\varphi_{-t'}(x_k))=\widetilde P_{-s}\circ\pi_{x_\infty}(P_{-t'}(x_k))=\pi_{x_\infty}(x_\ell)$.
By the local injectivity, one can choose $s'$ such that
$t'+s'=n_\ell-n_k$. In particular $S_{-s}\circ S_{-t}$
coincides with the one-dimensional map $S_{-\widetilde t}$ associated to the return $\widetilde t>0$
sending $\pi_{x_\infty}(x_k)$ to $\pi_{x_\infty}(x_\ell)$.
Note that $t'$ is a return as considered in the proof of the previous claim;
in particular we have proved that $a_\infty< S_{-\widetilde t}(a_k)<a_k$.
Hence $a_\infty< a'_k<a_k$.

We have obtained $\bar a_k<a_\infty<a'_k$, so that $S_{-s}$ shifts to the right
and $a_\infty<a'_\infty$.
On the other hand $a'_\infty<a'_k<a_k<b_\infty$.
So this gives $a'_\infty\in (a_\infty, b_\infty)$ which implies that $S_{-s}$
is a return of $I_\infty$ which shifts to the right as required.
{ Since $\pi_{x_\infty}\circ \varphi_{-t'}(x_k)\in R(I)$ and $\pi_{x_\infty}(x_\ell)\to 0_{x_\infty}$,
the sequence of returns $s$ one may build by this construction is deep.}
\end{proof}

\subsection{Aperiodic $\delta$-intervals}\label{ss.aperiodic}
We introduce the $\delta$-intervals which will produce normally expanded irrational tori.

\begin{Definition}\label{d.aperiodic}
A $\delta$-interval $J$ at $x\in K\setminus V$
is \emph{aperiodic} if there exist returns $t_1,t_2>0$
and intervals $J_1,J_2\subset J$ such that:
\begin{enumerate}
\item[--] $J_1,J_2$ have disjoint interior and $J=J_1\cup J_2$,
\item[--] $\widetilde P_{-t_1}(J_1), \widetilde P_{-t_2}(J_2)$
have disjoint interior and $J=\widetilde P_{-t_1}(J_1)\cup \widetilde P_{-t_2}(J_2)$,
\item[--] any non-empty compact set $\Lambda\subset J$ such that
$\widetilde P_{-t_1}\left(\Lambda\cap J_1\right)\subset \Lambda
\text{ and } \widetilde P_{-t_2}\left(\Lambda\cap J_2\right)\subset \Lambda$
coincides with $J$.
\end{enumerate}
\end{Definition}

We prove here that the second case of the Proposition~\ref{p.limit} gives aperiodic $\delta$-intervals.
\begin{Proposition}\label{p.aperiodic0}
Let $x\in K\setminus V$ be a $T_\cF$-Pliss point
and let $I$ be a $\delta$-interval whose large returns are all shifting
and which admit a deep sequence of returns $(t_n^l)$ shifting to the left
and another one $(t^r_n)$ shifting to the right.

Then $\alpha(x)$ contains a point $x'\in K\setminus V$
having an aperiodic  $\delta$-interval.
\end{Proposition}

We first select good returns for $I$.
\begin{Lemma}\label{l.aperiodic0}
Under the setting of Proposition~\ref{p.aperiodic0}, there exists a $\delta$-interval $L\subset I$, returns $t_1,t_2>0$
and intervals $L_1,L_2\subset L$
such that
\begin{itemize}
\item[--] $L_1,L_2$ have disjoint interiors and $L=L_1\cup L_2$,
\item[--] $S_{-t_1}(L_1), S_{-t_2}(L_2)$ have disjoint interior and
$L=S_{-t_1}(L_1)\cup S_{-t_2}(L_2)$,
\item[--] any non-empty compact set $\Lambda\subset L$ such that
$S_{-t_1}\left(\Lambda\cap L_1\right)\subset \Lambda
\text{ and } S_{-t_2}\left(\Lambda\cap L_2\right)\subset \Lambda$
coincides with $L$.
\end{itemize}
Moreover $t_1,t_2$ are compositions of large returns inside the sequences
$(t_n^l)$ and $(t^r_n)$.
\end{Lemma}
\begin{proof}
Let us consider two large returns $s,t>0$ shifting to the right and to the left respectively.
They can be chosen inside the sequences $(t_n^r)$ and $(t^l_n)$, hence by Lemma~\ref{l.return}
the compositions of the maps $S_{-s}, S_{-t}$, when they are defined, have no fixed point in $I$.
Denote $D_s=I\cap S_{-s}^{-1}(I)$ and $I_s=S_{-s}(D_s)$
and denote $D_t, I_t$  the domain and the image of $S_{-t}$ in $I$.

\smallskip

\emph{Step 1.} One will reduce the interval $I$ so that the assumptions of the Proposition~\ref{p.aperiodic0} still hold
but moreover either $D_s\cup D_t$ or $I_s\cup I_t$ coincides with $I$. See Figure~\ref{f.reduced-interval}.

\begin{figure}[ht]
\begin{center}
\includegraphics[width=12cm]{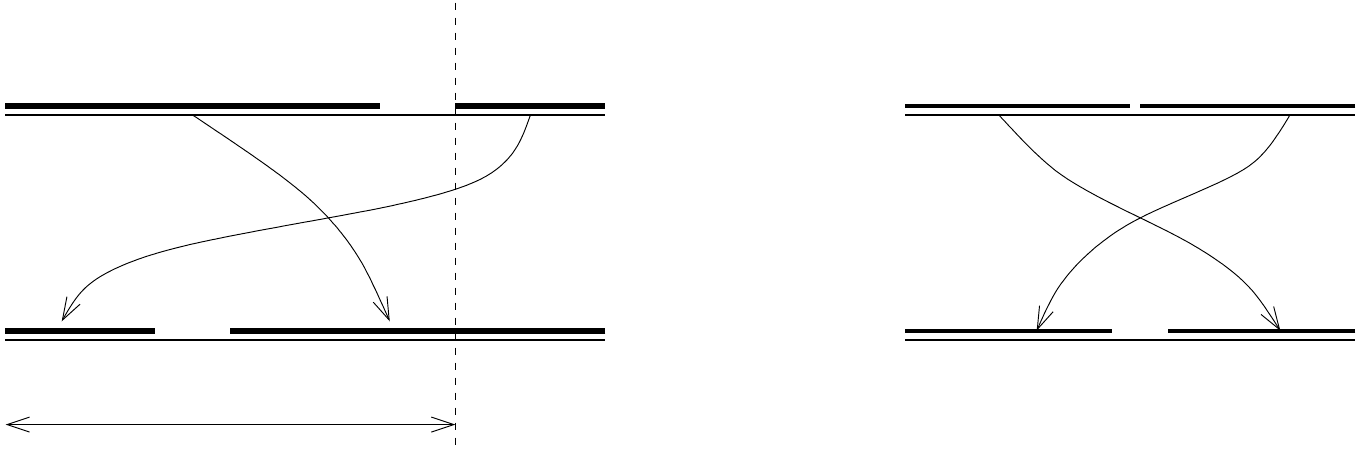}
\begin{picture}(0,0)
\put(-310,93){\scriptsize$D_s$}
\put(-215,93){\scriptsize$D_t$}
\put(-325,20){\scriptsize$I_t$}
\put(-245,20){\scriptsize$I_s$}
\put(-100,93){\scriptsize$D'_s$}
\put(-25,93){\scriptsize$D'_t$}
\put(-95,20){\scriptsize$I'_t$}
\put(-30,20){\scriptsize$I'_s$}
\put(-360,85){$I$}
\put(-360,28){$I$}
\put(-130,85){$I'$}
\put(-130,28){$I'$}
\put(-342,45){$S_{-t}$}
\put(-250,45){$S_{-s}$}
\put(-290,-10){$I'$}
\put(-132,45){$S_{-s}\circ S_{-t}$}
\put(-30,45){$S_{-s}$}
\end{picture}
\end{center}
\caption{Aperiodicity\label{f.reduced-interval}}
\end{figure}
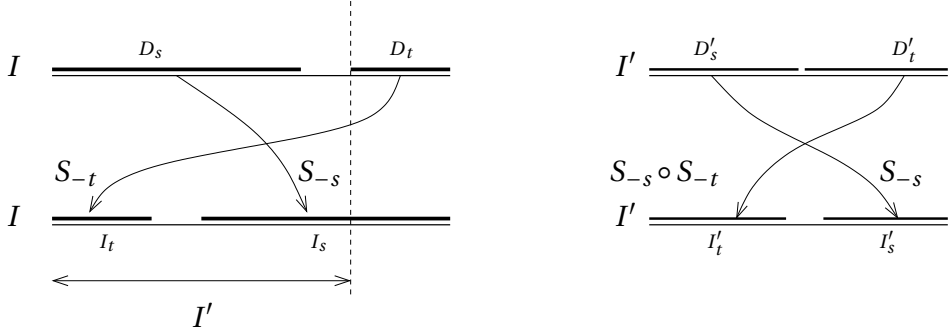

Note that both of these two sets contain the endpoints of $I$.
If both are not connected one reduces $I$ in the following way.
Without loss of generality $0_x$ does not belong to $D_t$.
One can thus move the right point of $I$ (and $D_t)$ inside $D_t$ to the left and reduce $I$ which still contains $0_x$.
This implies that the right endpoints of $D_s,D_t,I_s,I_t,I$ move to the left whereas the left endpoints
remain unchanged. At some moment one of the two intervals $D_s,D_t$ becomes trivial.
We obtain this way new intervals $I',D'_s,D'_t,I'_s,I'_t$.
Note that both can not be trivial simultaneously since otherwise $S_{-s}\circ S_{-t}$ preserves the right endpoint
of the new interval: this one dimensional map is associated to a large return of $I$ which has a fixed point - a contradiction.

Let us assume for instance that $D'_t$ (and $I'_t$) has become a trivial interval (the case
$D'_s$ is trivial is similar): the interval $I'$ is bounded by the left endpoint of $I$
and the left endpoint of $D_t$. Moreover the map $S_{-t}$ sends the right endpoint of $I'$ to its left endpoint.
The map $S_{-t}\circ S_{-s}$ is associated to a return $t'$ of $I'$ which sends the right endpoint of $D'_s$
to the left endpoint of $I'$, hence which shifts left and whose domain coincides with $I'\setminus D'_s$.
We keep the return $s$. We have shown that $D'_s\cup D'_{t'}$ coincides with $I'$.
\medskip

\emph{Step 2.} One again reduces 
$I$ so that the assumptions of Proposition~\ref{p.aperiodic0} still hold, still $D_s\cup D_t$ or $I_s\cup I_t$ coincide with $I$ and furthermore $D_s,D_t$
(resp. $I_s,I_t$) have disjoint interior.

One follows the same argument as in step 1. For instance one moves the right endpoint of $I$ to the left.
Different cases can occur:
\begin{itemize}
\item[--] the point $0_x$ becomes the right endpoint of $I$,
we then exchange the orientation of $I$ and reduce $I$ again; this case will not
appear anymore;
\item[--] the new domains or the new images have disjoint interior; when this occurs
either $D_s\cup D_t$ or $I_s\cup I_t$ coincide with $I$.
\end{itemize}
Note that the domains $D_s,D_t$ can not become trivial as in step 1. Indeed
if for instance $D_t$ (and $I_t$) becomes trivial, since $I$ still coincides
with $D_s\cup D_t$ or $I_s\cup I_t$, one gets that $I$ coincides with $D_s$ or $I_s$,
proving that $R(I)$ contains a fixed unstable manifold - a contradiction.
\medskip

\emph{Step 3.} We have now obtained a $\delta$-interval $L\subset I$ and two
returns $t_1,t_2$, shifting to the left and the right respectively, whose domains $L_1,L_2$
and images $S_1(L_1)$, $S_2(L_2)$ by $S_1:=S_{-t_1}$ and $S_2:=S_{-t_2}$
have disjoint interiors and which satisfy one of the two cases:
\begin{description}
\item[Case 1] $L_1\cup L_2=L$,
\item[Case 2] $S_1(L_1)\cup S_2(L_2)=L$.
\end{description}
After identifying the two endpoints of $L$,
one can define an increasing map $f$ on the circle which
coincides with $S_1$ (resp. $S_2$) on the interior of $L_1$ (resp. $L_2$):
in the first case it is injective and has one discontinuity (and $f^{-1}$ is continuous)
whereas in the second case $f$ is continuous.
By our assumptions on $s$ and $t$, $f$ has no periodic point.
We will prove that $f$ is a homeomorphism
conjugated to a minimal rotation. This will conclude the proof of the lemma.

We discuss the case 2 (the first case is very similar, arguing with $f^{-1}$ instead of $f$).
Poincar\'e theory of orientation preserving circle homeomorphisms extends to continuous increasing maps.
Since $f$ has no periodic point, there exists a unique minimal set $K$.
Let us assume by contradiction
that $K$ is not the whole circle. Let $J$ be a component of its complement.
It is disjoint from its preimages $J_{-n}=f^{-n}(J)$ and (up to replacing $J$ by one of its backward iterate),
$f\colon J_{-n}\to J_{-n+1}$ is a homeomorphism for each $n\geq 0$.
We now use the Denjoy-Schwartz argument to find a contradiction.

Let us fix $n\geq 0$.
For each $0\leq k\leq n$, the interval $f^k(J_{-n})$ is contained in one of the domains
$L_1$ or $L_2$. Hence $f^k|_{J_{-n}}$ coincides with a composition of the maps $S_1$ and $S_2$
and is associated to a return $s_k>0$ of $L$. Note that $\widetilde P_{-s_k}(J_{-n})=J_{k-n}$
is a $C^1$-curve in $R(L)$ tangent to the cone field $\cC^\cE$.
By Proposition~\ref{p.distortion}, there exists $\Delta>0$ such that any sub-rectangle of $R(L)$, bounded by two leaves
$W^u_{R(L)}(u), W^u_{R(L)}(u')$ has distortion bounded by $\Delta$.
Hence
$\widetilde P_{-s_k}(J_{-n})$ has length bounded by $\Delta.|J_{k-n}|$.
Consequently the sum $\sum_{k=0}^n |P_{-s_k}(J_{-n})|$ is uniformly bounded
as $n$ increases. The difference $s_{k+1}-s_k$ is uniformly bounded also,
so that there exists a uniform bound $C_{Sum}$ satisfying
$$\sum_{0\leq m <s_n} |P_{-m}(J_{-n})|<C_{Sum}.$$
From Lemma~\ref{Lem:schwartz}, there exists an interval $\widehat J_{-n}\subset \cW^{cs}(x)$
containing $J_{-n}$
satisfying $|\widehat J_{-n}|\leq 2|J_{-n}|$ such that
each component of
$P_{-s_n}(\widehat J_{-n}\setminus J_n)$ has length
$\eta_S|P_{-s_n}(J_{-n})|$,
where $\eta_S>0$ is a small uniform constant. Note that the projection through unstable holonomy
of $\widetilde P_{-s_n}(\widehat J_{-n})$ in $L$ contains a uniform neighborhood $\widehat J=S_{-s_n}(\widehat J_{-n})$ of $J$.

The large integer $n$ can be chosen so that the small interval
$J_{-n}$ is arbitrarily close to $J$.
Consequently, $\widehat J_{-n}$ is contained in $\widehat J$. This means
$S_{-s_n}(\widehat J_{-n})\supset \widehat J_{-n}$ implying that $\widehat J_{-n}$ contains
a fixed point of the map $S_{-s_n}$.
This is a contradiction since we have assumed that all the returns are shifting.
As a consequence $f$ is a minimal homeomorphism, which implies the lemma.
\end{proof}
\smallskip

\begin{proof}[Proof of Proposition~\ref{p.aperiodic0}]
Let us consider some intervals $L_1,L_2,L$ and some returns $t_1,t_2$ as in
Lemma~\ref{l.aperiodic0}. We denote $\widetilde P_i=\widetilde P_{-t_i}$
and $S_i=S_{-t_i}$, $i=1,2$.
Since $t_1,t_2$ are large inside deep sequences of returns,
the compositions $S_{i_k}\circ\dots\circ S_{i_1}$,
when they are defined, are associated to returns of $L$
(see Lemma~\ref{l.return}).
\medskip

\begin{Claim}
$\widetilde P_1$ and $\widetilde P_2$ commute.
\end{Claim}
\begin{proof}
Both compositions $\widetilde P_1\circ \widetilde P_2$ and $\widetilde P_2\circ \widetilde P_1$
are associated to returns $s,s'>0$.
Note that the common endpoint of $L_1$ and $L_2$ has the same image by
$S_1\circ S_2$ and $S_2\circ S_1$.
If the two compositions do not coincide, the two returns are different,
for instance $s'>s$, but there exists two points $u,u'$ in $R(L)$
in a same unstable manifold such that $u=\pi_x\circ P_{-s'+s}\circ \pi_{\varphi_{-s}(x)}(u')$.
Note that since $t_1,t_2$ can be obtained from deep sequences of returns,
$u,u'$ are arbitrarily close to $L$. The time $s'-s$ can be assumed to be
arbitrarily large: otherwise, we let $t_1,t_2$ go to $+\infty$
in the deep sequences of returns; if $s'-s$ remain bounded,
the points $u,u'$ converge to a point of $L$
a point which is fixed by some limit of the $\pi_x\circ P_{-s'+s}$. This is a contradiction
since the returns of $I$ are non-shifting.
Since $s'-s$ is large, {Lemma~\ref{l.return-existence}} implies that there is a large return sending
$u$ on $u'$, which is a contradiction since large returns are shifting.
Consequently the two returns are the same and the compositions coincide.
\end{proof}
\medskip

By the properties on $S_1,S_2$, there is $(i_k)\in \{1,2\}^\NN$
such that for each $k$.  
$$\widetilde P_{i_k}\circ \dots\circ\widetilde P_{i_1}(0_x)\in R(L).$$
From Lemma~\ref{l.return}, for each $k$ there exists a return $s_k$ such that
$$\widetilde P_{-s_k}=\widetilde P_{i_k}\circ \dots\circ\widetilde P_{i_1}.$$
The dynamics of $S_1,S_2$ is minimal in $L$,
hence the iterates $x_k:=\varphi_{-s_k}(x)$ have
a subsequence $x_{k(j)}$ converging
to a point $x'\in K\setminus V$ such that $\pi_x(x')$ belongs to the unstable manifold of $x$.
{We define the intervals $J,J_1,J_2$} as limits of the iterates of $L,L_1,L_2$
by the maps $P_{-s_{k(j)}}$. In particular $J$ is a $\delta$-interval and the
first item of Definition~\ref{d.aperiodic} holds.

By Remark~\ref{r.intersection}, there exist times $t'_1,t'_2$
such that backward orbit of $x$ during time $[-t_i,0]$
is shadowed by the backward orbit of $x'$ during the time $[-t'_i,0]$.
By the Global invariance, the maps
$\pi_{x'}\circ P_{-t'_i}$, $i=1,2$, from a neighborhood of $0_{x'}$ to $\cN_{x'}$
are conjugated to $\widetilde P_i$ by the projection $\pi_x$ and will be still
denoted by $\widetilde P_i$.
In particular, the map {$P_{-s_{k(j)}}$} from a neighborhood of $0_x\in \cN_x$
to $\cN_{x'}$ coincides with $\widetilde P_{i_{k(j)}}\circ \dots\circ\widetilde P_{i_1}\circ \pi_{x'}$
and $\pi_{x'}\circ \widetilde P_{i_{k(j)}}\circ \dots\circ\widetilde P_{i_1}$.
\medskip

\begin{Claim}
$J$ is aperiodic.
\end{Claim}
\begin{proof}
The projections by $\pi_{x'}$ of $L_1,L_2$ and $S_1(L_1), S_2(L_2)$
have images under the maps {$P_{-s_{k(j)}}$}
which converge to subintervals of $J$:
{by definition,} the first ones are $J_1,J_2$ whereas
{by Global invariance} the last ones
denoted by $J'_1,J'_2$ have disjoint interior and satisfy $J=J'_1\cup J'_2$.
{Since $S_1(L_1)$ and $\widetilde P_1(L_1)$ have the same projection by unstable holonomy,}
Note that $J'_1,J'_2$ are also the limits of
$\widetilde P_1(L_1)$ and $\widetilde P_2(L_2)$
under the maps {$P_{-s_{k(j)}}$}.
Since $\widetilde P_1$ and $\widetilde P_2$ commute this implies that
$J'_1=\widetilde P_1(J_1)$ and $J'_2=\widetilde P_2(J_2)$.
Consequently the second item of the Definition~\ref{d.aperiodic} holds.

Let us consider the projection $\varphi$ from $L$
to $J$ obtained as composition of $\pi_{x'}$ with the unstable holonomy.
Note that it conjugates the orbit of $0_x$ in $L$ by $S_1,S_2$ with the orbit
of $0_{x'}$ in $J$ by $\widetilde P_1,\widetilde P_2$.
Passing to the limit $\varphi$ induces a conjugacy between the dynamics of $S_1,S_2$
on $L$ and $\widetilde P_1,\widetilde P_2$ on $J$.
The third item is thus a consequence of Lemma~\ref{l.aperiodic0}.
\end{proof}

The proof of Proposition~\ref{p.aperiodic0} is now complete.
\end{proof}

\subsection{Construction of a normally expanded irrational torus}\label{ss.topological}
The whole section is devoted to the proof of the next proposition.
\begin{Proposition}\label{p.aperiodic}
If $x\in K\setminus V$ has an aperiodic $\delta$-interval $J$,
then the orbit of $x$ is contained in a minimal set $\cT\subset K$
which is a normally expanded irrational torus, and $\pi_{x}(\cT\cap B(x,r_0/2))\supset J$.
\end{Proposition}

\begin{proof}
The definition of {aperiodic $\delta$-interval}
is associated to two returns $t_1,t_2>0$.
As before we denote $\widetilde P_1=\widetilde P_{-t_1}$ and $\widetilde P_2=\widetilde P_{-t_2}$.

Let us define $\cC$ to be the compact set of points $z\in K$ such that
$$d(z, x)\leq r_0/2 \text{ and } \pi_{x}(z)\in J.$$
The map $z\mapsto \pi_{x}(z)$ is continuous from $\cC$ to $J$.
We also define the invariant set $\cT$ of points whose orbit meets $\cC$.

For any $u\in J$, we define the sets
$$\Lambda_u^+=\{v\in J,~\exists n\in\NN,~k(1),\dots,k(n)\in\{1,2\},~{\rm s.t.}~\widetilde P_{k(n)}\circ \widetilde P_{k(n-1)}\circ \cdots\circ\widetilde P_{k(1)}u=v\},$$
$$\Lambda_u^-=\{v\in J,~\exists n\in\NN,~k(1),\dots,k(n)\in\{1,2\},~{\rm s.t.}~\widetilde P^{-1}_{k(n)}\circ \widetilde P^{-1}_{k(n-1)}\circ\cdots\circ \widetilde P^{-1}_{k(1)}u=v\}.$$

By the definition of aperiodic interval and of $\Lambda_u^+$ and $\Lambda_u^-$,
the set of accumulation points of $\Lambda_u^+$ and $\Lambda_u^-$ are $J$.

\begin{Claim-numbered}\label{c.surjective}
The map $z\mapsto \pi_{x}(z)$ from $\cC$ to $J$ is surjective.
\end{Claim-numbered}
\begin{proof} 
This follows from the fact that the closure of $\Lambda^+_{0_x}$ is $J$,
and that $\widetilde P_{k(n)}\circ \widetilde P_{k(n-1)}\circ\cdots\circ \widetilde P_{k(1)}.0_x$
belongs to the projection of the orbit of $x$ by $\pi_x$.
\end{proof}

\begin{Claim-numbered}\label{c.return}
For any non-trivial interval $J'\subset J$ there is $T>0$ such that any $z\in \cC$ has iterates
$\varphi_{t^+}(z)$, $\varphi_{-t^-}(z)$ with $t^+,t^-\in (1, T)$ which belong to $\cC$
and project by $\pi_{x}$ in $J'$.
\end{Claim-numbered}
\begin{proof} For $z\in \cC$, since $\Lambda_{\pi_x(z)}^+$ is dense in $J$, there {are}
$u_0=\pi_{x}(z)$, $u_1$,\dots, $u_{\ell(z)}$ in $J$ such that
\begin{itemize}

\item[--] $u_{\ell(z)}$ belongs to the interior of $J'$,

\item[--] For $0\le i\le \ell(z)-1$, there is $k(i)\in\{1,2\}$ such that $\widetilde P_{k(i)}(u_i)=u_{i+1}$.
\end{itemize}

Thus by compactness, there is a uniform $\ell$ such that $\ell(z)\le \ell$ for any $z\in{\cC}$. Moreover by the Global invariance, there is $t(i)>0$ such that $\pi_x\varphi_{-t(i)}(z_i)$ is close to $u_i$ and $t(i+1)-t(i)$ is smaller than $2\max(t_1,t_2)$.

Since $\pi_{x}(z)\in J$ there exists
a $2$-Lipschitz homeomorphism $\theta$ of $[0,+\infty)$ such that
$\varphi_{-\theta(t)}(z)$ is close to $\varphi_{-t}(x)$ for each $t\geq 0$.
This proves that $\varphi_{-\theta(t(\ell))}(z)$ projects by $\pi_{x}$ on $u_\ell$,
hence belongs to $\cC$. Since $\ell$ and the differences $t(i+1)-t(i)$ are bounded
and since $\theta$ is $2$-Lipschitz, the time $t^-:=\theta(t(\ell))$ is bounded by a constant which only
depends on $J'$, as required.

The time $t^+$ is obtained in a similar way since $\Lambda_z^-$ is dense in $J$ for any $z\in J$.
\end{proof}

Set $\cT=\cup_{t\in\RR}\varphi_t(\cC)$. From the previous claim, there exists $T>0$ such that
$\cT=\varphi_{[0,T]}(\cC)$. Since $\cC$ is compact, $\cT$ is also compact.
Note that (using Claim~\ref{c.return}) any point in $\cT$ has arbitrarily large forward iterates in $\cC$,
whose projection by $\pi_x$ belongs to $J\subset \cN_x$.
Since $x\in K\setminus V$,
by choosing $\delta>0$ small enough, the local injectivity implies:
\begin{Claim-numbered}\label{c.large-returns}
Any point in $\cT$ has arbitrarily large forward iterates in the $r_0$-neighborhood of $K\setminus V$.
\end{Claim-numbered}

\begin{Claim-numbered}
There exists a curve $\overline \gamma\subset \cC\cap B(x,r_0/2)$ which
projects homeomorphically by $z\mapsto \pi_{x}(z)$ on a non-trivial compact interval of $\operatorname{Interior}(J)$ .
\end{Claim-numbered}
\begin{proof}
We note that:
\begin{enumerate}
\item\label{i.delta} By the ``No small period" assumption, for $k\in\NN$, there exists $\varkappa_k>0$ such that for any $y,y'\in \cC$
and $t\in [0,1]$ satisfying $d(y,y')<\varkappa_k$ and $d(y,\varphi_t(y'))<\varkappa_k$,
then for any $s\in [0,t]$ we have $d(y,\varphi_s(y'))<2^{-k-1}r_0$.

\item\label{i.epsilon} For each $\varkappa_k>0$ as in Item~\ref{i.delta}, there exists $\varepsilon_k>0$ with the following property.
For any $y\in \cC$ satisfying $B(y,\varkappa_k)\subset B(x,r_0)$ and
for any $u\in J$ such that $d(u,\pi_{x}(y))<\varepsilon_k$,
then there is $y'\in B(y,\varkappa_k/2)\cap\cC$ such that $\pi_x(y')=u$. Indeed, by the Local injectivity property, there is $\beta_k>0$ such that for any points $w\in B(y,r_0)$, if $\|\pi_y w\|<\beta_k$, then there is $t\in[-1/2,1/2]$ such that $d(\varphi_t(w),y)<\varkappa_k/2$. By the uniform continuity of the identification $\pi$, for $\varepsilon_k>0$ such that for any $v_1,v_2\in {\cN}_x$, if $\|v_1-v_2\|<\varepsilon_k$, then $\|\pi_y(v_1)-\pi_y(v_2)\|<\beta_k$.  Now for any $u\in J$ such that $d(u,\pi_{x}(y))<\varepsilon_k$, by Claim~\ref{c.surjective} there is $y_0\in B(x,r_0/2)\cap\cC$  such that $\pi_x(y_0)=u$ and $d(\pi_x(y_0),\pi_x(y))<\varepsilon_k$. Hence $d(\pi_y(y_0),0_y)=d(\pi_y\circ\pi_x(y_0),\pi_y\circ\pi_x(y))<\beta_k$. By using the Local injectivity property, $d(\varphi_t(y_0),y)<\varkappa_k/2$ for some $t\in[-1/2,1/2]$. Set $y'=\varphi_t(y_0)$. By Local invariance, we have that $\pi_x(y')=u$.

\item\label{i.lift} By letting $k\to\infty$ and $\beta_k\to 0$, one deduces that
for any $u\in J$ and any $y\in \cC$ such that $\pi_{x}(y)$ is close to $u$,
there exists $y'$ close to $y$ such that $\pi_{x}(y')=u$.
\end{enumerate}

We build inductively an increasing sequence of finite sets $Y_k$ in $\cC$
satisfying:
\begin{itemize}
\item[--] $\pi_{x}$ is injective on $Y_k$,
\item[--] for any $y,y'\in Y_k$ such that
$\pi_{x}(y)$, $\pi_{x}(y')$ are consecutive points of $\pi_{x}(Y_k)$ in $J$,
$$d(\pi_{x}(y),\pi_{x}(y'))<\varepsilon_k.$$
\item[--] if moreover $y''\in Y_{k+1}$ satisfies that $\pi_{x}(y'')$ is
between  $\pi_{x}(y)$ and $\pi_{x}(y')$ in $J$,
then $y''$ is $2^{-k}r_0$-close to $y$ and $y'$.
\end{itemize}
Let us explain how to obtain $Y_{k+1}$ from $Y_k$.
We fix $y,y'\in Y_k$ such that $\pi_{x}(y)$, $\pi_{x}(y')$ are consecutive points of $\pi_{x}(Y_k)$ in $J$ and define the points in $Y_{k+1}$
which project between $\pi_{x}(y)$, $\pi_{x}(y')$.
By items~\ref{i.epsilon},
we introduce a finite set $y''_1,y''_2,\dots,y''_j$
of points in $B(y, \varkappa_k/2)$ such that

\begin{itemize}

\item[--] $y_1''=y$ and $y_j''=y'$,
\item[--] $\pi_{x}\{y_1'',y_2'',\dots,y_j''\}$
is $\varepsilon_{k+1}$-dense in the arc of $J$ bounded by $\pi_{x}(y)$ and $\pi_{x}(y')$,
\item[--] $\pi_{x}(y''_i),\pi_{x}(y''_{i+1})$ 
are consecutive points of the projections in $J$ for $i\in \{1,\dots,j-1\}$,

\end{itemize}

The distance of $y_i''$ and $y_{i+1}''$ may be larger than $r_0 2^{-k-2}$. We need to modify the construction. By Item~\ref{i.epsilon}, there is $t_i\in[-1,1]$ such that $z_i=\varphi_{t_i}(y_i'')$ is $\varkappa_{k+1}/2$-close to $y_{i+1}''$. Choose $n\in\NN$ large enough. Consider the points $\{\varphi_{mt_i/n}(y_i'')\}_{m=0}^n$.  By using the item~\ref{i.lift}
there exists a finite collection $X_i$ of points that
are arbitrarily close to the set $\{\varphi_{mt_i/n}(y_i'')\}_{m=0}^n$ such that they
have distinct projection by $\pi_{x}$
and such that any two such point with consecutive projections are 
$2^{-k-2}r_0$-close.
By the item~\ref{i.delta}, the set $X_i$
is contained in $B(y,2^{-k-1}r_0)$.
Since $d(y,y')<2^{-k-1}r_0$, it is also contained in $B(y',2^{-k}r_0)$.
The set  of points of $Y_{k+1}$which project between $\pi_{x}(y)$, $\pi_{x}(y')$ 
is the union of the $\{y_i'',y_{i+1}''\}\cup X_i$ for any $i$.

Let us define $Y=\cup Y_k$.
The restriction of $\pi_{x}$ to $Y$ is injective
and has a dense image in a non-trivial interval $J'$
contained in $\operatorname{Interior}(J)$.
Its inverse $\chi$ is uniformly continuous:
indeed for any $k$, any $y,y'\in Y_k$
with consecutive projection, and any
$y''\in Y$ such that $\pi_{x}(y'')$
is between $\pi_{x}(y)$, $\pi_{x}(y')$,
the distance $d(y'',y)$ is smaller than $2^{-k+1}r_0$.
As a consequence $\chi$ extends continuously to $J'$
as a homeomorphism such that $\pi_{x}\circ \chi=\id_{J'}$.
The curve $\bar \gamma$ is the image of $\chi$.
\end{proof}

Let $\gamma$ be the open curve obtained by removing the endpoints of $\overline \gamma$.
By the Local invariance,
for $\varepsilon>0$ small,
$\varphi$ is injective on $[-\varepsilon,\varepsilon]\times \overline\gamma$;
its image is contained in $\cC$ and is homeomorphic to the ball $[0,1]^2$. This is because a continuous bijective map from a compact space to a Hausdorff space is a homeomorphism. Thus $\varphi$ is a homeomorphism from $(-\varepsilon,\varepsilon)\times \gamma$ to its
image.

\begin{Claim-numbered}\label{c.open}
The set $\varphi_{(-\varepsilon,\varepsilon)}(\gamma)$ is open in $\cT$.
\end{Claim-numbered}
\begin{proof} Let us fix $z_0\in \varphi_{t}(\gamma)$
for some $t\in (-\varepsilon,\varepsilon)$ and let us consider any point $z\in \cT$ close to $z_0$.
We have to prove that $z$ belongs to the open set $\varphi_{(-\varepsilon,\varepsilon)}(\gamma)$.
Note that
$d(z, x)<r_0$ and $\pi_{x}(z)$ belongs to $R(J)$.
By Claim~\ref{c.return} and the definition of $\cT$,
the point $z$ has a negative iterate $z'=\varphi_{-s}(z)$ in $\cC$, with $s$ bounded, so that $\pi_x(z')\in J$.

One deduces that $\pi_{x}(z)$ belongs to $J$.
Otherwise, $\pi_x(z)\in R(J)\setminus J$.
Since $z$ and $x$ have arbitrarily large returns in the $r_0$-neighborhood of $K\setminus V$
(Claim~\ref{c.large-returns}), the Proposition~\ref{p.coherence} can be applied and
$\pi_x\circ P_{s}\circ \pi_{z'}(J)$ is disjoint from $J$.
Since $s$ is bounded, the distance $d(\pi_x\circ P_{s}\circ \pi_{z'}(J),J)$
is bounded away from zero. We have $\pi_x(z_0)\in J$ whereas
$\pi_x(z)\in P_{s}\circ \pi_{z'}(J)$ which contradicts the fact that
$z$ is arbitrarily close to $z_0$.

Since $z$ is close to $z_0$ and the curve $\gamma$ is open,
we have $\pi_{x}(z)\in \pi_{x}(\gamma)$
and $z\in \varphi_{t'}(\gamma)$ for some $|t'|<1/2$ by the Local injectivity.

By the Local invariance $\varphi_{-t}(z)$ and $\varphi_{-t'}(z)$
have the same projection by $\pi_{x}$.
If $z$ is close enough to $z_0$, both are arbitrarily close to $\gamma$.
Since $\pi_{x}$ is injective on $\gamma$ one deduces
that $d(\varphi_{-t}(z),\varphi_{-t'}(z))$ is arbitrarily small.
By the ``No small period" assumption,  this implies that $|t'-t|$ is arbitrarily small.
Hence $|t'|<\varepsilon$
proving that $z\in \varphi_{(-\varepsilon,\varepsilon)}(\gamma)$, as required.
\end{proof}

By Claim~\ref{c.return}, any point $z\in \cT$ has a backward  iterate in $\cC$
which project by $\pi_{x}$ in $\pi_{x}(\gamma)$, hence 
by Local injectivity has an iterate $\varphi_{-t}(z)$ in $\gamma$. Since $\varphi_{-t}$
is a homeomorphism, one deduces that $z$ has an open neighborhood of the form
$\varphi_{t}(\varphi_{(-\varepsilon,\varepsilon)}(\gamma))$ which is homeomorphic
to the ball $(0,1)^2$. As a consequence, $\cT$ is a compact topological surface.

Since any forward and backward orbit of $\cT$ meets the small open set $\varphi_{(-\varepsilon',\varepsilon')}(\gamma')$, for any $\varepsilon'\in (0,\varepsilon)$
and $\gamma'\subset \gamma$,
the dynamics induced by $(\varphi_{t})$ on $\cT$ is minimal.

From classical results on foliations on surfaces (see~\cite[Theorem 4.1.10, chapter I]{hector-hirsch-foliation}),
we get:
\begin{Claim-numbered}
$\cT$ is homeomorphic to the torus $\TT^2$ and
the induced dynamics of $(\varphi_t)_{t\in \RR}$ on $\cT$
is topologically conjugated to the suspension of an irrational rotation of the circle.
\end{Claim-numbered}

By Claim~\ref{c.open},
$\varphi_{(-\varepsilon,\varepsilon)}(\gamma)$
is a neighborhood of $x$ in $\cT$
and 
$\pi_{x}(\gamma)$
is contained in $\cW^{cs}(x)$, hence is a $C^1$-curve
tangent to $\cE(x)$ at $0_{x}$.
Hence $\cT$ is a normally expanded irrational torus.
Moreover $\pi_{x}(\cT\cap B(x,r_0/2))$ contains $J$
by Claim~\ref{c.surjective}.
This ends the proof of Proposition~\ref{p.aperiodic}.
\end{proof}

\subsection{Proof of the topological stability}\label{ss.Lyapunov}
We prove Proposition~\ref{Prop:dynamicsofinterval} and Lemma~\ref{l.periodic} that imply the topological stability
in Section~\ref{sss.lyapunov-stable}.

\begin{proof}[Proof of Proposition~\ref{Prop:dynamicsofinterval}]
By Proposition~\ref{p.limit}, there is a sequence of $\delta$-intervals $(\widehat I_k)$
at $T_\cF$-Pliss backward iterates $\varphi_{-n_k}(x)$ in $K\setminus V$ and converging to a
$\delta$-interval $I_\infty$ at a $T_\cF$-Pliss point $x_\infty$.
\medskip

Let us assume first that all large returns of $I_\infty$ are shifting.
By Proposition~\ref{p.shifting-return}, $I_\infty$ has two deep sequences of returns,
one shifting to the right and one shifting to the left.
Proposition~\ref{p.aperiodic0} then implies that $\alpha(x_\infty)$ contains a point $x'\in K\setminus V$
with an aperiodic $\delta$-interval.
By Proposition~\ref{p.aperiodic}, the orbit of $x'$ is contained in a normally expanded
irrational torus $\cT$. We have proved that the first case of Proposition~\ref{Prop:dynamicsofinterval} holds.
\hspace{-1cm}\mbox{}
\medskip

Let us assume then that $I_\infty$ admits arbitrarily large non-shifting returns.
Proposition~\ref{p.periodic-return} implies that $I$ is contained in the unstable set of some periodic $\delta$-interval $J$.
\end{proof}

\begin{proof}[Proof of Lemma~\ref{l.periodic}]
Let us consider the rectangle $R(I)$: it is mapped into itself by $P_{-2s}$, where $s$ is the period of $q$.
By assumption (A3), $\cE$ is contracted over the orbit of $q$.
One deduces that there exists a neighborhood $B$ of $0_q$ in $R(I)$
such that for any $u\in B\setminus \cW^{cu}(0_q)$, the backward orbit of $u$
by $(P_t)$ converges towards a periodic point of $I$ which is not contracting along $\cE$.
So Lemma~\ref{l.closing0} implies that $\pi_q(K)$ is disjoint from $B\setminus \cW^{cu}(0_q)$.

For any interval $L\subset \cN_z$ as in Lemma~\ref{l.periodic}, the point $\pi_q(z)$
does not belong to $B\setminus \cW^{cs}(0_q)$. Hence $\pi_q(L)$ contains an interval $J$
which is contained in $B$,
meets $\cW^{cu}(0_q)$ and has a length larger than some constant $\chi_0>0$.
The backward iterate $P_{-2s}(J)$ still contain an interval $J'$ having this property
since $q$ is attracting along $\cE$.

Since $\pi_q(J)$ intersects $R(I)$, there exists $\theta\in \lip$
such that $d(\varphi_{-t}(z),\varphi_{-\theta(t)}(q))$ remain small for any $t\in [0,T]$.
By the Global invariance, if $k$ is the largest integer such that $\theta(T)>ks$,
then $P_{-\theta^{-1}(ks)}(L))$ contains an interval of length larger than $\chi_0/2$.
Since $\theta^{-1}(T)$ is bounded, this implies that $P_{-T}(L)$ has length larger
than some constant $\chi>0$.
\end{proof}

\subsection{Proof of the topological contraction}
We prove a proposition that will allow us to conclude the topological contraction.

\begin{Definition}
$K$ admits \emph{arbitrarily small periodic intervals} if for any $\delta>0$, there is a periodic point $p\in K$, whose orbit supports a periodic $\delta$-interval.
\end{Definition}

\begin{Proposition}\label{Pro:smallperiodic-interval}
Let us assume that:
\begin{itemize}
\item[--] assumptions (A1), (A2), (A3) hold,
\item[--] $K$ does not contain a normally expanded irrational torus,
\item[--] $K$ is transitive,
\item[--] there exists periodic points $p_n\in K$ and positive numbers $\delta_n\to 0$
such that $p_n$ has a periodic $\delta_n$-interval.
\end{itemize}
Then there are  $C_0>0$, $\varepsilon_0>0$, and 
a non-empty open set $U_0\subset K$ such that for any $z\in U_0$, we have

$$\sum_{k\in\NN}|P_k(\cW^{cs}_{\varepsilon_0}(z))|<C_0.$$

\end{Proposition}

In the next 3 subsections, we assume that the setting of Proposition~\ref{Pro:smallperiodic-interval} holds.
We also assume that $\cE$ is not uniformly contracted (since otherwise Proposition~\ref{Pro:smallperiodic-interval} holds by Lemma~\ref{l.summability-hyperbolicity}).
Proposition~\ref{Pro:smallperiodic-interval} is proved in Section~\ref{ss.sum}, and then
Theorems~\ref{Thm:recurrent-contraction} and~\ref{Thm:topologicalcontracting} are proved in Section~\ref{ss.conclusion-topological}.

\subsubsection{The unstable set of periodic points}

\begin{Lemma}\label{l.unstable-set}
For any $\beta>0$, there exist:
\begin{itemize}
\item[--] a periodic point $p\in K\setminus V$ (with period $T$),
\item[--] a point $x\in K\setminus \{p\}$ which is $r_0/2$-close to $p$ and whose $\alpha$-limit set is the orbit of $p$,
\item[--] $r_x>0$ and a connected component $Q$ of $B(0_x,r_x)\setminus \cW^{cu}(x)$ in $\cN_x$ {such that $Q\cap \pi_x(K)=\emptyset$ and the diameter of $P_{-t}(Q)$ is smaller than $\beta$ for each $t>0$.}
\end{itemize}
\end{Lemma}

\begin{proof} 
Since $K$ admits arbitrarily small periodic intervals, there is a periodic point $p\in K$ with period $T>0$ and a periodic $\delta$-interval
$I\subset {\cN}_p$ for $\delta$ small.
Since $\cE$ is uniformly contracted in $V\supset K\setminus U$ (assumption (A2)), one can replace $p$ by one of its iterates so that $p\in K\setminus V$.
By Lemma~\ref{Lem:hyperbolicreturns} (and the continuity of $(P_t)$),
\begin{itemize}
\item[--] the restriction of $\cF$ to the orbit of $I$ by ($P_t)$ is an expanded bundle. 
\end{itemize}
Since $\cE$ is uniformly contracted over the orbit of $0_p$ (by (A3)),
one can assume that:

\begin{itemize}
\item[--] Only the endpoints of $I$ are fixed by $P_{T}$. One is $0_p$, the other one attracts any point of $I\setminus \{0_p\}$
by negative iterations of $P_T$.
\item[--] There is $\beta_p>0$ such that for any $u\in \cN_p$ satisfying $\|P_{-t}(u)\|<\beta_p$ for each $t>0$,
then $u$ is in the unstable manifold of $0_p$ for $P_t$. 
\end{itemize}
Let $r_p>0$ such that $B(p,r_p)\subset U$ and $P_{-t}\circ \pi_p(B(p,r_p))\subset B(0_{\varphi_{-t}(p)},\beta_p)$ for each $t\in [0,T]$.\hspace{-1cm}\mbox{}
\medskip

Since $K$ is transitive, there are sequences $(x_n)$ in $K$ and $(t_n)$ in $(0,+\infty)$ such that
\begin{itemize}
\item[--] $\lim_{n\to\infty}x_n=p$ and $\lim_{n\to\infty}t_n=+\infty$,
\item[--] $d(\varphi_{t}(x_n), \orb(p))<r_p$ for $t\in (0,t_n)$ and $d(\varphi_{t_n}(x_n), \orb(p))=r_p$ for each $n$.
\end{itemize}
Taking a subsequence if necessary, we let $x:=\lim \varphi_{t_n}(x_n)$.
By the Global invariance, $P_{-t}\circ \pi_p(x)\subset B(0_{\varphi_{-t}(p)},\beta_p)$ for each $t>0$, hence $\pi_p(x)$ lies in the unstable manifold of $0_p$.
Combined with the Local injectivity, there exists $\theta\in \lip$ such that $\theta(0)=0$ and $d(\varphi_{\theta(t)}(x),\varphi_t(p))\to 0$ as $t\to -\infty$.
Hence the $\alpha$-limit set of $x$ is $\orb(p)$.
\medskip

We now consider the dynamics of $P_{T}$ in restriction to $\cN_p$. The periodic interval $I$ is normally expanded.
Consequently, for $r_x$ small, one of the components $Q$ of $B(0_x,r_x)\setminus \cW^{cu}(x)$ in $\cN_x$ has an image by $\pi_x$ contained in the unstable set of $I\setminus \{0_p\}$.

Let us assume by contradiction that there exists $y\in U$ such that $\pi_x(y)\in Q$.
The backward orbit of $\pi_p(y)$ by $P_T$ converges to the endpoint $v$ of $I$ which is not attracting along $\cE$.
Lemma~\ref{l.closing0} and the Global invariance imply that the backward orbit of $y$ converges to a periodic orbit in $K$ whose eigenvalues
at the period for the fibered flow coincide with the eigenvalues at the period of the fixed point $v$ for $P_T$. This is a contradiction since all the eigenvalues of
$v$ are non-negative whereas $\cE$ is uniformly contracted over the periodic orbits of $K$.
So $Q$ is disjoint from $\pi_x(K)$.
\medskip

One can choose $p$ and $I$ such that any iterates $P_{-t}(I)$ has diameter smaller than $\beta/2$.
Reducing $r_x$, one can ensure that the iterates of $P_{-t}(Q)$ have diameter smaller than $\beta$
until some time, where it stays close to the orbit of $I$ and has diameter smaller than $2\sup_t\diam(P_{-t}(I))$.
This concludes the proof of the lemma.
\end{proof}

\subsubsection{Wandering rectangles}\label{ss.wandering-rectangles}
One chooses $\delta\in (0,r_0/2)$,  $\beta,\varepsilon>0$, and a component $Q$
as in Lemma~\ref{l.unstable-set} such that
\begin{itemize}
\item[--] if $\varphi_{-t}(x)$, $t>1$, is $2\delta$-close to $x$, then
$Q\cap \pi_x(P_{-t}(Q))=\emptyset$;
\item[--] if $y,z\in K$ are close to $x$ and $\theta\in \lip$ satisfies
$d(\varphi_{\theta(t)}(y),\varphi_t(z))<\delta$ for $t\in [-2,0]$, then $\theta(0)-\theta(-2)>3/2$;
\item[--] $\beta>0$ associated to a shadowing at scale $\delta$ as in the Global invariance (Remarks~\ref{r.identification}.e);
\item[--] for any point $z$, the forward iterates of $\cW^{cs}_\varepsilon(z)$ have length smaller than $\beta/3$ and than
 the constant $\delta_\cE$ given by Lemma~\ref{l.summability-hyperbolicity}
 (this is possible since $\cE$ is topologically stable);
\item[--] the backward iterates of $Q$ have diameter smaller than $\beta/2$.
\end{itemize}

For $z$ close to $x$, one considers the closed curve $J(z)\subset \cW^{cs}(z)$ of length $\varepsilon$
bounded by $0_z$, such that $\pi_x(J(z))$ intersects $Q$.
Since $\pi_x(z)$ does not belong to $Q$ (by Lemma~\ref{l.unstable-set}), the unstable manifold of $0_p$
intersects $\pi_p(J(z))$, defining two disjoint arcs $J(z)=J^0(z)\cup J^1(z)$ such that $J^1(z)$ is bounded by $0_z$, disjoint from $Q$
and $\pi_x(J^0(z))\subset Q$.
\medskip

Let $H(z)$ denote the set of integers $n\geq 0$ such that $\varphi_n(z)\in K\setminus V$ and
$(z,\varphi_n(z))$ is $T_\cF$-Pliss. For $n\in H(z)$,
we set $J_n(z)=P_n(J(z))$ and $J^{0}_n(z)=P_n(J^{0}(z))$.

\begin{Lemma}\label{l.construction-rectangle}
There is $C_R>0$ with the following property.
For any $z$ close to $x$ and any $n\in H(z)$,
there exists a rectangle $R_n(z)\subset \cN_{\varphi_n(z)}$ which is the image of $[0,1]\times B_{d-1}(0,1)$ by a homeomorphism $\psi_n$ such that
(see Figure~\ref{f.topological-hyperbolicity}):
\begin{enumerate}

\item\label{i.construction-rectangle} $\text{Volume\;}(R_n(z))>C_R.|J^0_n(z)|$,
\item the preimages $P_{-t}(R_n(z))$ for $t\in [0,n]$ have diameter smaller than $\beta/2$,
\item $\pi_x\circ P_{-n}(R_n(z))$ is contained in $Q$.
\end{enumerate}
\end{Lemma}
\begin{figure}[ht]
\begin{center}
\includegraphics[width=10cm]{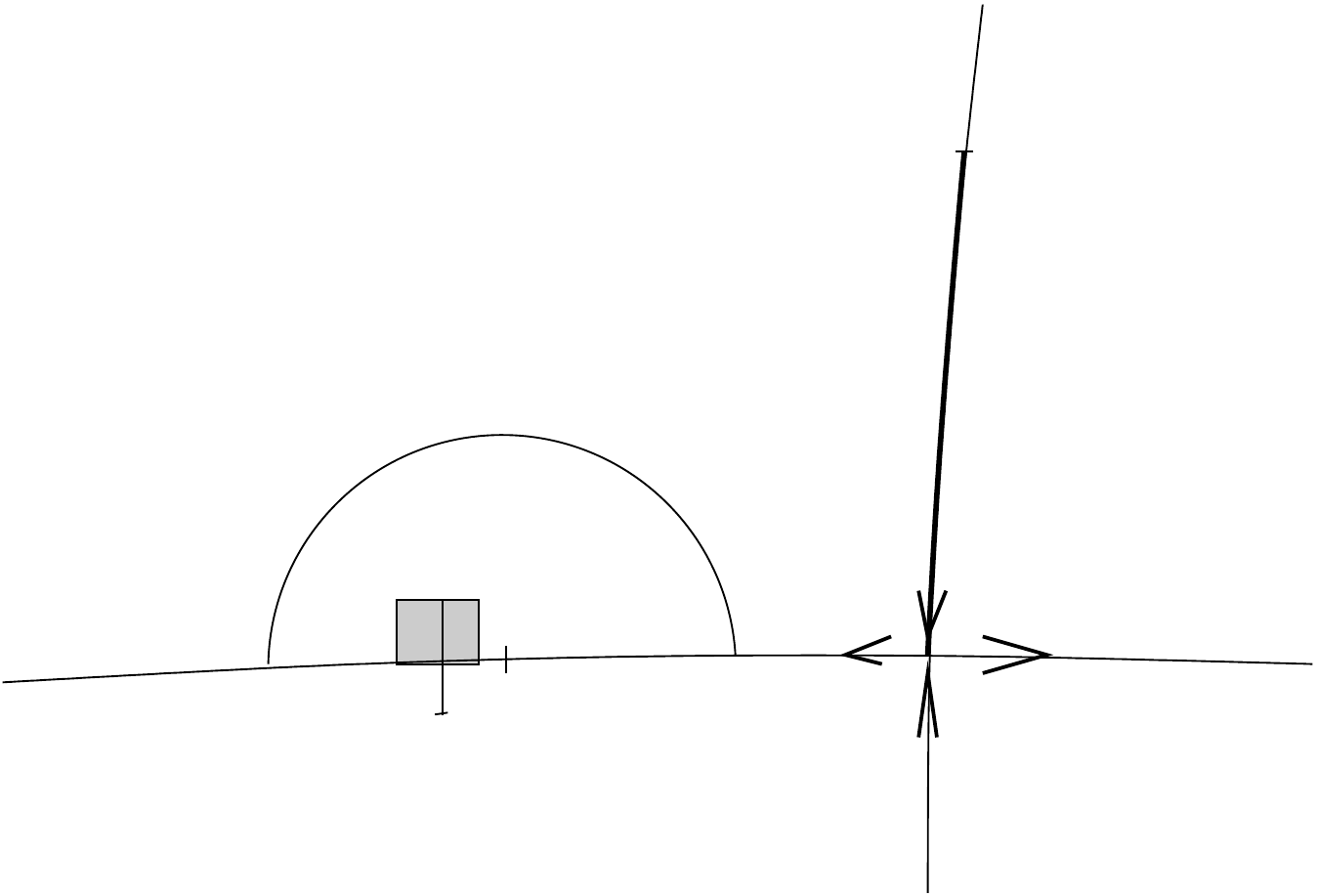}
\begin{picture}(0,0)
\put(-82,41){$p$}
\put(-75,120){$I$}
\put(-182,40){$\pi_p(x)$}
\put(-222,35){$\pi_p(z)$}
\put(-210,70){\scriptsize$\pi_p(R_0(z))$}
\put(-210,105){\scriptsize$\pi_p(Q)$}
\end{picture}
\end{center}
\caption{Wandering rectangles\label{f.topological-hyperbolicity}}
\end{figure}
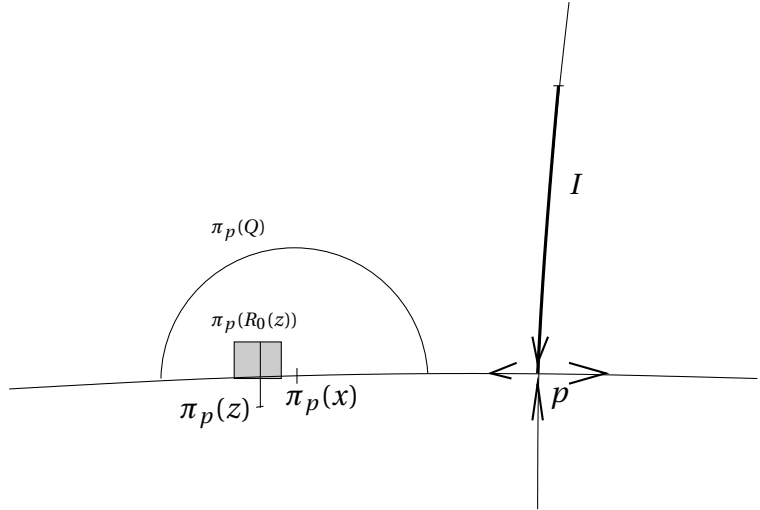
\begin{proof}
The construction is very similar to the proof of Proposition~\ref{p.rectangle}.
Let us fix $\alpha'>0$ and $\alpha_{min}>0$ much smaller.
One considers a rectangle in $\cN_z$ whose interior projects by $\pi_x$ in $Q$,
given by a parametrization $\psi_0$ such that $\psi_0([0,1]\times \{0\})=J^0(z)$
and each disc $\psi_0(\{u\}\times B_{d-1})$ is tangent to the center-unstable cones,
has diameter smaller than $\alpha'$ and contains a center-unstable ball centered at $J^0(z)$
and with radius much larger than $\alpha_{min}$. Moreover $\pi_p\circ \psi_0(\{0_p\}\times B_{d-1})$ is contained in
the unstable manifold of $p$.

Since the center-unstable cone-field is invariant,
at any iterate $\varphi_n(z)$ such that $(z,\varphi_n(z))$ is $T_\cF$-Pliss,
one can build a similar rectangle $R_n$ such that $P_{-t}(R_n)$, $0<t<n$, has center-unstable disc
of diameter smaller than $\alpha'$ and $P_{-n}(R_n)\subset R_0$.

Since the forward iterates of $J^0(z)$ have length smaller than $\beta/3$,
by choosing $\alpha'$ small enough, one guaranties that the $P_{-t}(R_n)$, $0<t<n$,
have diameter smaller than $\beta/2$, and that $P_{-n}(R_n)\subset R_0(z)$ are contained in $\pi_z(Q)$.
The estimate on the volume is obtained from Fubini theorem and the distortion estimate
given by Proposition~\ref{p.distortion}.
\end{proof}

\begin{Lemma}\label{l.rectangle-disjoint}
For each $z$ close to $x$ and each $n< m$ in $H(z)$,
if $d(\varphi_n(z),\varphi_m(z))<\delta$, then $\pi_{\varphi_n(z)}(R_m(z)) \cap R_n(z)=\emptyset$.
\end{Lemma}
\begin{proof}
Let us assume by contradiction that $\pi_{\varphi_m(z)}(R_m(z))$ and $R_n(z)$ intersect.

Since $P_{-t}(R_n(z)\cup J_n(z))$, $t\in [0,n]$, and $P_{-t}(R_m(z)\cup J_m(z))$, $t\in [0,m]$, have diameter smaller than $\beta$,
the Global invariance (Remark~\ref{r.identification}.(g)) applies: there is $\theta\in \lip$ such that
\begin{itemize}
\item[--] $|\theta(n)-m|<1/4$,
\item[--] for any $t\in [0,n]$, one has $d(\varphi_{t}(z),\varphi_{\theta(t)}(z))<\delta$,
\item[--] $\pi_x\circ P_{\theta(0)-m}(R_m(z))$ intersects $\pi_x\circ P_{-n}(R_n(z))$, hence $Q$.
\end{itemize}
In particular $\theta(n)>n+1/2$ and Proposition~\ref{p.no-shear} gives $\theta(0)>2$.

Since the backward iterates of $Q$ by $P_t$ have diameter smaller than $\beta$,
the Global invariance (item (g) in Remarks~\ref{r.identification})
can be applied to the points $x$ and $\varphi_{\theta(0)}(z)$.
It gives $\theta'\in \lip$ with $|\theta'(\theta(0))|\leq 1/4$ such that
$d(\varphi_{\theta'(t)}(x),\varphi_{t}(z))<\delta$ for each $t\in [0,\theta(0)]$.
Moreover $\pi_x\circ P_{\theta'(0)}(Q)$ intersects $R_0(z)$, hence $Q$.
We have $d(\varphi_{\theta'(0)}(x),x)<2\delta$ and $1/4>\theta'(\theta(0))>\theta'(0)+3/2$ by our choice of $\delta$
at the beginning of Section~\ref{ss.wandering-rectangles}

We proved that $\pi_x\circ P_{-t}(Q)\cap Q\neq \emptyset$ for some $t>1$ such that $d(\varphi_{-t}(x),x)<2\delta$.
This contradicts the choice of $Q$. The rectangles $\pi_{\varphi_n(z)}(R_m(z))$ and $R_n(z)$ are thus disjoint.
\end{proof}

As a consequence of Lemma~\ref{l.rectangle-disjoint}, one gets

\begin{Corollary}\label{c.volume-bounded}
There exits $C_H>0$ such that for any $z$ close to $x$,
$$\sum_{n\in H(z)}|J^0_{n}(z)|< C_H.$$
\end{Corollary}
\begin{proof}
As in the proof of Lemma~\ref{Sub:disjointcase}, one fixes a finite set $Z\subset U$ such that
any point $z\in U$ is $\delta/2$-close to a point of $Z$. Let $C_{Vol}$ be a bound on the volume
of the balls $B(0_z,\beta_0)\subset \cN_z$ over $z\in K$ and let $C_H=2C_R^{-1}C_{Vol}\text{Card} (Z)$.
Since identifications are $C^1$, up to reducing $r_0$, one can assume that the modulus of their Jacobian is smaller than $2$.

The statement now follows
from the item~\ref{i.construction-rectangle} of Lemma~\ref{l.construction-rectangle} and the disjointness of the rectangles (Lemma~\ref{l.rectangle-disjoint}).
\end{proof}

\subsubsection{Summability. Proof of Proposition~\ref{Pro:smallperiodic-interval}}\label{ss.sum}

\begin{Lemma}\label{l.summability-topological-hyperbolicity}
There exists $C_{sum}>0$ such that for any point $z$ close to $x$,  
we have
$$\forall n\geq 0,~~~\sum_{k=0}^n |P_k(J^0(z))|< C_{sum}.$$
\end{Lemma}

\begin{proof}

Let us denote by $n_0=0<n_1<n_2<\dots$ the integers in $H(z)$.
By Proposition~\ref{l.summability}, for any $i$, the piece of orbit $(\varphi_{n_i}(z),\varphi_{n_{i+1}}(z))$
is $(C_{\cE},\lambda_{\cE})$-hyperbolic for $\cE$.

Let $\delta_\cE, C'_{\cE}$ be the constants associated to $C_{\cE},\lambda_{\cE}$ by Lemma~\ref{l.summability-hyperbolicity}.
We have built $J(z)$ such that any forward iterate has length smaller than $\delta_\cE$.
Hence Lemma~\ref{l.summability-hyperbolicity} implies that
$$\sum_{k=n_i}^{n_{i+1}} |P_k(J^0(z))|< C_{\cE}'|J^0_{n_{i+1}}|.$$
With Corollary~\ref{c.volume-bounded}, one deduces
$$\sum_{k=0}^n |P_k(J^0(z))|< C_{Sum}:=C_{\cE}'C_{H}.$$
\end{proof}

We can now end the proof of the proposition.

\begin{proof}[Proof of the Proposition~\ref{Pro:smallperiodic-interval}]

Let $\eta_S>0$ be the constant associated to $C_{Sum}$ by Lemma~\ref{Lem:schwartz}.
If $z$ belongs to a small neighborhood $U_0$ of $x$
and if $\varepsilon_0$ is small enough,
the intervals $\cW^{cs}_{\varepsilon_0}(z)$ and $J^0(z)$ are both contained in an interval $\widehat J(z)\subset \cW^{cs}(z)$
such that $|\widehat J(z)|\leq (1+\eta_S)|J^0(z)|$.
Combining Lemma~\ref{Lem:schwartz} with Lemma~\ref{l.summability-topological-hyperbolicity}, one gets
$$\sum_{k=0}^n |P_k(\cW^{cs}_{\varepsilon_0}(z))|\leq  \sum_{k=0}^n |P_k(\widehat J(z))|< 2\sum_{k=0}^n |P_k(J^0(z))|<C_0:=2C_{sum}.$$
\end{proof}

\subsubsection{Proof of Theorems~\ref{Thm:recurrent-contraction} and~\ref{Thm:topologicalcontracting}}\label{ss.conclusion-topological}
Theorems~\ref{Thm:recurrent-contraction} and~\ref{Thm:topologicalcontracting}
will follow from the following.

\begin{Theorem}\label{Thm:recurrent-contraction-bis}
Let us assume that $K$ is transitive, satisfies assumptions (A1), (A2), (A3),
does not contain a normally hyperbolic irrational torus, and that $\cE$ is not uniformly contracted.
Then, there exists a forward invariant compact set $K'\subsetneq K$ (possibly empty) such that:
\begin{enumerate}
\item[--] $\cE|_{K'}$ is not uniformly contracted.
\item[--] $\cE$ is \emph{topologically contracted} on $K\setminus K'$:
for any $\rho>0$ there is $\varepsilon_0>0$ such that for any point $x\in K\setminus K'$, the image $P_t({\cW}^{cs}_{\varepsilon_0}(x))$
is well-defined for any $t\ge 0$, has diameter smaller than $\rho$ and satisfies
$$\lim_{t\to+\infty} |P_t({\cW}^{cs}_{\varepsilon_0}(x))|=0.$$
\end{enumerate}
\end{Theorem}

\begin{proof}[Proof of Theorem~\ref{Thm:recurrent-contraction}]
Let us assume that we are not in the first case of the theorem, that $\cE$ is not topologically contracted (otherwise the second case holds trivially) and that $K$ is transitive (otherwise there is not point $x\in K$ such that $\omega(x)=K$).
Theorem~\ref{Thm:recurrent-contraction-bis} then applies and provides a subset $K'\subsetneq K$.
Any point $x$ such that $\omega(x)=K$ belongs to $K\setminus K'$ and satisfies the second case of the theorem.
\end{proof}

\begin{proof}[Proof of Theorem~\ref{Thm:topologicalcontracting}]
Let us assume that we are not in the two first cases case of the theorem and that $\cE$ is not topologically contracted (otherwise the third case holds trivially).

The set $K$ is transitive. Indeed since $\cE$ is not uniformly contracted,
there exists an ergodic measure whose Lyapunov exponent along $\cE$ is non-negative.
$\cE$ is uniformly contracted on each proper invariant subset of $K$ (because we are not in the first case of the theorem).
Hence the support of the measure coincides with $K$ and $K$ is transitive.

Theorem~\ref{Thm:recurrent-contraction-bis} then applies and provides $\varepsilon_0$ and a forward invariant subset $K'\subsetneq K$.
Since the first case of the theorem does not hold, $\cE$ is uniformly contracted on $K'$.
In particular, (maybe after decreasing $\varepsilon_0$) for any $z\in K'$ the conclusion of the theorem holds.
For any $z\in K\setminus K'$, the conclusion of the theorem holds also by Theorem~\ref{Thm:recurrent-contraction-bis}.
Theorem~\ref{Thm:topologicalcontracting} is now proved.
\end{proof}

\begin{proof}[Proof of Theorem~\ref{Thm:recurrent-contraction-bis}]
Let us fix $\delta>0$ arbitrarily small (in particular smaller than $\rho$). Since $\cE$ is topologically stable,
there is $\varepsilon>0$ such that for any $x\in K$ and any $t>0$, one has
$$P_t({\cW}^{cs}_{\varepsilon}(x)) \subset {\cW}^{cs}_{\delta}(\varphi_t(x)).$$
Since the topological contraction fails, there are
 $(x_n)$ in $K$, $(t_n)\to +\infty$ and $\chi>0$ such that
$$\chi<|P_{t_n}({\cW}^{cs}_{\varepsilon}(x_n))| \text{ and } |P_{t}({\cW}^{cs}_{\varepsilon}(x_n))|<\delta,~\forall t>0.$$
Let $I=\lim_{n\to\infty}P_{t_n}({\cW}^{cs}_{\varepsilon}(x))$. It is a $\delta$-interval and by Proposition~\ref{Prop:dynamicsofinterval},
it is contained in the unstable set of a periodic $\delta$-interval since $K$ contains no normally expanded irrational tori.
This proves that $K$ admits arbitrarily small periodic intervals and Proposition~\ref{Pro:smallperiodic-interval} applies. 

One gets a non-empty open set $U_0\subset K$ such that at any $z\in U_0$ a summability holds in the $\cE$ direction, i.e., there are $C_0>0$ and $\delta_0>0$, such that for any $z\in U_0$, we have that
$$\sum_{k\in\NN}|P_k(\cW^{cs}_{\delta_0}(z))|<C_0.$$
This certainly implies that $\lim_{k\to\infty}|P_k(\cW^{cs}_{\delta_0}(z))|=0$ for any $z\in U_0$.

Let $\bar K$ be the set of points whose forward orbit does not intersect $U_0$: it is forward invariant and compact.
If $\cE$ is not uniformly contracted on $\bar K$ we set $K'=\bar K$ and the conclusion of the theorem holds.

If $\cE$ is not uniformly contracted on $\bar K$, we set $K'=\emptyset$.
(Maybe after decreasing $\varepsilon_0$) for any $z\in \bar K$ the conclusion of the theorem holds
by uniform contraction and for any $z\in K\setminus \bar K$ the conclusion of the theorem holds
by what we proved above for points in $U_0$.
So Theorem~\ref{Thm:recurrent-contraction-bis} holds in all cases.

\end{proof}

\section{The proof of the main theorems}\label{s.main}

\subsection{The local fibered flow}
We recall that the linear Poincar\'e flow $(\psi_t)_{t\in \RR}$ has been defined in the introduction.
It is convenient to introduce a rescaled version $(\psi^*_t)_{t\in \RR}$, defined on $\cN_x$ for $x\in M\setminus \sing(X)$ by:
$$\psi_t^*\cdot u=\frac{\|X(x)\|}{\|X(\varphi_t(x))\|}\psi_t\cdot u.$$

We also consider the \emph{non-linear Poincar\'e flow} $(P_t)_{t\in \RR}$, which is a local flow in a neighborhood of the $0$-section of the bundle $\cN$ and which fibers above the restriction of the flow $(\varphi_t)_{t\in \RR}$ to $M\setminus \sing(X)$. It is defined as follows.
For any $x\in M\setminus \sing(X)$ and any $t\in \RR$, one considers the images of neighborhoods of $0$ in $\cN_x$
and $\cN_{\varphi_t(x)}$ by the exponential maps at $x$ and $\varphi_t(x)$ respectively: these are transverse sections to $X$. The Poincar\'e map for the flow $\varphi$ between these sections induces, after conjugacy by the exponential maps, a diffeomorphism
$P_t$ between the neighborhoods of $0$ in $\cN_x$ and $\cN_{\varphi_t(x)}$.

One then introduces a rescaled version $(P^*_t)_{t\in \RR}$ defined by:
$$P_t^*(u)=\|X(\varphi_t(x))\|^{-1}\cdot P_t(\|X(x)\|\cdot u).$$

One has the following result from \cite{CY2}.

\begin{Theorem}[Compactification]\label{Thm:local-fibered-flow-model}
Let $X$ be a $C^k$-vector field, $k\geq 1$, over a compact manifold $M$.
Let $\Lambda\subset M$ be a compact set which is invariant by the flow $(\varphi_t)_{t\in\RR}$
associated to $X$ such that $DX(\sigma)$ is invertible at each singularity $\sigma\in \Lambda$.

Then, there exist a topological flow $(\widehat \varphi_t)_{t\in \RR}$ over a compact metric space $\widehat \Lambda$,
and a local $C^k$-fibered flow $(\widehat P^*_t)$ on a Riemannian vector bundle $\widehat {\cN }\to \widehat \Lambda$
whose fibers have dimension $\dim(M)-1$ such that:
\begin{itemize}
\item[--] The restriction of $\varphi$ to $\Lambda\setminus {\rm Sing}(X)$ embeds in $(\widehat \Lambda, \widehat \varphi)$ through a map $i$; moreover the image $i(\Lambda\setminus{\rm Sing}(X))$ is open and dense in $\widehat\Lambda$.
\item[--] The restriction of $\widehat {\cN}$ to $i(\Lambda\setminus \sing(X))$ is isomorphic to the normal bundle $\cN |_{\Lambda\setminus \sing(X)}$
through a map $I$, which is fibered over $i$ and which is an isometry along each fiber $\cN_x$.
\item[--] The fibered flow $\widehat P^*$ over $i(\Lambda\setminus \sing(X))$ is conjugated by $I$
near the zero-section to the rescaled sectional Poincar\'e flow $P^*$:
$$\widehat P^*=I\circ P^*\circ I^{-1}.$$
\item[--]  For any open set $\widehat U\subset \widehat \Lambda$
whose closure is disjoint from
$\widehat \Lambda\setminus i(\Lambda\setminus{\rm Sing}(X))$, there exist $C>0$ and a $C^k$ identification $\widehat\pi$ on $\widehat U$
which is compatible with the flow $({\widehat P}_{C.t}^*)$.
\end{itemize}
\end{Theorem}

Note that $\widehat \psi^*:=D\widehat P^*(0)$ is an extension of the rescaled linear Poincar\'e flow $\psi^*$.
\medskip

We recall the constructions in Theorem~\ref{Thm:local-fibered-flow-model} from \cite{CY2}.
\begin{itemize}
\item We blow up every singularity $\sigma$ to be its sphere bundle $T_\sigma^1M$; in this way, we get a compact manifold $\widehat M$ with boundary and a continuous projection $p:~\widehat M\to M$ which is one-to-one above $\Lambda\setminus{\rm Sing}(X)$. The map $i$ is just the inverse of $p$ on the regular set. The flow $\widehat\varphi$ is the unitary tangent flow of $\varphi$ as in \cite[Proposition 3.1]{CY2}.
\item As in \cite[Section 3.4]{CY2}, for each $u\in p^{-1}(\{\sigma\})$, the line field $\mathbb RX$ can be texted to be the space spanned by $DX(\sigma)\cdot u$. Thus, for each $(0,u)\in \widehat \Lambda\cap p^{-1}(\{\sigma\})$ the space $\widehat \cN_u$ is the orthogonal space of $DX(\sigma)\cdot u$ inside $T_\sigma M$.
\item By combining with \cite[Lemma 6.3]{CY2}, the image by the flow $\widehat \psi^*$ of any $v\in \widehat \cN_u$
coincides with the orthogonal projection of $\frac{\|DX(\sigma)u\|}{\|DX(\sigma)D\varphi_t(\sigma)u\|}D\varphi_t(\sigma)\cdot v$ to $\widehat \cN_{\widehat \varphi_t(u)}$.
\end{itemize}

\subsection{The lifted local fiber flow with dominated structures}
When the compact invariant set admits a singular dominated splitting, one can get more accurate properties for the lifted local fiber flow.

\begin{Theorem}\label{Thm:local-fibered-flow-model-dominated}
Let $X$ be a $C^k$-vector field, $k\geq 1$, over a compact manifold $M$.
Let $\Lambda\subset M$ be a compact set which is invariant by the flow $(\varphi_t)_{t\in\RR}$
associated to $X$ such that each singularity $\sigma\in \Lambda$ is hyperbolic.
Let $(\widehat\varphi_t)_{t\in\RR}$, $\widehat\Lambda$, the local $C^k$-fibered flow $({\widehat P}_t^*)$, the vector bundle $\widehat{\cN}$, the maps $i$ and $I$ be as in the statement of Theorem~\ref{Thm:local-fibered-flow-model}.
And let us assume that $\Lambda$ admits an index-$1$ singular dominated splitting $\cN|_{\Lambda\setminus{\rm Sing}(X)}=\cN^{cs}\oplus\cN^{cu}$.

Then $\widehat\Lambda$ admits a dominated splitting $\widehat\cN|_{\widehat\Lambda}=\widehat\cN^{cs}\oplus\widehat\cN^{cu}$ w.r.t. the extended linear Poincar\'e flow with $\dim(\widehat \cN^{cs})=1$
and there is a compact neighborhood $\widehat V$ of $\widehat\Lambda\setminus i(\Lambda\setminus{\rm Sing}(X))$ such that $\widehat\cN^{cs}$ is uniformly contracted in $\widehat V$ with respect to the extended linear Poincar\'e flow.

Moreover, one has the following properties:
\begin{itemize}
\item For any regular periodic orbit $\cO\subset\Lambda$, if $\cO$ has only negative Lyapunov exponents along $\cN^{cs}$, then $\widehat\cO=i(\cO)$ has only negative Lyapunov exponent along along $\widehat\cN^{cs}$ for $\widehat \psi^*$.

\item If $\Lambda$ contains no normally expanded irrational torus, then $\widehat\Lambda$ does not contain any normally expanded irrational torus.

\end{itemize}

\end{Theorem}
\begin{proof} We prove successively all the properties of the theorem.

\paragraph{\it (a) The extension of the dominated splitting.}

Through the map $I$, the bundles $\cN^{cs}$ and $\cN^{cu}$ can be lifted over $i(\Lambda\setminus{\rm Sing}(X))$ as bundles $\widehat\cN^{cs}$ and $\widehat\cN^{cu}$. This defines a dominated splitting $\widehat \cN=\widehat \cN^{cs}\oplus \widehat \cN^{cu}$ over $\widehat\Lambda$ for the extended rescaled linear Poincar\'e flow $\widehat \psi^*$ (hence for $\widehat P^*$).
Indeed:
\begin{itemize}
\item[--] dominated splittings are preserved by rescaling, hence $\widehat\cN^{cs}\oplus \widehat \cN^{cu}$ is a dominated splitting for
$\widehat \psi^*$ over $i(\Lambda\setminus \sing(X))$;
\item[--] for continuous linear cocycles, dominated splittings extend to the closure.
\end{itemize}

\paragraph{\it (b) The existence of uniformly contracted neighborhood $\widehat V$.}
We only consider the case where $\Lambda$ does contain singularities, since otherwise $\widehat\Lambda\setminus i(\Lambda\setminus{\rm Sing}(X))=\emptyset$ and the conclusion will be trivial.
There exists a continuous projection $p\colon \widehat \Lambda\to \Lambda$.

Each singularity $\sigma\in \Lambda$ admits a splitting $T_\sigma M=E^{ss}(\sigma)\oplus E^{cu}(\sigma)$ with $\dim E^{ss}(\sigma)=1$ by definition of index-$1$ singular domination. Let $\Delta_\sigma$ be the set of unit vectors $u\in {E^{cu}(\sigma)}\subset T_\sigma M$.
The set $\widehat \Lambda\cap \Delta_\sigma$ is compact and $\widehat \varphi$-invariant.

For each $u\in \widehat \Lambda\cap \Delta_\sigma$ the space $\widehat \cN_u$ is the orthogonal space of
$DX(\sigma)\cdot u$ inside $T_\sigma M$. The image by the flow $\widehat \psi^*$ of any $v\in \widehat \cN_u$
can be obtained as the orthogonal projection of $D\varphi_t(\sigma)\cdot v$ to $\widehat \cN_{\widehat \varphi_t(u)}$.
Since $D\varphi_t(\sigma)$ preserves the splitting $T_\sigma M=E^{ss}\oplus { E^{cu}}$,
and since $DX(\sigma)\cdot u$ has a uniform angle with $E^{ss}(\sigma)$, the orthogonal projection of this splitting
induces a splitting $\widehat \cN_u= \widehat\cE(u)\oplus \widehat\cF(u)$ above $\Delta_\sigma\cap \widehat \Lambda$
which is invariant by $\widehat \psi^*$ such that $\cE(u)$ is one dimensional and uniformly contracted by the flow $\widehat \psi^*$.

On $\widehat\Lambda\cap\Delta_\sigma$, the dominated splittings $\widehat \cN^{cs}\oplus \widehat \cN^{cu}$ and $\widehat\cE\oplus \widehat\cF$ have the same dimensions,
hence coincide. Moreover since $W^{ss}(\sigma)\cap \Lambda=\{\sigma\}$, we have $p^{-1}(\sigma)\cap \widehat\Lambda=\Delta_\sigma\cap \widehat\Lambda$,
{where $p$ denotes the projection $\widehat\Lambda\to \Lambda$}.
This shows that $\widehat \cN^{cs}$ is uniformly contracted by $\widehat P^*$ over $\widehat\Lambda\setminus i(\Lambda\setminus{\rm Sing}(X))$,
hence on any small neighborhood $\widehat V$.

\paragraph{\it (c) Lyapunov exponents of periodic orbits.}
\begin{Lemma}\label{Lem:extended-periodic-contracted}
For any regular periodic orbit $\cO$ in $\Lambda$, the
Lyapunov exponents of $\cO$ for $\psi$ and the Lyapunov exponents of $\widehat \cO=i(\cO)$ for
$\widehat \psi^*$ coincides.
\end{Lemma}
\begin{proof} Indeed the two linear flows are conjugated by the (derivative of the) map $I$.
\end{proof}

\paragraph{\it (d) Normally expanded irrational tori.}
\begin{Lemma}\label{Lem:blowup-no-torus}
If $\cT\subset \widehat \Lambda$ is a normally expanded irrational torus,
then $p(\cT)$ is an invariant torus which carries a dynamics topologically equivalent to an irrational flow.
\end{Lemma}
\begin{proof}
Let $\cT\subset \widehat \Lambda$ be an expanded irrational torus $\cT$.
By Lemma~\ref{l.torus}, the bundle $\widehat \cN^{cu}$ is uniformly expanded over $\cT$.
Since it does not contain fixed points of $\widehat \varphi$, it projects by $p$ in $\Lambda\setminus \sing(X)$.
By construction, the dynamics of $\widehat \psi^*$ over $\cT$ and $\psi^*$ over $p(\cT)$ are the same,
hence $\cN^{cu}$ is uniformly expanded over $p(\cT)$ by $\psi^*$.
Since $p(\cT)\cap \sing(X)=\emptyset$, one deduces that $\cN^{cu}$ is uniformly expanded over $p(\cT)$ by $\psi$.
It follows from a standard argument that the tangent flow over $p(\cT)$ has a dominated
splitting $TM|_{p(\cT)}=E^c\oplus E^{uu}$, with $\dim(E^{c})=2$, see \cite[Lemma 2.13]{GY} for instance.

By partial hyperbolicity, each $x\in p(\cT)$ has a strong unstable manifold $W^{uu}(x)$ tangent to $E^{uu}(x)$.
Note that $W^{uu}(x)\cap p(\cT)=\{x\}$ (since the dynamics is topologically equivalent to
an irrational flow). Then~\cite{BC-whitney} implies that $p(\cT)$ is contained in a two-dimensional submanifold $\Sigma$
transverse to $E^{uu}$ and locally invariant by $\varphi_1$: there exists a neighborhood $U$ of $p(\cT)$ in $\Sigma$ such that
$\varphi_1(U)\subset \Sigma$. Since $p(\cT)$ is homeomorphic to $\TT^2$, it is open and closed in $\Sigma$, hence coincides with $\Sigma$.
This shows that $p(\cT)$ is $C^1$-diffeomorphic to $\TT^2$, normally expanded
and carries a dynamics topologically equivalent to an irrational flow.
\end{proof}

Theorem~\ref{Thm:local-fibered-flow-model-dominated} is now proved.
\end{proof}

\subsection{Proof of Theorem~\ref{Thm:measure-stable}}
We state and prove a more precise version of Theorem~\ref{Thm:measure-stable}.
We first introduce a notion of stable manifold adapted to a set with a normal dominated structure.

\begin{Definition}[Normal stable manifold]\label{Def:normal-parametrized}
Assume that  $\Lambda$ is a $(\varphi_t)_{t\in \RR}$-invariant set with a singular dominated splitting $\cN_\Lambda=\cN^{cs}\oplus \cN^{cu}$. Given $x\in \Lambda\setminus \sing(X)$, a manifold $V^s(x)\subset \cN(x)$ is said to be a \emph{normal stable manifold tangent to $\cN^{cs}(x)$}, if:
\begin{itemize}
\item it contains $x$ and is tangent to $\cN^{cs}(x)$ at $x$;

\item $P_t(V^s)\subset \cN_{\varphi_t(x)}$ for any $t>0$ and ${\rm Diam}(P_t(V^s))\to 0$ as $t\to\infty$.

\end{itemize}
\end{Definition}
 
 Definition~\ref{Def:normal-parametrized} gives a generalization of the stable manifold of the return map to a cross section at a periodic point.
 In order to have uniformity with respect to the base point $x$, we introduce the following definition
 (compare with Definition~\ref{d.plaque}).
 
\begin{Definition}[Normal plaque family]
Let $\Lambda$ be an invariant compact set with a singular dominated splitting $\cN_{\Lambda\setminus{\rm Sing}(X)}=\cE\oplus \cF$.
A \emph{$C^k$ normal plaque family} tangent to $\cE$ over an invariant set $Y\subset \Lambda\setminus \sing(X)$
is a continuous family of $C^k$-diffeomorphisms onto their image $\Theta_x\colon \cE(x)\to \cN_x$, $x\in Y$,
such that $\Theta_x(0)=x$ and the image is tangent to $\cE(x)$. We denote by
$\cW^\cE(x), \cW^\cE_\delta(x)$ the images $\Theta_x(\cE(x)), \Theta_x(B(0,\delta))$.

It is locally invariant if there exists $\alpha>0$ such that for any $x\in Y$ we have
$$P_1^*\bigg(\cW^\cE_\alpha(x)\bigg) \subset \cW^\cE(\varphi_t(x)).$$
\end{Definition}

\begin{Definition}[Family of normal stable manifolds]
Let $\Lambda$ be an invariant compact set with a singular dominated splitting $\cN_{\Lambda\setminus{\rm Sing}(X)}=\cE\oplus \cF$.
A \emph{continuous family of normal stable manifolds} (tangent to $\cE$) over an invariant set $Y\subset \Lambda\setminus \sing(X)$
is a locally invariant $C^1$-plaque family tangent to $\cE$ such that for any
$\rho>0$, there is $\alpha>0$ with the following properties:
$$P_t^*\bigg(\cW^\cE_\alpha(x)\bigg) \subset \cW^\cE_\rho(\varphi_t(x))
\text{ and } 
\diam\bigg(P_t^*(\cW^\cE_\alpha(x))\bigg)\underset{t\to +\infty} \longrightarrow 0.$$
In particular each set $\cW^\cE_\alpha(x)$ is a normal stable manifold.
\end{Definition}

\begin{Remark}
Note that after rescaling, for any $x\in Y$
the diameter of the iterates of the plaque $\cW^\cE_{\alpha\|X(x)\|} (x)$ under the Poincar\'e flow $(P_t)_{t\in \RR}$
decreases to $0$ as $t\to +\infty$.
\end{Remark}

\begin{Theorem A'}\label{Thm:measure-stable-bis}
Let $M$ be a compact Riemannian manifold without boundary, $X$ be a $C^2$ vector field over $M$ and $\mu$ be a regular ergodic measure of the flow generated by $X$. Let us assume that on the support of $\mu$:
\begin{itemize}
\item[--] there exists an index-$1$ singular dominated splitting  $\cN_{{\rm supp}(\mu)\setminus{\rm Sing}(X)}=\cN^{cs}\oplus\cN^{cu}$,
\item[--] any (regular) periodic orbit (if it exists) has at least one negative Lyapunov exponent,
\item[--] the dynamics is not topologically equivalent to an irrational flow on a $2$-torus.
\end{itemize}

Then there exists a continuous family of normal stable manifolds $\cW^{cs}_\alpha(x)$ over a set with full $\mu$-measure.
\end{Theorem A'}
In order to build the stable disc $V^s(x)$ in Theorem~\ref{Thm:measure-stable},
we first project the rescaled normal stable manifold to $M$ by the exponential map at $x$
and then flow it by $(\varphi_t)_{t\in\RR}$ during a small interval of time $[-\epsilon,\epsilon]$, that is:
$$V^s(x)=\varphi_{[-\epsilon,\epsilon]}(\exp_x(\cW^{cs}_{\alpha\|X(x)\|}(x))).$$
 
\begin{proof}
Let us set $\Lambda={\rm supp}(\mu)$.
Note that if $\mu$ is supported by a periodic orbit, the conclusion holds trivially since the periodic orbit has a negative Lyapunov exponent by assumption. We will thus assume that $\mu$ is not supported on a periodic orbit.

By Theorems~\ref{Thm:local-fibered-flow-model} and~\ref{Thm:local-fibered-flow-model-dominated}, there exist a topological flow $(\widehat\varphi_t)_{t\in\RR}$, a compact metric space $\widehat\Lambda$, a local $C^2$-fibered flow $({\widehat P}^*_t)_{t\in \RR}$ over a Riemannian vector bundle $\widehat{\cN}\to\widehat\Lambda$, some conjugacies $i$, $I$; since $\Lambda$ admits an index-$1$ singular dominated splitting  
$\cN^{cs}\oplus\cN^{cu}$, the set $\widehat\Lambda$ admits a dominated splitting $\widehat\cN^{cs}\oplus\widehat\cN^{cu}$, and there is a small neighborhood $\widehat V$ of $\widehat\Lambda\setminus i(\Lambda\setminus{\rm Sing}(X))$ such that the bundle $\widehat\cN^{cs}$ is uniformly contracted on $\widehat V$ by $({\widehat P}^*_t)_{t\in \RR}$.

By the assumptions and the last part of Theorem~\ref{Thm:local-fibered-flow-model-dominated}, $\widehat\cN^{cs}$ is contracted over any (regular) periodic orbit in $\widehat\Lambda$ and there is no normally hyperbolic irrational torus in $\widehat\Lambda$.

By Theorem~\ref{Thm:local-fibered-flow-model}, there is $C>0$ and an identification $\widehat\pi$ compatible with $({\widehat P}^*_{C.t})$ on an open set $\widehat U$ satisfying $\widehat\Lambda\setminus{\widehat U}\subset {\widehat V}$. To simplify the presentation we will assume $C=1$ in the following.

Now we consider the lift of the measure $\widehat\mu$ of $\mu$ on $\widehat \Lambda$. Since $\mu$ is a regular ergodic measure, one has that $\widehat\mu(i(\Lambda\setminus{\rm Sing}(X)))=1$.

The setting of Section~\ref{s.topological-hyperbolicity} is satisfied and we can consider a $C^2$-plaque family $\widehat \cW^{cs}$ tangent to $\widehat \cN^{cs}$.
By Corollary~\ref{Cor:ergodic-measure}, for any $\rho>0$ here is $\alpha>0$ such that for $\widehat\mu$-almost every point $\widehat x$, the image ${\widehat P}_t({\widehat \cW}^{cs}_{\alpha}({\widehat x}))$
is well-defined for any $t\ge 0$, has diameter less than $\rho$ and satisfies
$$\lim_{t\to+\infty} |{\widehat P}^*_t({\widehat \cW}^{cs}_{\alpha}({\widehat x}))|=0.$$
Thus, at $\mu$-almost every point $x$, the rescaled sectional Poincar\'e flow contracts
the manifold ${\cV}^{cs}_{\varepsilon_0}({x})=I^{-1}({\cW}^{cs}_{\varepsilon_0}({\widehat x}))$:
$$\lim_{t\to+\infty} |{P}^*_t({\cV}^{cs}_{\varepsilon_0}({x}))|=0.$$
Then we project the family of plaques $\widehat \cW^{cs}$ by the fibered isometry $I$
and get a continuous family $\cW^{cs}$ of normal stable manifolds, as wanted.
\end{proof}

\subsection{Proof of Theorem~\ref{Thm:minially-non-contracted-stable}}
We state and prove a more precise version of Theorem~\ref{Thm:minially-non-contracted-stable}:

\begin{Theorem C'}\label{Thm:minially-non-contracted-stable-recaled-poincare}
Let $M$ be a Riemannian compact manifold without boundary, $X$ be a $C^2$ vector field over $M$ and $\Lambda$ be a compact set invariant by the flow generated by $X$ which is not reduced to a periodic orbit.
Let us assume that :
\begin{itemize}
\item[--] there exists an index-$1$ singular dominated splitting $\cN_{\Lambda\setminus{\rm Sing}(X)}=\cN^{cs}\oplus\cN^{cu}$,
\item[--] any regular ergodic measure $\mu$ (if there exists) whose support is a proper subset of $\Lambda$ has at least one negative Lyapunov exponent,
\item[--] the dynamics is not topologically equivalent to an irrational flow on a $2$-torus.
\end{itemize}
Then there exists a continuous family of normal stable manifolds over $\Lambda\setminus \sing(X)$.
\end{Theorem C'}
\begin{proof}
As in the proof of Theorem~\ref{Thm:measure-stable-bis},
we introduce $(\widehat\varphi_t)_{t\in\RR}$, $\widehat\Lambda$, $({\widehat P}^*_t)_{t\in \RR}$ over $\widehat{\cN}\to\widehat\Lambda$, $i$, $I$, a dominated splitting $\widehat \cN=\widehat\cN^{cs}\oplus\widehat\cN^{cu}$, a small neighborhood $\widehat V$ of $\widehat\Lambda\setminus i(\Lambda\setminus{\rm Sing}(X))$ such that the bundle $\widehat\cN^{cs}$ is uniformly contracted on $\widehat V$ by ${\widehat P}^*_t$. Moreover $\widehat\cN^{cs}$ is contracted over any regular periodic orbit in $\widehat\Lambda$ and there is no normally hyperbolic irrational torus in $\widehat\Lambda$.

We want to apply Theorem~\ref{Thm:topologicalcontracting}. We have to check that on any proper compact invariant set $\widehat\Gamma\subset \Lambda$, $\widehat\cN^{cs}$ is uniformly contracted by the flow $(\psi^*_t)_{t\in\RR}$. It suffices to prove that for any ergodic measure $\widehat\mu$ supported on $\widehat\Gamma$, the Lyapunov exponent of $\widehat\mu$ with respect to $(\psi_t^*)_{t\in\RR}$ along $\widehat\cN^{cs}$ is negative.
By ergodicity, either $\widehat\mu(i(\Lambda\setminus{\rm Sing}(X)))=1$ or $\widehat\mu(\widehat\Lambda\setminus i(\Lambda\setminus{\rm Sing}(X)))=1$. 

If $\widehat\mu(\widehat\Lambda\setminus i(\Lambda\setminus{\rm Sing}(X)))=1$,
then the support of $\widehat \mu$ is contained in $\widehat V$; by the definition of $\widehat V$, one knows that $\widehat\cN^{cs}$ is uniformly contracted. Consequently, $\widehat\mu$ has negative Lyapunov exponents along $\widehat\cN^{cs}$.

If $\widehat\mu(i(\Lambda\setminus{\rm Sing}(X)))=1$, then the support of $\widehat \mu$ cannot be dense in the open set
$i(\Lambda\setminus{\rm Sing}(X))$ since this set is dense in $\widehat \Lambda$ and the support of $\widehat \mu$
is proper in $\widehat \Lambda$. Let $\mu$ be the projection of $\widehat \mu$ to $\Lambda$.
Since $i$ is a homeomorphism between the open sets
$\Lambda\setminus \sing(X)$ and $i(\Lambda\setminus{\rm Sing}(X))$, the support of $\mu$ is proper in $\Lambda$.
Note also that $\mu$ is regular. By assumption its Lyapunov exponent along $\cN^{cs}$ is negative.
Then by the isometric property of $I$, one knows that $\widehat\mu$ has a negative Lyapunov exponent along $\widehat\cN^{cs}$.

Now we can apply Theorem~\ref{Thm:topologicalcontracting}. there is $\varepsilon_0>0$ such that for every point ${\widehat x}\in\widehat\Lambda$, the image ${\widehat P}_t({\cW}^{cs}_{\varepsilon_0}({\widehat x}))$
is well-defined for any $t\ge 0$, and
$$\lim_{t\to+\infty} \sup_{{\widehat x}\in\widehat\Lambda}|{\widehat P}^*_t({\cW}^{cs}_{\varepsilon_0}({\widehat x}))|=0.$$
Then we project the family of plaques $\widehat \cW^{cs}(\widehat x)$ for $\widehat x\in i(\Lambda\setminus \sing(X))$ by the fibered isometry $I$
and get a continuous family $\cW^{cs}$ of normal stable manifolds, as wanted.
\end{proof}

\vskip 5pt

\begin{tabular}{l l l}
\emph{\normalsize Sylvain Crovisier}
& \quad &
\emph{\normalsize Dawei Yang}
\medskip\\

\small Laboratoire de Math\'ematiques d'Orsay
&& \small School of Mathematical Sciences\\
\small CNRS - Universit\'e Paris-Saclay
&& \small Soochow University\\
\small Orsay 91405, France
&& \small Suzhou, 215006, P.R. China\\
\small \texttt{Sylvain.Crovisier@universite-paris-saclay.fr}
&& \texttt{yangdw1981@gmail.com}\\
&& \texttt{yangdw@suda.edu.cn}\
\end{tabular}

\end{document}